\documentclass[a4paper,11pt]{article}

\usepackage[ansinew]{inputenc}
\usepackage[T1]{fontenc}
\usepackage{geometry}
\usepackage[none]{hyphenat}
\usepackage{fancyhdr}
\usepackage{babel}
\usepackage{amsmath}
\newcommand\norm[1]{\lVert#1\rVert}
\usepackage{amsthm}
\usepackage{amssymb}
\usepackage{stackrel}
\usepackage{graphicx}
\usepackage{mathtools}
\usepackage{hyperref}
\usepackage{verbatim}
\usepackage{wrapfig}

\usepackage[colorinlistoftodos, textwidth=\marginparwidth]{todonotes} 

\usepackage[inline]{enumitem}
\usepackage{float}

\usepackage{breakurl}
\usepackage{listings}
\usepackage{color}
\definecolor{qwe}{rgb}{0,0.7,0}
\definecolor{gr}{rgb}{0.6,0.6,0.6}
\usepackage{pdflscape} 

\makeatletter

\usepackage{enumitem}		
      
  \theoremstyle{plain}
  \newtheorem*{prop*}{\protect\propositionname}
  \theoremstyle{plain}
  \newtheorem*{lem*}{\protect\lemmaname}

\newtheorem{theorem}{Theorem}%[section] 
\newtheorem{definition}{Definition}%[section] %
\newtheorem{corollary}{Corollary}[theorem]

\newtheorem{question}{Open Question}%[section]
\newtheorem{conjecture}{Conjecture}%[section]
\newtheorem{proposition}{Proposition}%[section]

\usepackage{tabularx}
\newcolumntype{Y}{>{\centering\arraybackslash}X} 
\usepackage{booktabs}

\usepackage{listings}

\makeatother
  \providecommand{\lemmaname}{Lemma}
  \providecommand{\propositionname}{Satz}

\begin{document}
\emergencystretch 4em 
\fancyhead{} 
\fancyhead[LE,RO]{\thepage}
\fancyhead[LO,RE]{\rightmark}

\begin{titlepage}
    \begin{center}
        \vspace*{1cm}
            
        \LARGE
        \textbf{The $n$-Queens Problem in Higher Dimensions}
            
        \vspace{1.5cm}
        \normalsize
        Masterarbeit unter der Betreuung von\\
        \Large
Prof. Dr. Thorsten Koch \\
Institut f\"ur Mathematik \\
Technische Universit\"at Berlin \\
        \vspace{0.8cm}
        
        \normalsize
        vorgelegt von \\
        \Large
        Tim Kunt \\
            
        \vspace{1.5cm}

        \vspace{0.5cm}

        Berlin, 20.12.2023
            
    \end{center}
\end{titlepage}

\thispagestyle{fancy}

\newpage
\tableofcontents

\newpage
\section*{Acknowledgments}
Two years ago, I unsuspectingly visited the lecture 'Industrial Data Science' at Technical University Berlin, hoping to learn more about the field I had pivoted towards. The lecture contained a few applications of mathematical optimization, and to my surprise, these parts were very much relevant for the final oral exam (I had most of my attention focused on the probability theory and modeling parts, where I felt comfortable). Thorsten Koch and Charlie Vanaret, who had been responsible for my weekly mathematical entertainment, allowed me to repeat the exam after I had barely passed it, unable to write down the IP formulation for a given problem.\\
Regardless, I enjoyed Thorsten's teaching style, the tangents and stories, and the terrors of applied mathematics, so that I would join the following two seminars on integer programming, still mostly clueless about the topic. Trying to catch up, the first IP model I came up with was regarding three-dimensional chess after having played \textit{5D Chess With Multiverse Time Travel} \cite{5dchess} and through the visual intuition (see Fig. \ref{fig:sudoku}) of how to model a sudoku using integer programming.
\begin{figure}[H]
    \centering
    \includegraphics[width=1\linewidth]{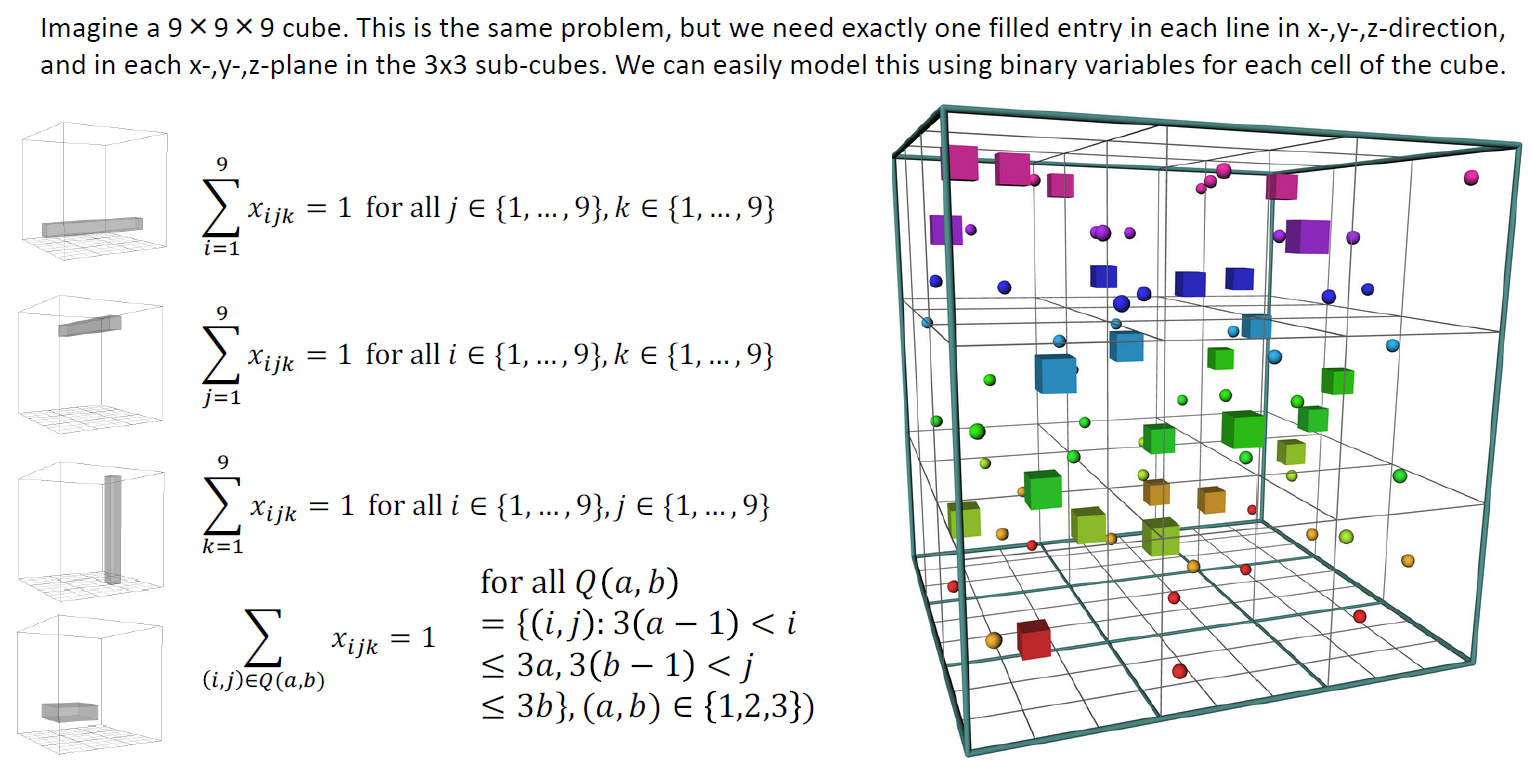}
    \caption{ILP formulation of a sudoku \cite{Koch2020}}
    \label{fig:sudoku}
\end{figure}
What started as a coding and learning project for the seminar turned out to be an easy-to-understand yet complex-to-solve problem. Having picked up the $n$-queens problem again for this thesis, I have learned to appreciate the torture of seemingly unsolvable instances and the excitement of incremental progress.

I thank Michael Simkin for his comments on my questions and ideas regarding the $n$-queens problem in 3D and its enumeration. As the known results on the $n$-queens problem in higher dimensions are quite scattered or unpublished, the \textit{Al Zimmermann's Attacking Queens Contest} \cite{ZimmermannContest} turned out to be a valuable source and starting point. Thanks to Michael Ecker, who was able to find and share the original announcement of the contest. I thank Stephen Montgomery-Smith, Joe Zbiciak and Karl Grill, all of whom were involved in the contest, for their notes and help regarding the contest's history.
\newpage 
I am very grateful to the entire A2IM department at ZIB, in particular Milena Petkovic, Janina Zittel and Thorsten Koch for their trust and support throughout our work and during this thesis. I greatly appreciate the creative freedom we are given and hope to find and work on many more diverse research topics and projects.

Thanks to Mark Ruben Turner, whom I could confront with questions at any time. Your advice helped me in both the technical implementation and gaining confidence and insight. When I learned that Timo Berthold would be supervising this thesis together with Thorsten Koch, I was reassured and scared at the same time. With your expertise, Timo, you will surely spot all the loose ends and opportunities for improvement in this thesis. Thank you for listening and for your help at all times. Lastly, I thank my professor and supervisor, Thorsten Koch. You sparked my interest in the topic in the first place and are the reason I ended up at ZIB.\\\
\\
\bigskip \bigskip
\\
\textit{"I think that I may personally be done with the n-queens problem for a while, not because there isn't anything more to do with it but just because of I've been dreaming about chess and I'm ready to move on with my life."}
\begin{flushright}
Michael Simkin \cite{physorg}
\end{flushright}

\newpage
\section{Introduction}
\subsection{Motivation}
The $n$-queens problem asks whether it is possible to place $n$ mutually non-attacking queens on an $n \times n$ chessboard. This thesis considers the $n$-queens problem in higher dimensions with the goal of solving its instances or obtaining bounds towards their solutions.

The classical $n$-queens problem, while over 150 years old, is the source of various optimization and decision problems. While solving the classical $n$-queens problem (i.e., finding a certificate to the decision problem) can be achieved through various constructions in linear time, the generalization of the $n$-queens problem to higher dimensions turns out to be highly non-trivial.

$n$-queens and its many variations can be explained in just a few sentences, and its mathematical formulation is quite compact, yet it poses interesting, hard and unsolved questions. For this reason, the topic is well-suited for benchmarks that compare different methodologies and algorithms. Throughout the long history of the problem, there have been steady and incremental advancements. Both the innovation in technology and algorithms, as well as new theoretical results, enable this progress.\\
At the same time, the $n$-queens problem allows emerging technologies to prove themselves, comparing them to classical methods. Several publications on solving the problem with neural networks were published in the past few decades, and the first ones are tackling the problem using quantum computing now. Thus, we are also interested in stating and comparing our integer programming methods on instances of the $n$-queens problem in higher dimensions.

There are (at least) three angles from which the $n$-queens problem in 3D and higher dimensions can be motivated:  
\begin{itemize}
    \item[1)] \textbf{Chess problems in higher dimension}\\
    The first and most intuitive one is through the imagination of three- (or higher-) dimensional chess. This opens up the question of how different chess pieces move in higher dimensions and provides a starting point for a great variety of generalizations of mathematical- and chess problems into higher dimensions.
    \item[2)] \textbf{Higher dimensional permutations with additional constraints}\\
    Second, the $n$-queens problem in 3D can be described as a Latin square with additional constraints. This will be described later on in detail, but it has a similar visual intuition as it is demonstrated by Fig.\ref{fig:sudoku} for sudoku (which can also be described as a Latin square with additional constraints).
    \item[3)] \textbf{Maximal independent sets on grid graphs}\\
    Lastly, we may interpret the chessboard for a given chess piece as a graph. Two nodes, i.e., squares are connected through an edge if the chess piece can move from one to the other. This is called the queen's graph in the case of the queen. The $n$-queens problem is equivalent to the independent set problem on the queen's graph. This is expanded on in \ref{subsubsection:queensgraph}.
\end{itemize}
While (1) is mainly helpful for an intuition of the problem and is eponymous to its name, (2) and (3) may already hint at suitable methods for solving. 
\begin{figure}[H]
    \centering
    \includegraphics[width=0.8\linewidth]{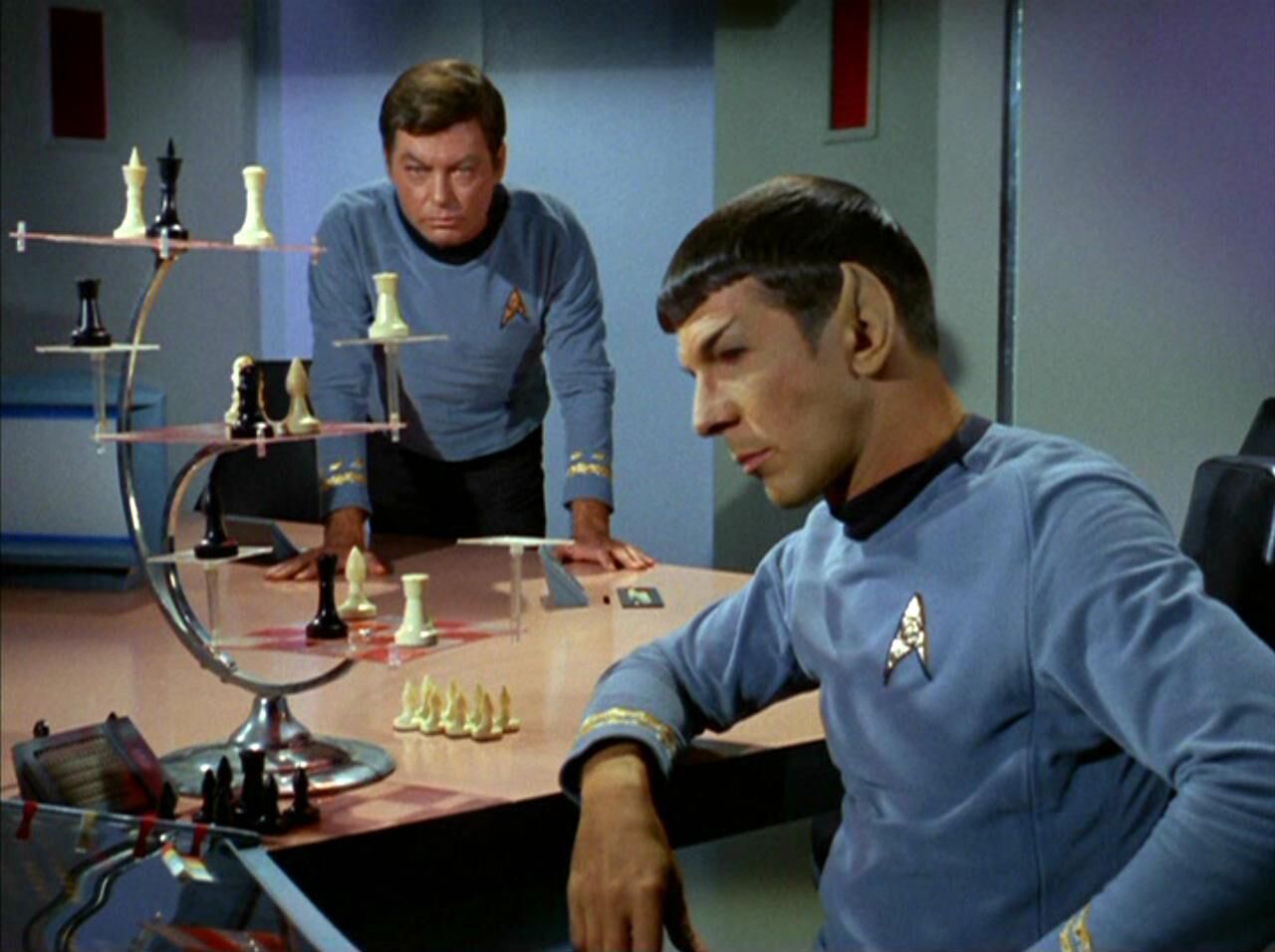}
    \caption{Three dimensional chess as shown in Star Trek \cite{trek1966original}}
    \label{fig:star_trek}
\end{figure}

We differentiate between varying approaches to the $n$-queens problem, some of which may benefit from one another. The most straightforward approach is brute force search methods, possibly abusing the symmetries of the problem. Multiple constructions exist that provide solutions or lower bounds to the problem in both two and higher dimensions. Additionally, there are several interesting insights into the problem through combinatorics. Our approach to the problem is using integer programming. This has several advantages:
\begin{itemize}
    \item[a)] we can prove the maximality of solutions
    \item[b)] we may incorporate and benefit from insights of other solution methods
    \item[c)] we are able to solve instances that are potentially out of reach for other methods
\end{itemize}
Our primary focus is the IP formulation and possible strengthenings, as well as using heuristics for warmstarts. This is complemented by lower bounds for the instances that we cannot solve as of today.

\newpage 
\subsection{Contribution}
We provide a comprehensive overview of existing literature and methods for solving the $n$-queens problem in higher dimensions. Our starting point is \textit{"A survey of known results and research areas for n-queens"} \cite{bell2008results} from 2009, which we extend with recent results and our methods regarding the variation of the problem in higher dimensions and the solving of instances. While this variation of the problem is not unknown, literature on the topic is partially scattered or unpublished, which is best illustrated by the fact that several results have been independently (re-)discovered by different authors over the years. Thus, we consider collecting and connecting the existing results and solving techniques necessary.  
We describe and briefly compare the varying approaches to the problem and draw the connection between several related problems.

Second, we provide an integer programming formulation for the $n$-queens problem in higher dimensions. We discuss and compare several possible strengthenings of the formulation.
We note that the recent preprint \textit{Complexity of Chess Domination Problems} \cite{langlois2022complexity} by Alexis Langlois-R{\'e}millard,  Christoph M{\"u}{\ss}ig  and {\'E}rika R{\'o}ldan was published during the work on this thesis. Among their results regarding complexity of domination problems on polyminoes, they propose an IP formulation for the $n$-queens problem in higher dimensions. In order to incorporate their results, we compare our methods to their formulation and achieve a speedup between $15.5\times -71.2 \times $ less computational time over all instances.\footnote{Measuring this speedup, we compare average computational time, gurobi work, and nodes for their method and our best-performing method and limit ourselves to instances that initially took more than 10s to solve in order to minimize noise. See section \ref{sec_comp_results} in detail.}.

Third, we are able to verify the maximality of existing certificates to the problem for different $n,d$.

Fourth, we present a heuristic method to derive lower bounds for instances still out of reach today.

Lastly, we give several preliminary results for enumeration and the density of solutions to the $n$-queens problem in higher dimensions and point towards insights regarding related problems.

\newpage
\subsection{Notation and Definitions}
Notation follows \cite{gent2017complexity} and \cite{Bell2009} in large parts. There are several cases of varying names or notation for the same problem, property, or result in different sources. In those cases, we will list and define them once and proceed with a choice that allows for a coherent and descriptive notation.

\subsubsection{Board}\label{subsubsection_board}
We consider \textbf{boards} of size $n$ in dimension $d$, where integer $n$ is the number of squares in each dimension and integer $d$ is the number of dimensions the board extends to. We will describe a board as $(n,d)$-board, if not otherwise apparent from the context. As an example, following this notation, a conventional chess board would be called an $(8,2)$-board.

A generalization of the $n$-queens and related problems to rectangular boards or polyominoes is possible; see \cite{langlois2022complexity}. Further variants of the problem identify opposite faces of the board, creating different topologies such as $n$-queens on a Moebius strip or a torus. For an overview see \cite{bell2008results} and \cite{Bell2009}. 
In particular, the $n$-queens problem on the modular $n^d$ board \cite{Nudelman1995} proves useful for the $n$-queens problem in higher dimensions and will be expanded on in the following section. 

We will refer to individual \textbf{squares} on such a board of dimension $d$ using a tuple $(a_1, a_2, ..., a_d)$, where each element marks the position in the respective dimension.

The \text{$i$-th} \textbf{layer} \text{in the $j$-th dimension} of a board is given by the set of squares, for which the $a_j = i$. For $d=2$, the two possible layers are simple rows and columns. For $d=3$, the layers are two-dimensional boards themselves. The concept is introduced, as it naturally leads to a possible algorithm to construct solutions for the $n$-queens problem in higher dimensions. Layers of the $(n,d)$-board correspond to $(n,d-1)$-boards.\\
\begin{figure}[H]
\setkeys{Gin}{width=0.3\linewidth}
\includegraphics{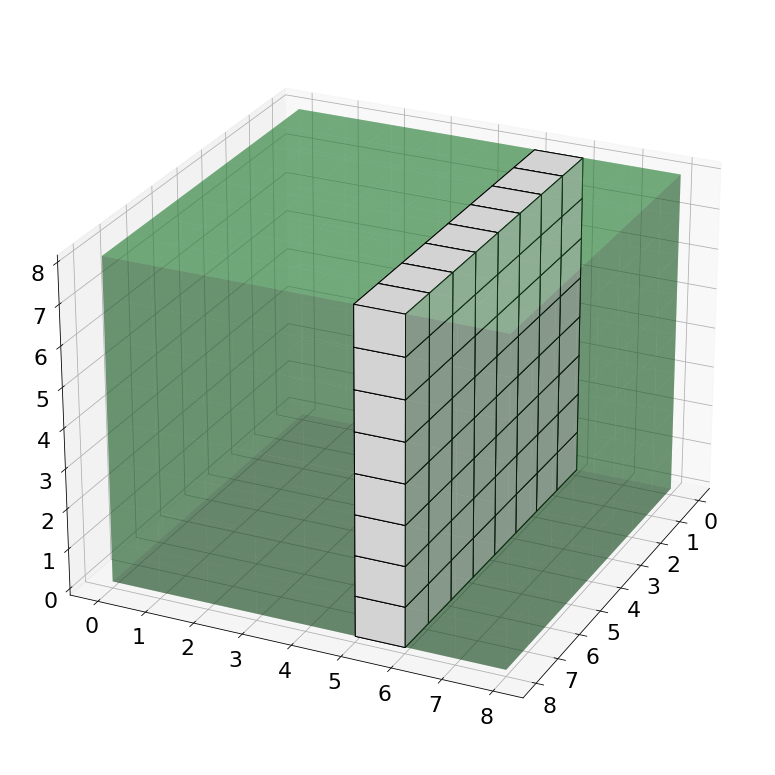}
\hfill
\includegraphics{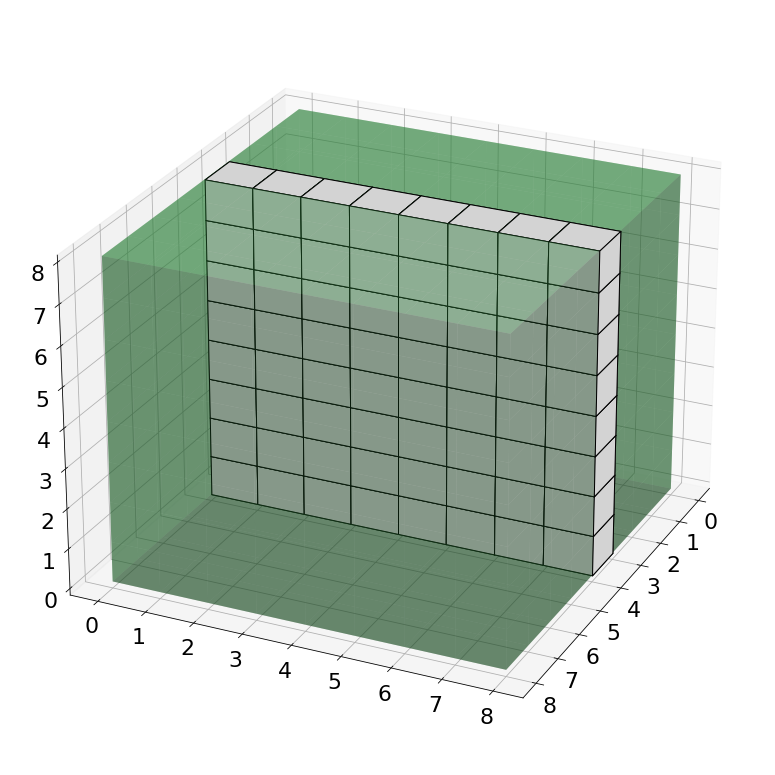}
\hfill
\includegraphics{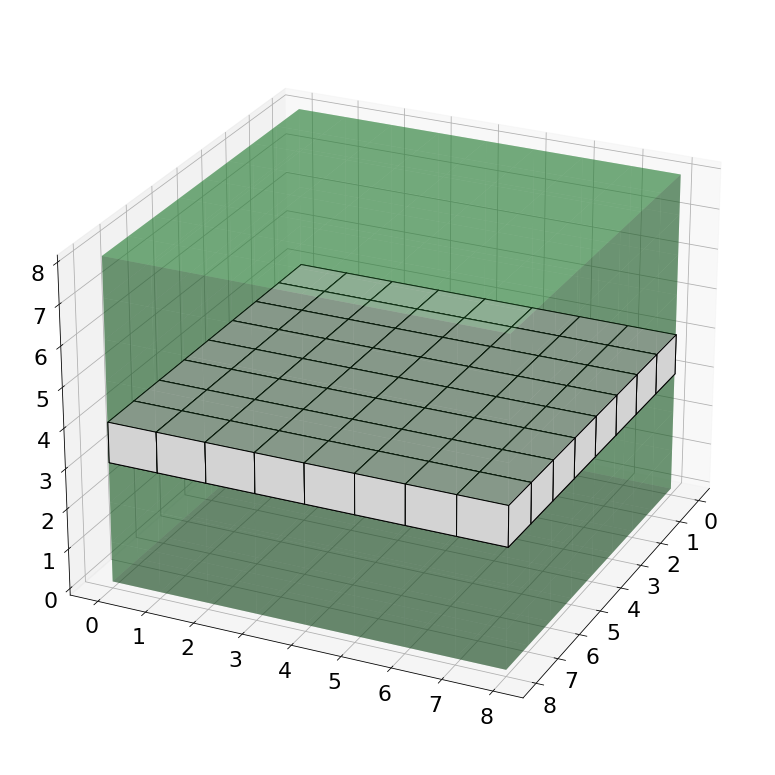}
\caption{Example of one layer in each dimension of the $(8,3)$-board}
\label{tab:vis_layer}
\end{figure}
We point out that developing additional variants of the $n$-queens problem is possible using boards that do not follow a rectangular grid, such as hexagonal chess. A generalization to other grid graphs (as we may identify the chessboard as a square grid graph) is motivated naturally by the formulation of the $n$-queens problem as a maximal independent set problem on grid graphs. The queen graph is later touched on in \ref{subsubsection:queensgraph}. These variants may also be formulated so that they correspond to a version of the Chv\'atal's art gallery problem (c.f. \cite{alpert2021art}), provided a definition of a queens vision (which equates to their legal moves) on the chosen grid.

\begin{figure}[H]
    \centering
    \includegraphics[width=0.4\linewidth]{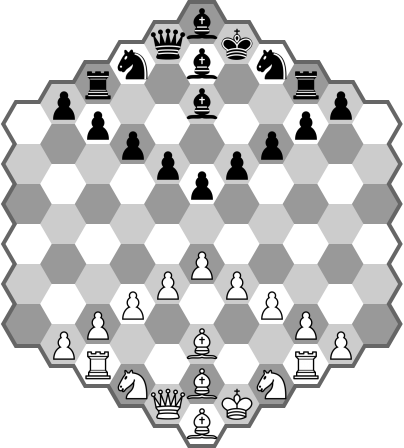}
    \caption{Hexagonal chess (Gli\'{n}ski's variant)}
    \label{fig:vis_hexchess}
\end{figure}

\subsubsection{Pieces}
We define chess pieces through their moves, i.e., lines of attack on a board. To avoid confusion in the context of the minimum dominating set problem, any piece is always attacking its own square. Definitions build upon \cite{gent2017complexity}.

Similar to boards, it is possible to come up with numerous chess pieces and examine the resulting maximal independent set problem or other related problems. The reason why queens are particularly interesting is that they present the first step in complexity above rooks. For rooks, the maximal independent sets directly correspond to permutation matrices and Latin squares. Therefore, the cardinality of those maximal independent sets on the $(n,d)$-board is trivially $n^{d-1}$. Enumeration of solutions may still be non-trivial; in fact, it has only been solved recently for such permutation matrices \cite{keevash2018existence}, which we will expand on in section \ref{section_enumeration_general_d}. For queens, however, we cannot always place $n^{d-1}$ in a non-attacking configuration on the $(n,d)$-board. That is why we are concerned with finding such maximal independent sets of queens.

\begin{definition}[Rook] A rook \textbf{r} on a $(n,d)$-board is a tuple of integers $(a_1, a_2, ... , a_d)$ with $1 \leq a_i \leq n$ for $i \in \{1,2,...,d\}$.
\begin{itemize}
    \item The \textbf{size of \textbf{$r$}}, written $\norm{r}$ is defined by $\norm{r}:=\text{max}(a_1, a_2, ..., a_d)$. \\
    We say that $r$ fits a board of size $n$ iff $\norm{r}\leq n$. We extend this to sets of rooks, so that $\norm{R}:=\text{max}_{r\in R}(\norm{r})$. Any set of rooks is pairwise distinct. We write $|R|$ for the number of elements in $R$.
    \item For any pair of rooks $r_1 = (a_1, a_2, ... , a_d)$ and $r_2 = (b_1, b_2, ... , b_d)$ with $r_1 \neq r_2$, we say that \textbf{$r_1$ attacks $r_2$} if there exists some $m\in \mathbb{Z}$ and $\epsilon = (0,..., \epsilon_i, ... , 0)$ with one non-zero entry $\epsilon_i \in \{-1,0,1\}$ \textit{such that}
    \begin{align}
        a_i = \epsilon_i \cdot m + b_i \; \; \forall i \in \{1,2,...,d\}
    \end{align}
\end{itemize}
\end{definition}
In other words, a pair of rooks $r_1 = (a_1, a_2, ... , a_d)$ and $r_2 = (b_1, b_2, ... , b_d)$ attack each other if $a_i = b_i$ for all but one $i \in 1,...,d$. The above might not be the most compact definition of a rook's movement in higher dimensions, but it introduces us to the concept of hyperplanes of attacks described by the vectors $\epsilon$ and naturally leads to the following definition for queens.

\begin{definition}[Queen] A queen \textbf{q} on a $(n,d)$-board is a tuple of integers $(a_1, a_2, ... , a_d)$ with $1 \leq a_i \leq n$ for $i \in \{1,2,...,d\}$.
\begin{itemize}
    \item The \textbf{size of \textbf{$q$}}, written $\norm{q}$ is defined by $\norm{q}:=\text{max}(a_1, a_2, ..., a_d)$. \\
    We say that $q$ fits a board of size $n$ iff $\norm{q}\leq n$. We extend this to sets of queens, so that $\norm{Q}:=\text{max}_{q\in Q}(\norm{q})$. Any set of queens is pairwise distinct. We write $|Q|$ for the number of elements in $Q$.
    \item For any pair of queens $q_1 = (a_1, a_2, ... , a_d)$ and $q_2 = (b_1, b_2, ... , b_d)$ with $q_1 \neq q_2$, we say that \textbf{$q_1$ attacks $q_2$} if there exists some $m\in \mathbb{Z}$ and nonzero  $\epsilon = (\epsilon_1, \epsilon_2, ... , \epsilon_d)$ with $\epsilon_i \in \{-1,0,1\}$ such that
    \begin{align}
        a_i = \epsilon_i \cdot m + b_i \; \; \forall i \in \{1,2,...,d\}
    \end{align}
\end{itemize}
\end{definition}

The attacking relation is symmetric for rooks and queens. Note that the relation is not necessarily symmetric for all chess pieces and boards. This is the reason why the given definition differentiates between \textbf{$p_1$} attacking \textbf{$p_2$} and \textbf{$p_2$} attacking \textbf{$p_1$}. We will call a pair of queens for which neither \textbf{$q_1$} is attacking \textbf{$q_2$} nor \textbf{$q_2$} is attacking \textbf{$q_1$} \textbf{mutually non-attacking}.

The $\epsilon \cdot m$ describe the hyperplanes a queen can attack in for general dimensions. We will also refer to a piece $p$ attacking a set of squares if a piece placed on each of those squares would be attacked by $p$. In the context of chess domination problems, a square that is attacked by at least one piece is considered \textbf{dominated}.

\begin{figure}[H]
\setkeys{Gin}{width=0.45\linewidth}
\includegraphics{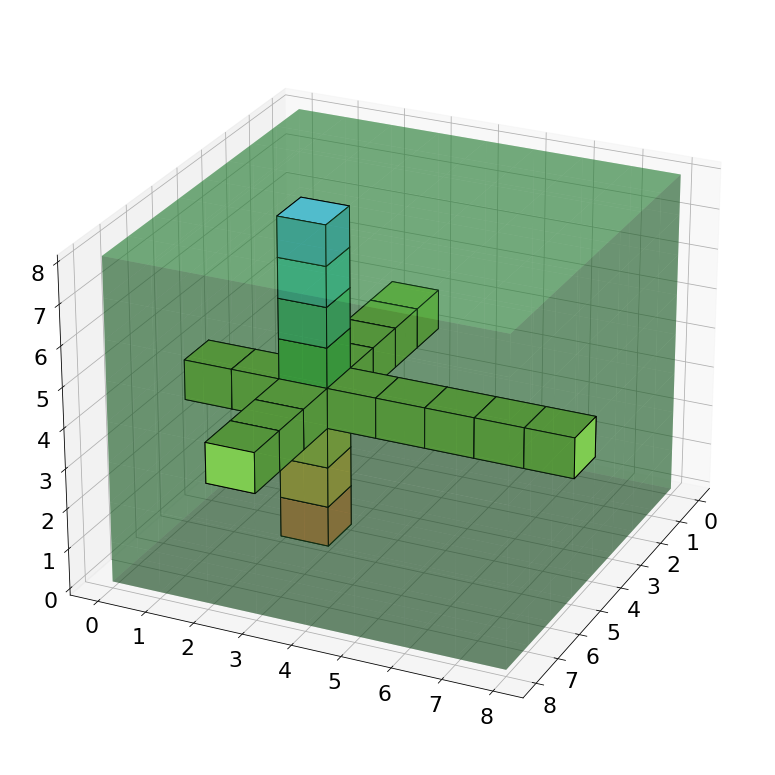}
\hfill
\includegraphics{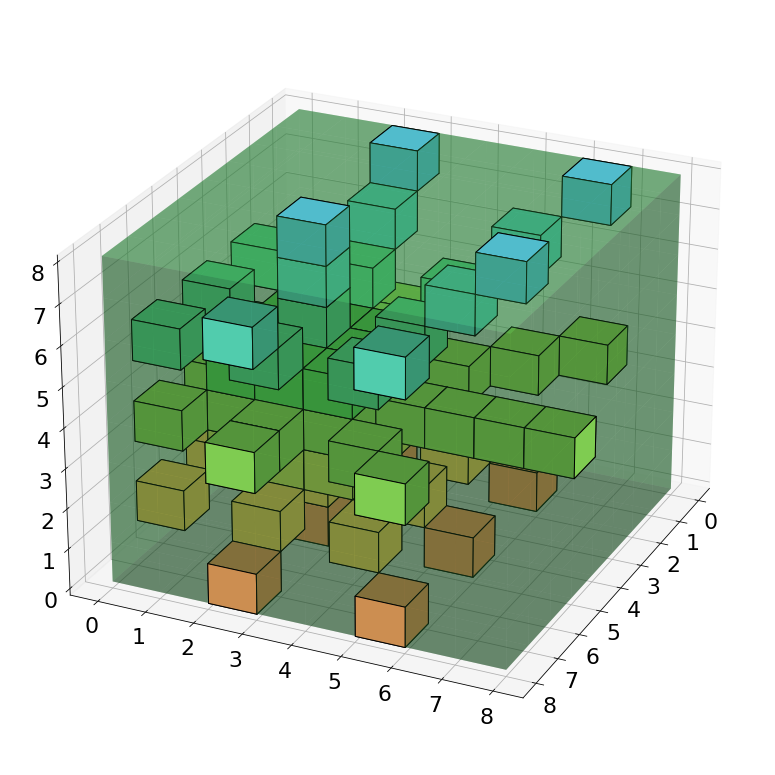}
\caption{Dominated squares by a rook (left) and a queen (right) placed on the square $(4,2,3)$ on the $(8,3)$-board}
\label{fig:vis_rook_attack}
\end{figure}

\newpage
\subsection{Problem Definitions}
In the following, we define the $(n,d)$-queens problem as a generalization of the $n$-queens problem and a selection of variants as decision problems. This allows us to formulate corresponding integer program formulations as feasibility problems and provides the necessary foundation to discuss complexity (c.f. \cite{gent2017complexity} \cite{langlois2022complexity}). 
Notation follows \cite{gent2017complexity} in large parts.\\

\textbf{Problem 1} ($(n,d)$-queens). \vspace{0.2cm} \\
\begin{tabular}{rll}
    \textsc{Problem:} & \textit{A $(n,d)$-board,} $n,d \in \mathbb{N}^{+}$\vspace{0.2cm} \\
    \textsc{Solution:} & \textit{A set $Q$ of queens such that: \vspace{-0.2cm}}\\
\end{tabular}
\begin{itemize}
    \item[i)] $|Q|=n^{d-1}$
    \item[ii)] $\norm{Q} \leq n $
    \item[iii)] \textit{For any distinct pair of queens} $q_1$, $q_2$ $\in Q$, $q_1$ \textit{does not attack} $q_2$
\end{itemize}
This is a generalization of the definition \cite{gent2017complexity} of the standard $n$-queens problem to higher dimensions. In particular, the two-dimensional case $d=2$ for which (i) $|Q|=n$ is commonly known and referred to as the standard $n$-queens problem.\\

\textbf{Problem 2} (modular $(n,d)$-queens). \vspace{0.2cm} \\
\begin{tabular}{rll}
    \textsc{Problem:} & \textit{A $(n,d)$-board,} $n,d \in \mathbb{N}^{+}$\vspace{0.2cm} \\
    \textsc{Solution:} & \textit{A set $Q$ of queens such that: \vspace{-0.2cm}}\\
\end{tabular}
\begin{itemize}
    \item[i)] $|Q|=n^{d-1}$
    \item[ii)] $\norm{Q} \leq n $
    \item[iii)] \textit{For any distinct pair of queens} 
    $q_1 = (a_1, a_2, ... , a_d)$\textit{,} $q_2 = (b_1, b_2, ... , b_d)$  $\in Q$
    \textit{and for any non-zero} $\epsilon = (\epsilon_1, \epsilon_2, ... , \epsilon_d)$
    \begin{align*}
        \sum^d_{i=1}(\epsilon_i \cdot a_i) \neq \bigg( \sum^d_{i=1}(\epsilon_i \cdot b_i) \bigg) \text{mod n}
    \end{align*}
\end{itemize}

The modular $n$-queens problem, first introduced by Polya, \cite{polya1921uber} 
is a variation of the $n$-queens problem where opposite sides of the board are identified, forming a torus. This generalizes to higher dimensions $d\geq3$ as the $n$-queens problem on the $d$-torus.

\newpage
\begin{figure}[H]
    \centering
    \includegraphics[width=0.5\linewidth]{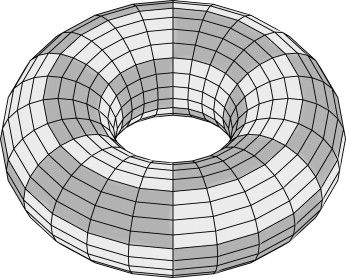}
    \caption{Queens on a donut \cite{bell2008results}}
    \label{fig:my_label}
\end{figure}

\begin{proposition}
    A solution to the modular $(n,d)$-queens problem is also a solution to the $(n,d)$-queens problem.
\end{proposition}
This follows from the fact that the set of attacked squares for any queen in the $n$-queens problem is contained in the set of attacked squares of a queen on the same square on the modular board. Therefore any mutually non-attacking configuration on the modular board is also a mutually non-attacking configuration on the standard board. Analog arguments can be made for certain sets of mutually non-attacking chess pieces, for which the move-set of one piece is contained in the move-set of another (as is the case with queen and superqueen or bishop, rook, and queen).

\begin{proposition} \label{prop_rooks_contained_in_queens}
    For a pair of pieces $p_1$ and $p_2$ for which the set of dominated squares of $p_2$ is contained in the set of dominated squares of $p_1$, a solution to the $(n,d)$-$p_1$ problem is also a solution to the $(n,d)$-$p_2$ problem.
\end{proposition}
\begin{figure}[H]
    \centering
    \includegraphics[width=0.4\linewidth]{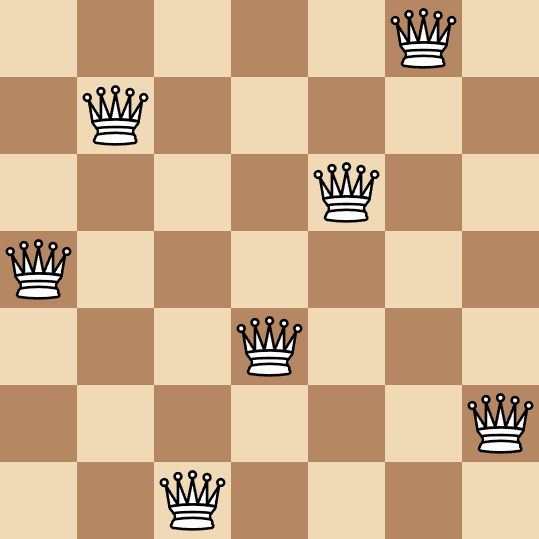}
    \caption{solution to the modular $(7,2)$-queens problem}
    \label{fig:my_label}
\end{figure} 

\newpage
While the standard $(n,2)$-queens problem has a solution for all $n \geq 3$ (see Theorem \ref{thm_pauls}), a comparable result is not known for the $(n,d)$-queens problem for $d \geq 3$. For $n, d$, for which no solution exists, we are interested in finding sets of mutually non-attacking queens with maximal cardinality. This motivates the following definition of the \textit{partial} $(n,d)$-queens problem.\\

\textbf{Problem 3} (partial $(n,d)$-queens). \vspace{0.2cm} \\
\begin{tabular}{rll}
    \textsc{Problem:} & \textit{A $(n,d)$-board,} $k\leq n^{d-1}$, $n,d,k \in \mathbb{N}^{+}$\vspace{0.2cm} \\
    \textsc{Solution:} & \textit{A set $Q$ of queens such that: \vspace{-0.2cm}}\\
\end{tabular}
\begin{itemize}
    \item[i)] $|Q|=k$
    \item[ii)] $\norm{Q} \leq n $
    \item[iii)] \textit{For any distinct pair of queens} $q_1$, $q_2$ $\in Q$, $q_1$ \textit{does not attack} $q_2$
\end{itemize}

The special case $k=n^{(d-1)}$ of the partial $(n,d)$-queens problem is the $(n,d)$-queens problem. For given $n, d$, the corresponding optimization problem asks to find the maximal $k$, for which a solution to the partial $(n,d)$-queens problem exists. We call such a solution a \textbf{maximal partial solution} \cite{bell2008results}. For a given board of size $n$ in dimension $d$, we write $Q_{max}(n,d)$ for such a solution and $|Q_{max}(n,d)|$ for its cardinality. We distinguish between the $(n,d)$-queens problem and the partial $(n,d)$-queens problem, as certain statements, methods, and enumeration techniques may differ or only apply for $k=n^{(d-1)}$.

The main focus of this thesis is to find and enumerate maximal partial solutions. We will further define a small selection of variants and related problems. These may serve to provide lower bounds to the partial $(n,d)$-queens problem. Others, such as the minimal independent dominating set problem, are well suited to be solved by the integer programming methods presented in section \ref{sectionip_formulation}.\\ 

\textbf{Problem 4} ($(n,d)$-queens completion). \vspace{0.2cm} \\
\begin{tabular}{rll}
    \textsc{Problem:}
    & \textit{A $(n,d)$-board,} $n,d \in \mathbb{N}^{+}$, 
     \textit{a set of mutually non-attacking } \\
    & \textit{queens } $Q^*$ \textit{with} $|Q^*| \leq n^{d-1}$, $\norm{Q^*} \leq n $ \vspace{0.2cm} \\
    \textsc{Solution:} & \textit{A set $Q$ of queens such that: \vspace{-0.2cm}}\\
\end{tabular}
\begin{itemize}
    \item[i)] $|Q|=n^{d-1}$
    \item[ii)] $\norm{Q} \leq n $
    \item[iii)] \textit{For any distinct pair of queens} ${q}_1$, ${q}_2$ $\in Q$, ${q}_1$ \textit{does not attack} ${q}_2$
    \item[iv)] $Q^* \subseteq Q$
\end{itemize}

Note that this definition differs from the definition given by \cite{gent2017complexity}, which allows the queens in $Q^*$ to be set on the same diagonal. 

\newpage
\begin{proposition}
    The $(n,d)$-queens completion for some given $Q^*$ has no solution if and only if there does not exists a solution $Q$ to the corresponding $(n,d)$-queens problem with $Q^* \subseteq Q$
\end{proposition}
In particular, for $Q^*$ with  $|Q^*| = 1$ this means that it is possible to check whether a certain square may be occupied without having to check all solutions to the $(n,d)$-queens problem exhaustively.\\
Again, for instances where the $(n,d)$-queens problem does not have a solution, we are be interested in maximal partial solutions:\\

\textbf{Problem 5} (partial $(n,d)$-queens completion). \vspace{0.2cm} \\
\begin{tabular}{rll}
    \textsc{Problem:} 
    & \textit{A $(n,d)$-board,} $k\leq n^{d-1}$, $n,d,k \in \mathbb{N}^{+}$, 
     \textit{a set of mutually } \\
    & \textit{non-attacking queens } $Q^*$ \textit{with} $|Q^*| \leq n^{d-1}$, $\norm{Q^*} \leq n $ \vspace{0.2cm} \\
    \textsc{Solution:} & \textit{A set $Q$ of queens such that: \vspace{-0.2cm}}\\
\end{tabular}
\begin{itemize}
    \item[i)] $|Q|=k$
    \item[ii)] $\norm{Q} \leq n $
    \item[iii)] \textit{For any distinct pair of queens} ${q}_1$, ${q}_2$ $\in Q$, ${q}_1$ \textit{does not attack} ${q}_2$
    \item[iv)] $Q^* \subseteq Q$
\end{itemize}

As discussed in section \ref{subsubsection:queensgraph} 
, the partial $(n,d)$-queens problem (or to be precise its respective optimization problem) can be interpreted and formulated as maximal independent set on the $(n,d)$-queen graph. 
Naturally one may also ask, what the minimum dominating sets of queens for such a given board are. This problem in turn has a visual interpretation as a variant of Chv\'{a}tal's art gallery problem \cite{alpert2021art}. Hence we define the corresponding decision problem to the minimal $n$-queen domination:\\ 

\textbf{Problem 6} ($(n,d)$-queens domination). \vspace{0.2cm} \\
\begin{tabular}{rll}
    \textsc{Problem:} 
    & \textit{A $(n,d)$-board,} $k\leq n^{d-1}$, $n,d,k \in \mathbb{N}^{+}$ \vspace{0.2cm} \\
    \textsc{Solution:} & \textit{A set $Q$ of queens such that: \vspace{-0.2cm}}\\
\end{tabular}
\begin{itemize}
    \item[i)] $|Q|=k$
    \item[ii)] $\norm{Q} \leq n $
    \item[iii)] \textit{Every square of the board is attacked by at least one queen} $q \in Q$
\end{itemize}

Analog definitions to problems 1,2,4,5, and 6 for the rook are straight forward.

\newpage
\subsubsection{Solutions}
\begin{definition}[regular solution \cite{Bell2009}]
A regular (or sometimes called linear) solution is a certificate for the $(n,d)$-queens problem, that can be constructed by a starting square $(s_1,s_2,...)$ and a fixed movement $(m_1,m_2,...)$ that places the $k$-th queen at $(s_1 + (k \cdot m_1) \text{ mod } n, s_2 + (k \cdot m_2) \text{ mod } n,...)$.
\end{definition}
Note that the existence of such regular solutions implies that the corresponding problem is solvable in polynomial time.\\ 

\begin{definition}[superimposable solutions \cite{Bell2009}]
A disjoint set of solutions to the $(n,d)$-queens problem is called superimposable.
\end{definition}
In other words, a set of superimposable solutions to the $(n,d)$-queens problem can be placed on the $(n,d)$-board without overlap. The cardinality of a superimposable set cannot be greater than $n$.\\

\begin{definition}[partial solution]
A set of queens $Q$ is called a partial solution on the $(n,d)$-board if:
\begin{itemize}
    \item[i)] $\norm{Q}\leq n $
    \item[ii)] \textit{For any distinct pair of queens} ${q}_1$, ${q}_2$ $\in Q$, ${q}_1$ \textit{does not attack} ${q}_2$
\end{itemize}
\end{definition}
This is a generalization of maximal partial solutions and will be used to describe the output of heuristics, as maximality might not be proven for such solutions.

We refer to \cite{Bell2009} who introduce further classes of solutions, such as \textit{symmetric solutions}, \textit{doubly symmetric solutions} and \textit{doubly periodic solutions}, which will not be of importance for this thesis but may be of use for further research on the partial modular $(n,d)$-queens problem and for the enumeration of regular solutions to the $(n,d)$-queens problem.

\newpage
\section{The Classical $(n,2)$-Queens Problem}
The following is a brief overview of selected theoretical results regarding the $(n,d)$-queens problem, focused on those relevant to solving and understanding the problem in higher dimensions. For a detailed overview of all known results on n-queens up until 2009, \cite{Bell2009} remains the most comprehensive source. 
\subsection{Existence and Construction of Solutions}
\begin{theorem}[Pauls \cite{pauls1874maximalproblem}]\label{thm_pauls}
    For all $n \geq 4$ the $(n,2)$-queens problem has a solution.
\end{theorem}
Pauls provided the first proof, listing construction methods for all residue classes of $(n \text{ mod } 6)$. Throughout the history of the problem, a large number of different authors have given different proofs for various constructions. An overview of all construction techniques is expanded on in \cite{Bell2009}. The often cited construction of \cite{Hoffman1969} (c.f. \cite{Bernhardsson1991}) distinguishes between three cases:
\begin{itemize}
    \item[(A)] $n$ even, $n \neq 6k + 2$: Place queens on: \vspace{0.1cm}\\
    $(j,2j)$\\
    $(n/2+j,2j-1)$, for $j = 1,...,n/2$
    \item[(B)] $n$ even, $n \neq 6k$: Place queens on: \vspace{0.1cm}\\
    $(j,1+(2(j-1)+n/2-1 \text{ mod } n))$\\
    $(m+1.j, n-(2(j-1)+n/2-1 \text{ mod } n))$, for $j = 1,...,n/2$
    \item[(C)] $n$ odd: Use suitable A or B for $(n-1)$ and add one queen on (n,n)
\end{itemize}

The resulting solutions for even $n$ are regular solutions. Solutions for odd $n$ are not regular by definition, but as they are constructed by simply extending a regular solution, the same argument regarding complexity holds. Thus, finding one single certificate for the $(n,2)$-queens problem is trivial; however, finding all solutions for given $n$ is not.
\begin{figure}[H]
    \centering
    \includegraphics[width=0.9\linewidth]{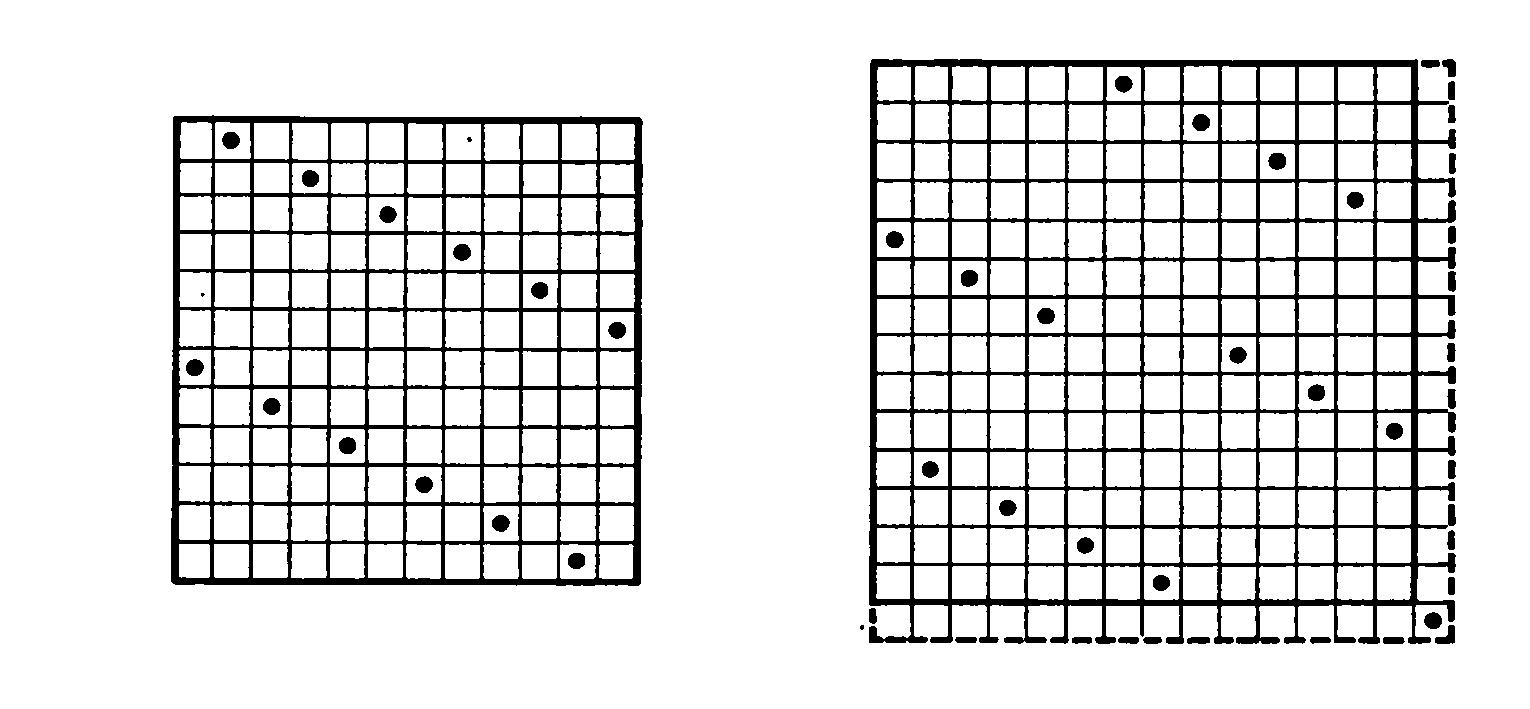}
    \caption{certificates for the $(n,2)$-queens problem by \cite{Hoffman1969}, $n=12,14,15$}
    \label{tab:vis_2d_constr}
\end{figure}

\subsection{Enumeration of Solutions}\label{subsection:enumeration_2d}
Let $\mathcal{Q}(n,d)$ denote the number of solutions to the $(n,d)$-queens problem. Table \ref{table_enumeration_2d} shows all currently known exact values of the sequence $\mathcal{Q}(n,2)$.\\
\begin{table}[H]
\begin{center}
\noindent
\begin{minipage}{0.8\linewidth}
\hspace{0.5cm}
\begin{tabularx}{0.3\textwidth}{*{3}{r}}
    \hline 
    $n$ & $\mathcal{Q}(n,2)$ \\ \hline%\hline
    0&		1 \\
    1&		1 \\
    2&		0 \\
    3&		0 \\
    4&		2 \\
    5&		10 \\
    6&		4 \\
    7&		40 \\
    8&		92 \\
    9&		352 \\
    10&		724 \\
    11&		2680 \\
    12&		14200 \\
    13&		73712 \\ \hline
\end{tabularx}
\hspace{0.5cm}
\begin{tabularx}{0.5\textwidth}{*{2}{r}}
    \hline 
    $n$ & $\mathcal{Q}(n,2)$ \\ \hline%\hline
    14&		365596\\
    15&		2279184\\
    16&		14772512\\
    17&		95815104\\
    18&		666090624\\
    19&		4968057848\\
    20&		39029188884\\
    21&		314666222712\\
    22&		2691008701644\\
    23&		24233937684440\\
    24&		227514171973736\\
    25&		2207893435808352\\
    26&		22317699616364044\\
    27&		234907967154122528\\ \hline
\end{tabularx}
\end{minipage}
\caption{\label{table_enumeration_2d}Number of solutions to the $(n,2)$-queens problem \cite{A000170}}
\end{center}
\end{table}
\begin{theorem}[Hsiang, Hsu, Shieh \cite{hsiang2004hardness}]
Finding all the solutions for the $n(n,2)$-queens problem and the modular $(n,2)$-queens problem are both beyond the \#P-class.
\end{theorem}
This implies that there exists no close form to enumeration of solutions to the $(n,2)$-queens problem in $d=2$. However, significant improvements have been made in tightening lower and upper bounds by \cite{luria2017new} \cite{Luria2021}  \cite{Simkin2021} and most recently \cite{nobel2023computing}.\\

\begin{theorem}[Nobel, Agrawal, Boyd \cite{nobel2023computing} and Simkin \cite{Simkin2021}]\label{thm_simkin_nobel}
    There exists a constant $1.944000752019729 < \alpha < 1.9440010813092217$ such that
    \begin{align}
        \lim_{n \to \infty} \frac{\mathcal{Q}(n,2)^{1/n}}{n} = e^{-\alpha}
    \end{align}
\end{theorem}
$\alpha$ is called the $n$-queens constant \cite{nobel2023computing} \cite{A359441}.

\newpage
\subsubsection{Density of Solutions}
\begin{theorem}[Simkin \cite{Simkin2021}]
    Let $\mathcal{R}$ be the collection of subsets of the plane of the form
    \begin{align*}
        \{ (x,y) \in [ -1/2, 1/2 ]² : a_1 \leq x+y \leq b_1, a_2 \leq y-x \leq b_2 \}
    \end{align*}
    for $a_1,a_2,b_1,b_2 \in [-1,1]$. Let $\gamma_1, \gamma_2$ be two finite Borel measures on  $[ -1/2, 1/2 ]²$. We define the distance between $\gamma_1$ and $\gamma_2$ by
    \begin{align*}
        d_{\diamond}(\gamma_1,\gamma_2) = sup\{ |\gamma_1(\alpha)-\gamma_2(\alpha)|: \alpha \in \mathcal{R} \}
    \end{align*}
    Let $Q$ be an $(n,2)$-queens solution.\\ 
    Define the step function $g_Q : [-1/2, 1/2]² \rightarrow \mathbb{R}$ by $g_Q \equiv n$ on every square 
    $( ( \frac{i-1}{n} - \frac{1}{2}, \frac{i}{n} - \frac{1}{2}) \times 
    ( ( \frac{j-1}{n} - \frac{1}{2}, \frac{j}{n} - \frac{1}{2})$ such that $(i,j) \in Q$ and $g_Q \equiv 0$ elsewhere. Let $\gamma_Q$ the probablity measure with density function $g_Q$. 

    Then there exists a Borel probability measure $\gamma^*$ on $ [-1/2, 1/2 ]^2$ such that the following holds:
    Let $\epsilon > 0$ be fixed and let $q$ be a uniformly random n-queens configuration. With high probability $d_{\diamond}(\gamma_Q,\gamma^*)< \epsilon$. 
\end{theorem}

\begin{figure}[H]
    \centering
    \includegraphics[width=0.7\linewidth]{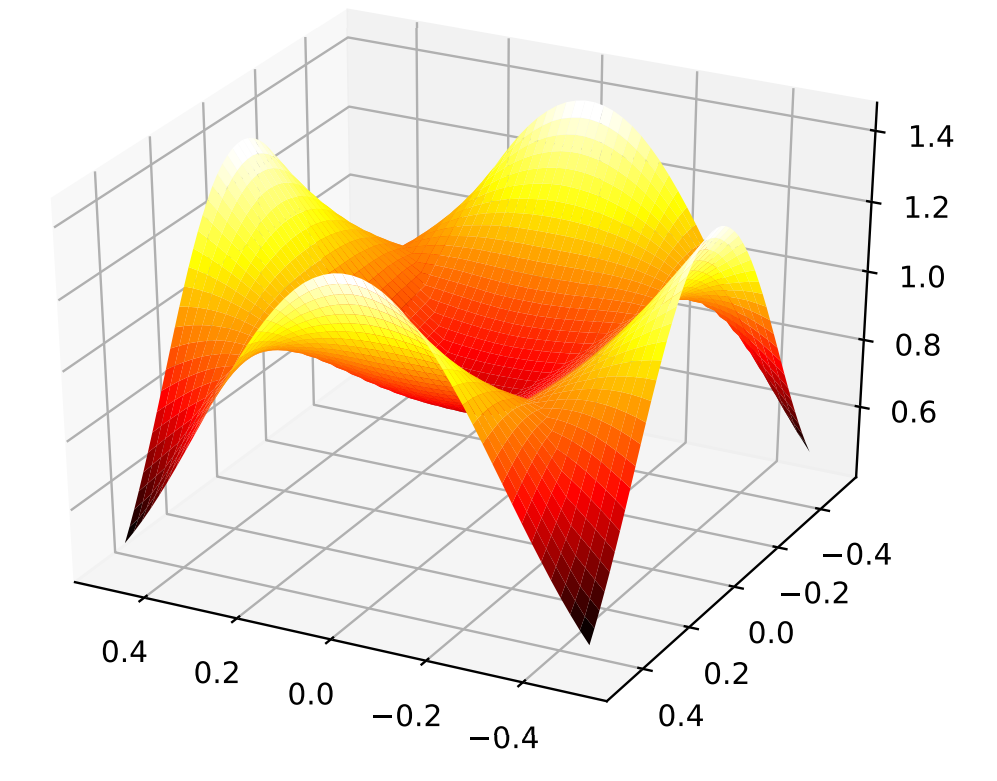}
    \caption{Density $\gamma^*$, the distribution of queens for the $(n,2)$-queens problem}
    \label{fig:dens_simkin}
\end{figure}

Intuitively, one may understand this density by observing that squares near the corners dominate fewer squares, while a queen placed in the middle of the board dominates the maximum number of squares. At the same time, placing a queen in a corner forbids the placement of further queens in both that row and column, so it appears to be least expensive to place queens that min-max this relation. Tied to this intuition, \cite{Simkin2021} derives the density function through a martingale construction that iteratively places queens on the board.

\begin{figure}[H]
\setkeys{Gin}{width=0.5\linewidth}
\includegraphics{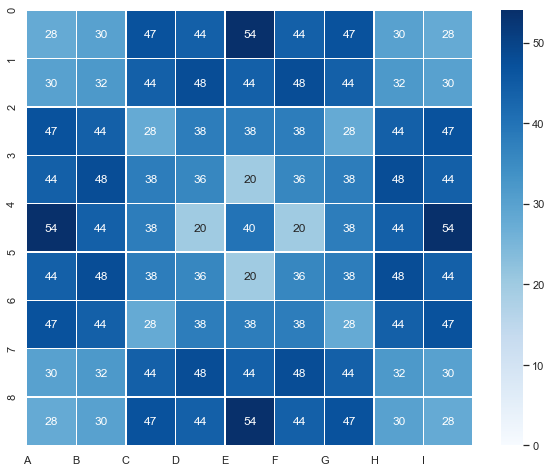}
\hfill
\includegraphics{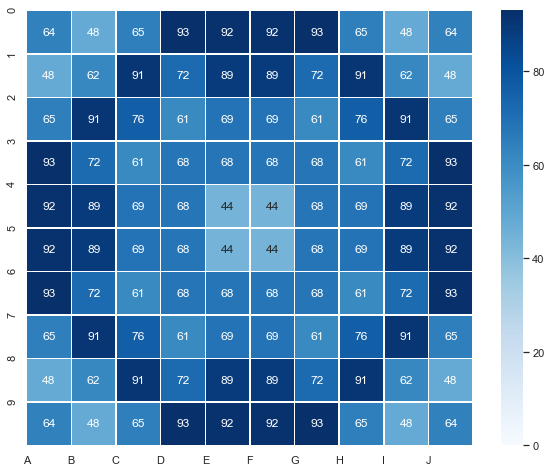}

\smallskip
\includegraphics{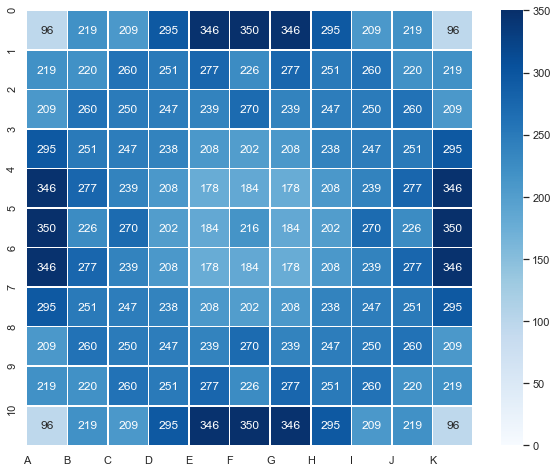}
\hfill
\includegraphics{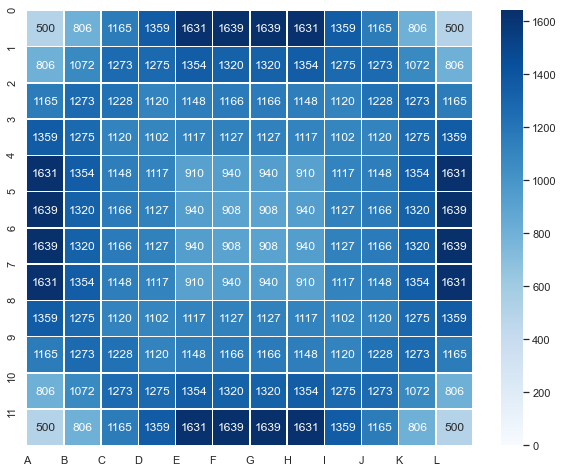}

\caption{Illustrating the density of solutions to the $(n,2)$-queens problem for $n=9,10,11,12$. The number in each square describes the number of solutions that place a queen on that square.}
\label{fig:density_d2}
\end{figure}

\newpage
\section{The $(n,3)$-Queens Problem}
While polynomial time construction methods exist for all $(n,2)$-boards, this is not the case for maximal partial solutions in $d\geq3$. This section discusses the specific case for $d=3$, i.e., the $(n,3)$-queens problem and different solution methods. These methods appear more intuitive when introduced in three-dimensional space. They connect to Latin squares and the queen's graph. Many results naturally extend to higher dimensions $d>3$, which will be the topic of the following section \ref{sec_gereral_dimension}.
\subsection{Existence and Construction of Solutions}
\begin{definition}[Latin Square \cite{mckay2005number}]
    For $n \in \mathbb{N}$, a Latin square is a $n \times n$ array with entries in $(1,2,\dots ,n)$ such that the entries in each row and in each column are distinct.
\end{definition}

The generalization of the $n$-queens problem to the third dimension is first proposed by McCarthy \cite{mccarty1978queen}. In his work, he draws the connection to Latin squares:\\
\\
\textit{"The placing of the $n^2$ non-attacking rooks into an $n$-cube is now reduced to the filling of an $n \times n$ grid with $n$ copies of the set $\{ 1,2,...,n\}$ such that no to element is in the same row or column twice. This problem is equivalent to finding a Latin square of order $n$. Such a square can be obtained, for
example, by cyclically permuting the elements $\{1, 2, ... , n\}$ and gives the desired maximal solution for
any $n$."}\\
\\ 
Continuing this thought, we may describe the $(n,3)$-queens problem as equivalent to finding Latin squares of order $n$ with additional diagonal constraints. This representation shines light on the problem of enumerating solutions; it suggests using methods of enumerating higher dimensional permutations and permutations with additional constraints as discussed by \cite{keevash2018existence} may be a suitable approach to determine bounds for the number of solutions to the $(n,3)$-queens problem.\\ 
\cite{mccarty1978queen} further mentions the trivial upper bound $|Q_{max}(n,3)| \leq n^2$ for the maximal partial solutions for $d=3$. We formulate the following observation:\\
\begin{proposition}[]
\label{cor_superimpo}
Any solution $Q$ to the $(n,3)$-queens problem also yields $n$ superimposable solutions to the $(n,2)$-queens problem in each of the three dimensions of the cube.  
\end{proposition}
\begin{proof}
    Given a solution to the $(n,3)$-queens problem choose one dimension without loss of generality. Consider each of the $n$ layers in the chosen dimension as a $(n,2)$-board, by dropping the coordinate of the chosen dimension. $n^2$ queens are placed on the $(n,3)$-board, so $n^2$ queens are placed on all $n$ $(n,2)$-boards in total. No more than $n$ queens can be placed on a single  $(n,2)$-board and exactly $n$ have to placed on each board for them to add up to $n^2$. Hence the layers of the $(n,3)$-solution yield $n$ $(n,2)$-solutions. Those $n$ solutions are superimposable due to the initial $(n,3)$-solution being a non-attacking configuration in the chosen dimension in particular.
\end{proof}
Every $n \times n$ layer of a partial solution of the $n \times n \times n$ cube may only contain a maximum of $n$ queens. Conversely, every layer of a solution to the $(n,2)$-queens problem has to contain exactly $n$ queens. 

Lastly, \cite{mccarty1978queen} provides the first lower bounds on maximal partial solutions, using search algorithms (the method is not further elaborated on) for $n$ up to 18. The following table also includes the currently best-known bounds in comparison:\\

\begin{table}[H]
\begin{center}
\noindent
\begin{minipage}{0.8\linewidth}
\hspace{0.5cm}
\begin{tabularx}{0.4\textwidth}{*{3}{r}}
    \hline 
    $n$ & \cite{mccarty1978queen} & current \\ \hline%\hline
    3 & 4 & \textbf{4} \\ %\hline
    4 & 7 & \textbf{7}\\% \hline
    5 & 13& \textbf{13}\\% \hline
    6 & 18& \textbf{21} \\% \hline
    7 & 27& \textbf{32} \\% \hline
    8 & 34&  \textbf{48}\\ %\hline
    9 & 43&  \textbf{67}\\ %\hline
    10& 58&  \textbf{91}\\ \hline
\end{tabularx}
\hspace{0.5cm}
\begin{tabularx}{0.4\textwidth}{*{3}{r}}
    \hline 
    $n$ & \cite{mccarty1978queen} & current \\ \hline%\hline
    11 & 68 & \textbf{121} \\ %\hline
    12 & 80& 133\\% \hline
    13 & 96& \textbf{169}\\% \hline
    14 & 111& 172 \\% \hline
    15 & 132& 199 \\% \hline
    16 & 151&  241\\ %\hline
    17 & 171&  \textbf{289}\\ %\hline
    18&  191&  307\\ \hline
\end{tabularx}
\end{minipage}
\caption{\label{table-1}lower bounds for $|Q_{max}(n,3)|$}
(bold entries are proven maximal)
\end{center}
\end{table}
The original table from \cite{mccarty1978queen} additionally includes the variable $R(n) := \frac{|Q_{max}(n,2)|}{n^2}$ and asks, if there exists an upper bound on $R(n)$. We now know that this bound is $1$ from the construction of the following theorem \ref{Klarner}. But we may still ask:\\
\\
\textit{For given $n_0$, can we find a lower bound $\alpha_{n_0}$ for $R(n)$, such that:}
\begin{align}
    \alpha_{n_0} \leq R(n) = \frac{|Q_{max}(n,3)|}{n^2} \text{ , } \forall n \geq n_0
\end{align}\\
This question will be addressed again in section \ref{subsec_lowerbounds3d}, where heuristics for partial solutions are discussed. 
\\
\subsubsection{Regular Solutions}
\begin{theorem}[Klarner \cite{klarner1979queen}]
\label{Klarner}
If $\text{gcd}(n,210) = 1$ then $n^2$ non-attacking queens can be placed on the $(n,3)$-board, i.e. the $(n,3)$-queens problem has a solution.
%\[ x^2 + y^2 = z^2 \]
\end{theorem}
\begin{proof}
There is a simple construction which gives a solution to this problem for all $n$ whose largest prime factor exceeds $7$. For any queen $q=(q_1,q_2,q_3)$, we may write $q_3$ as a function of $q_1$ and $q_2$: $q_3(q_1, q_2)$. For a non-attacking placement of $n^2$ queens the following equations are then both necessary and sufficient:\\
\begin{align}
&\text{For all $k$ with } 1 \leq k \leq n-1 \nonumber \\
\vspace{0.3cm}
    &q_3(q_1+k,q_2) - q_3(q_1,q_2) \neq 0 \nonumber\\
    &q_3(q_1+k,q_2) - q_3(q_1,q_2) \neq k \nonumber\\
    &q_3(q_1+k,q_2) - q_3(q_1,q_2) \neq -k \nonumber\\
    &q_3(q_1,q_2+k) - q_3(q_1,q_2) \neq 0 \nonumber\\
    &q_3(q_1,q_2+k) - q_3(q_1,q_2) \neq k \nonumber\\
    &q_3(q_1,q_2+k) - q_3(q_1,q_2) \neq -k \nonumber\\
    &q_3(q_1+k,q_2+k) - q_3(q_1,q_2) \neq 0 \nonumber\\
    &q_3(q_1+k,q_2+k) - q_3(q_1,q_2) \neq k \nonumber\\
    &q_3(q_1+k,q_2+k) - q_3(q_1,q_2) \neq -k 
\end{align}
The construction takes the form $q_3(q_1,q_2) = (a \cdot q_1 + b \cdot q_2) \text{ mod } n$ for $a,b \in \mathbb{N}$. The listed equations become the requirement that all of the numbers $e_0+e_1a+e_2b$ with $e_0,e_1,e_2 \in \{ -1,0,1 \}$ have no prime factor in common with $n$. It is easy to check that this condition cannot be met when 2, 3, 5 or 7 is a factor of $n$. However, if all of the prime factors of $n$ exceed 7, then it can be $a=3$, $b=5$ which yields a solution for any such $n$.
\end{proof}
\cite{van1981latin} proves the same results, so does \cite{MontgomerySmith}. A slightly weaker result is shown by \cite{herzberg1981latin}, stating that for any prime $n \geq 11$, a solution exists, following a construction using a similar knight-move pattern. 

Each solution yields a set of $n$ superimposable solutions for the $(n,2)$-queens problem. Different choices for $a,b$ other than 3, 5 are possible, depending on $n$. The constructed solutions are regular solutions. 
\begin{figure}[H]
    \centering
    \includegraphics[width=0.65\linewidth]{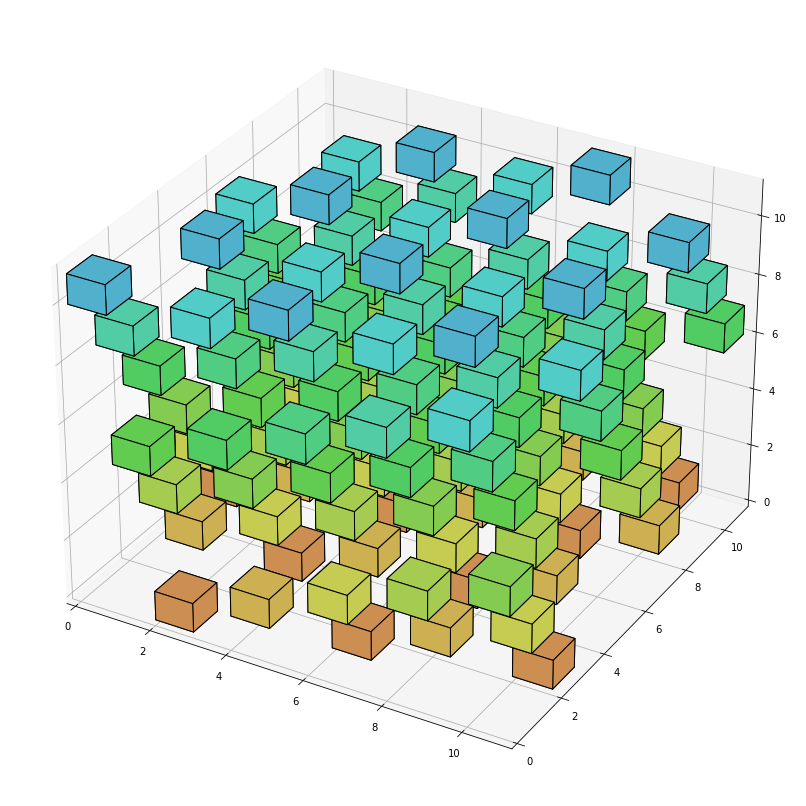}
    \caption{regular solution for $n=11, d=3$}
    \label{fig:klarner}
\end{figure}
Figure \ref{fig:klarner} is an example of the construction of a regular solution for $n=11$ with the choice of $a=3$ and $b=5$. Each $z$-layer is colored in the same hue, distinguishing the 11 superimposable solutions. Figure \ref{fig:klarner2} shows the same solution viewed from above the $xy$-plane.
\begin{figure}[H]
    \centering
    \includegraphics[width=0.65\linewidth]{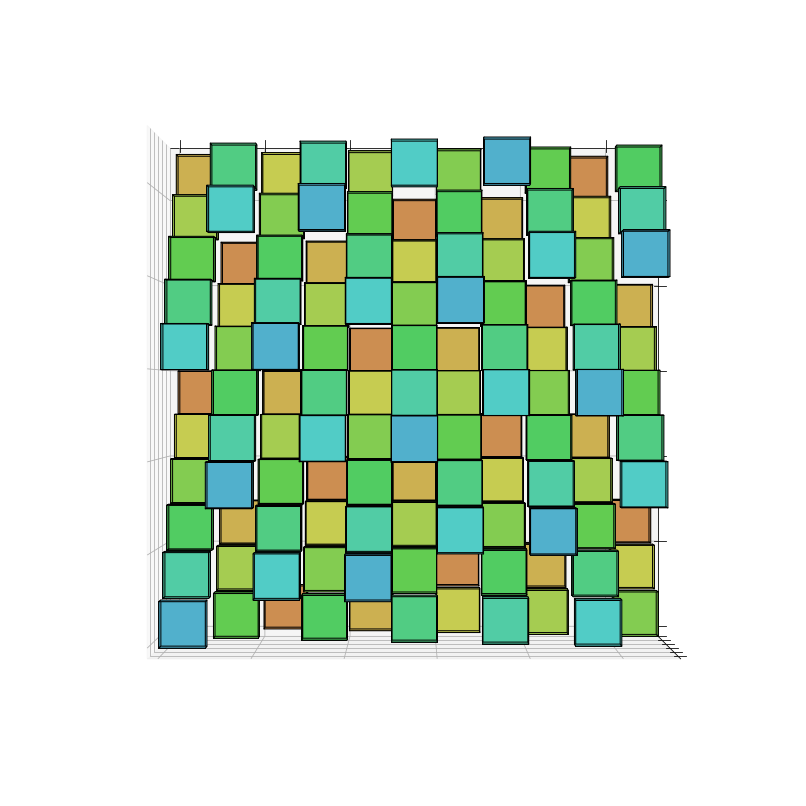}
    \vspace{-1cm}
    \caption{regular solution for $n=11, d=3$}
    \label{fig:klarner2}
\end{figure}

Due to their regularity, we can distinguish $n$ classes of solutions by fixing a queen on $(1,1,s)$, $s \in \{ 1, 2, ..., n\}$. These classes of solutions are equivalent, as all of them can be obtained by shifting the solutions of one class in the previously fixed dimension and applying modulo $n$ if necessary. A visual interpretation of this is given below in Fig. \ref{fig:shifting_property}, by identifying the two opposing sides in said dimension with one another. This shifting property allows us to focus on only one class of solutions, that is, to fix a queen on $(1,1,1)$ when discussing the enumeration and density of such solutions.

\begin{figure}[H]
    \centering
    \includegraphics[width=0.85\linewidth]{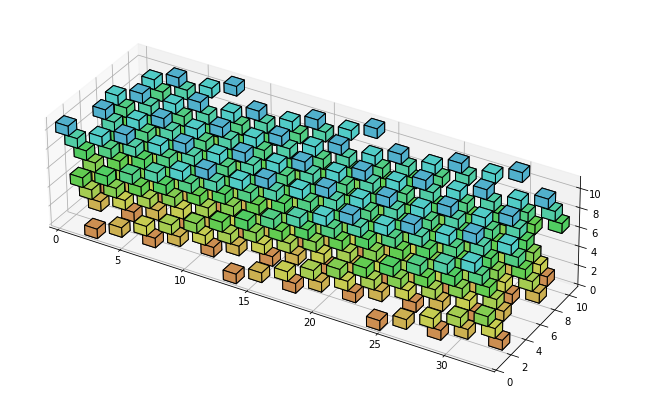}
    \caption{regular solution for $n=11, d=3$, continued}
    \label{fig:shifting_property}
\end{figure}

While the construction by \cite{klarner1979queen} shows that there exists a solution to the $(n,d)$-queens problem for $d=3$ for infinitely many $n$, a necessary condition for the existence of solutions remains to be shown. \cite{van1981latin} conjectures that $\text{gcd}(n,210)=1$ is not only sufficient (as shown in Theorem \ref{Klarner}), but also a necessary condition. So far, this conjecture is only supported by the fact that (a) no construction method for other $n$ exists and (b) instances for $n$, $\text{gcd}(n,210)>1$, for which the decision problem could be solved are infeasible.\\

\newpage
\subsubsection{Colouring the Queen Graph} \label{subsubsection:queensgraph}
Va\v{s}ek Chv\'{a}tal lists \textit{Colouring the queen graphs} under unsolved problems on his website \cite{chvatalonline}. He writes:

\textit{The $n \times n$ queen graph has the squares of an $n \times n$ chessboard for its vertices and two such vertices are adjacent if, and only if, queens placed on the two squares attack each other.}

\begin{figure}[H]
    \centering
    \includegraphics[width=0.5\linewidth]{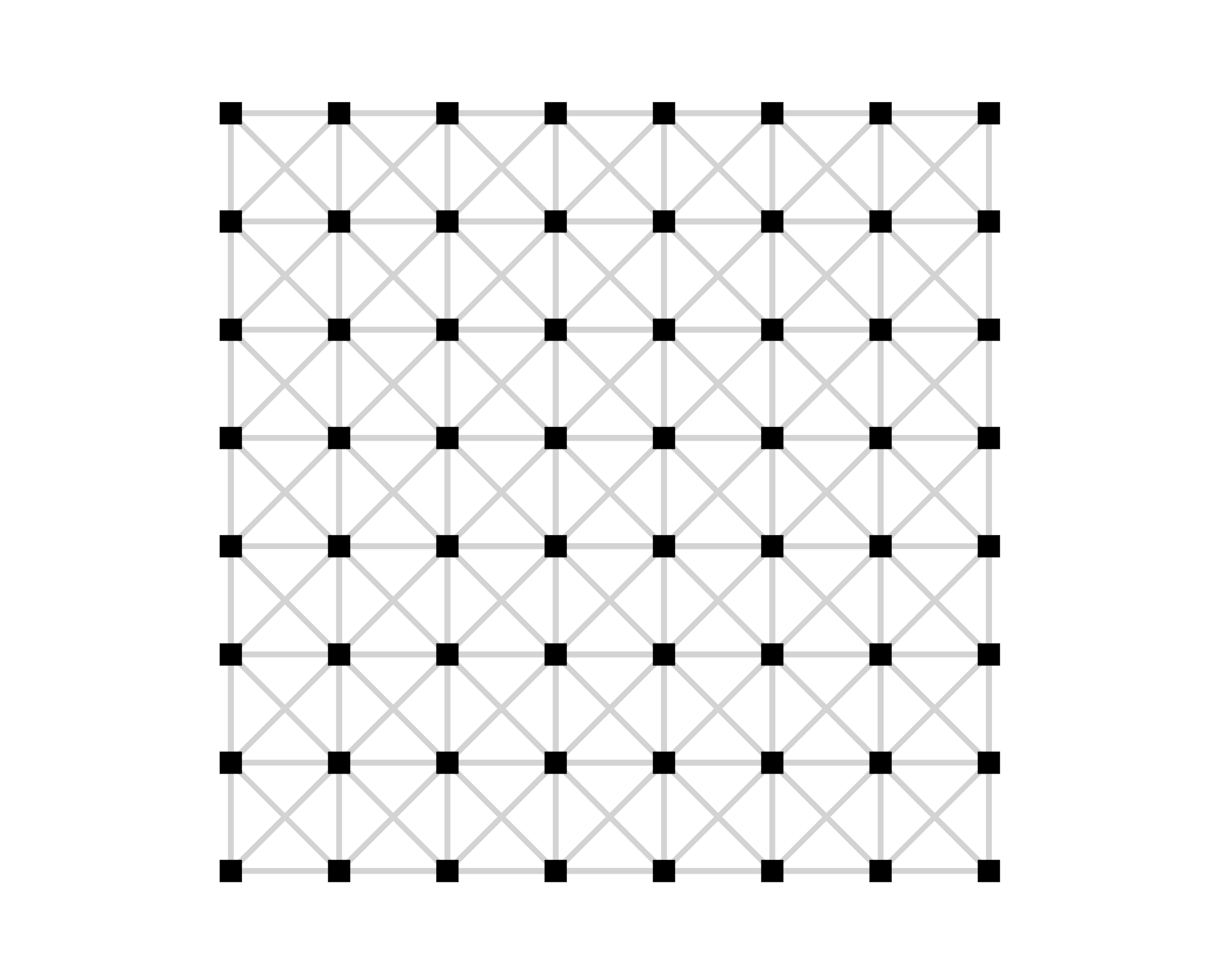}
    \caption{$(8,2)$-queen graph}
    \label{fig:queen_graph}
\end{figure}

Continuing the introduced notation, we will call such a graph the $(n,2)$-queen graph to distinguish it from graphs corresponding to boards of higher dimensions. An independent set of vertices on the $(n,2)$-queen graph corresponds to a mutually non-attacking configuration on the $(n,2)$-board. Thus, finding a solution to the $(n,2)$-queens problem is equivalent to finding a maximum independent set on the $(n,2)$-queen graph \cite{foulds1984application}. This, again, is equivalent to finding a maximal clique in the complement of the queens graph. We can generalize this observation for the $(n,d)$-queen graph for $d\geq 3$, leading to the integer programming formulation as an independent set problem. The resulting integer programming formulation and computational results are discussed in section \ref{sectionip_formulation}.

If we consider a vertex colouring of the $(n,2)$-queen graph $G$, called $G_{n,2}$, it is clear that for the chromatic number $\gamma (G_{n,2}) \geq n$. This is due to all $n$ vertices of a row, column, or the two long diagonals of the board being adjacent on the queen graph. The problem proposed by Chv\'{a}tal asks, for which $n$ there exists an $n$-colouring of the $(n,2)$-queen graph, meaning $\gamma (G_{n,2}) = n$.

An $n$-colouring of the queen graph contains exactly $n$ nodes of each color. As each set of $n$ nodes within the same color is an independent set on the queen graph, it directly corresponds to a solution to the $(n,2)$-queens problem.

\begin{proposition}
\label{cor_queensgraph}
    $n$-colouring the $(n,2)$-queen graph is equivalent to finding $n$ superimposable solutions to the $(n,2)$-queens problem.
\end{proposition}
\begin{proof}
$\Leftarrow$: clear. 
$\Rightarrow$: Through contradiction; Assume one colour appears on $k \neq n$ vertices. Then there exists at least one colour $c$ which appears on $k_c < n$ vertices. That would imply that there exists at least one row or column, in which $c$ does not appear, which leads to a contradiction as we require $n$ colours in each row and column to colour the graph. 
\end{proof}

\begin{table}[H]
    \centering
    \begin{tabular}{ccccc}
        0 & 1 & 2 & 3 & 4\\
        3 & 4 & 0 & 1 & 2\\
        1 & 2 & 3 & 4 & 0\\
        4 & 0 & 1 & 2 & 3\\
        2 & 3 & 4 & 0 & 1\\
    \end{tabular}
    \caption{certificate for the $(5,2)$-queen graph \cite{chvatalonline}}
    \label{tab:chvatal_queensgraph_certificate_n5}
\end{table}

Following Corollary \ref{cor_superimpo} and Corollary \ref{cor_queensgraph}, it is clear that for given $n$, the existence of an $n$-colouring of $G_{n,2}$ is a necessary condition for the existence of a solution to the $(n,3)$-queens problem. If it can be shown that no such an $n$-colouring of $G_{n,2}$, i.e., a set of $n$ superimposable solutions to the $(n,2)$-queens problem exists, then the $(n,3)$-queens problem has no solution. Chv\'{a}tal \cite{chvatalonline} discusses this in the context of coloring the queen graph and shows that no such set exists for $n=$ 8, 9, and 10. He further lists solutions for $n=$ 12, 14, 15, 16, 18, 20, 21, 22 and 24. \cite{vasquez2004complete} later extends this list with certificates for $n=$ 28 and 32. Given an $r$-colouring of the $r \times r$ graph and $s$-colouring of the $s \times s$ graph, a coloring of the $rs \times rs$ can be constructed (cf. \cite{abramson1986construction}). 

\begin{figure}[H]
    \centering
    \includegraphics[width=0.7\linewidth]{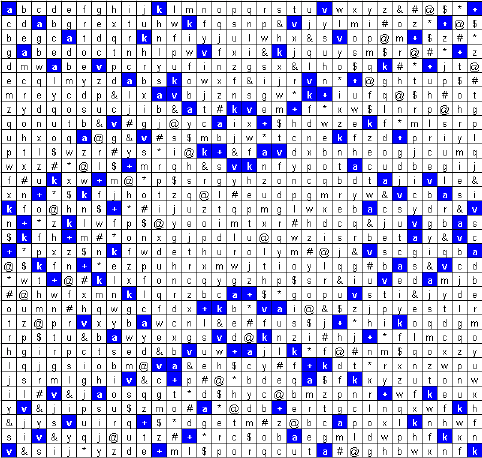}
    \caption{certificate for the $(32,2)$-queen graph \cite{vasquez2004complete}}
    \label{fig:klarner2}
\end{figure}

The two following theorems give insight into sufficient conditions for the existence of an $n$-colouring. However, a necessary condition remains to be found. The existence of such a necessary condition would have great implications for the problem at hand; as for the case of the existence of (infinitely many) $n$ that do not fulfill such a condition, we could conclude that the $(n,3)$-queens problem would not have a solution for those $n$. Conversely, the existence and construction of superimposable solutions may aid the construction of solutions to the (partial) $(n,3)$-queens problem. %, as demonstrated by \cite{Lyndenlea}.
\begin{theorem}[Iyer, Menon \cite{iyer1966coloring}] % rephrased from: If n is not divisible by 2 or 3, the chromatic number of the graph of Queen's move is n. 
\label{thm_iyer}
    If $n$ is not divisible by $2$ or $3$, the chromatic number of the $n$-queens graph is $n$.
\end{theorem}
\begin{theorem}[Vasquez \cite{vasquez2006coloration}] \label{thm_vasquez}
    There exist infinitely many integers $n$ of the form $n = 2 \cdot p$ or $n=3 \cdot p$ for $p$ prime, such that the $n$-queens graph is $n$-colourable.
\end{theorem}

\subsubsection{Stacking Superimposable Solutions}
Given sufficient upper bounds to $|Q_{max}(n,d|$ or searching for solutions for which $|Q_{max}(n,d|=n^{d-1}$ a possible approach to constructing maximal partial solutions is to stack superimposable solutions of the $(n,d-1)$-board. \\
First, consider the set of all solutions to the partial $(n,d-1)$-queens problem or a suitable chosen subset thereof. Then, find a feasible order of stacking those solutions that yields a maximal partial solution of the required size. \cite{Lyndenlea} follows this approach for the $(n,3)$-queens problem, choosing a subset of the regular solutions to the $(n,2)$-queens problem.

\newpage
\subsection{Lower Bounds} \label{subsec_lowerbounds3d}
\subsubsection{Subcube Heuristic}
We present a simple construction method for partial solutions on the $(n,3)$-board. This method achieves to construct the largest maximal partial solutions known for several of the instances in question. As it follows from the construction of solutions of Theorem \ref{Klarner}, it allows us to directly compute a lower bound calculating a closed expression for given $n$. In contrast to brute-force search algorithms or other computationally expensive heuristics, it does so for arbitrarily large $n$.

Given a partial solution to the $(n,3)$-queens problem of $k$ queens, which we suspect also to be a maximal partial solution, we may prove its maximality by showing that no configuration of $k+1$ queens exists using IP. This turns out to be a time-saving strategy compared to solving an integer program that tries to maximize the size of a valid configuration (see \ref{sectionip_formulation}) and allows us to make general statements about lower bounds. Additionally, partial solutions derived through this method may be used to warmstart the IPs.\\

\begin{theorem} \label{thm_subcubeheur_3d}
    For any given integer $m$, find suitable $n>m$, $n = m+k$ for which $\text{gcd}(n,210) = 1$. We can construct a partial solution of $n^2 - 3k \cdot n + 3k^2 - 2p$ queens on the $(m,3)$-board under the following condition: There exists a regular solution to the $(n,3)$-queens problem that contains a subcube of size $k$ with $p$ queens, meaning there are exactly $p$ queens placed in at least one $k \times k \times k$ subset of the board. 
\end{theorem}
\begin{proof}
    By construction. The proof idea is to find the $(m,3)$-board within in one of the regular solutions to the $(n,3)$-queens problem, that contains the maximum amount of queens.

    \begin{figure}[H]
        \centering
        \includegraphics[width=0.6\linewidth]{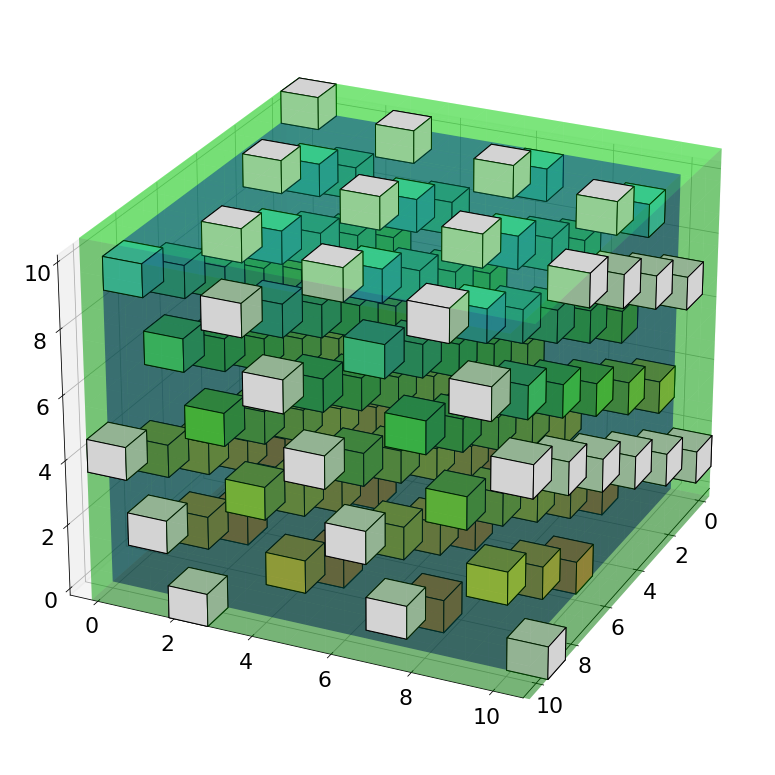}
        \caption{partial solution for the $(10,3)$-board (blue) contained in a solution to the $(11,3)$-queens problem (green) - cut off queens are light grey}
        \label{fig:subcube_example}
    \end{figure}
Given any solution to the $(n,3)$-queens problem, we recall that it contains $n^2$ queens and each of its layers contains $n$ queens. 
\begin{align*}
    |Q_{max}(n,3)| = n^2
\end{align*}
We obtain $Q_s(m,3)$, a partial solution for the $(m,3)$-board by removing $k$ layers in each dimension of the solution. The $k$ layers in each dimension have to be outer layers of the cube or adjacent layers, as we can shift the solution wlog. Let us call the set of queens in the removed $x$-layers $L_x$, the set of queens in the removed $y$-layers $L_y$ and the set of queens in the removed $z$-layers $L_z$.
\begin{align}
    |L_x| = |L_y| = |L_z| = k \cdot n
\end{align}
Our goal is to find a maximal subset of the initial solution that is a partial solution for the $(m,3)$-board, i.e. to minimize the queens removed, as:
\begin{align}
    |Q_s(m,3)| = |Q_{max}(n,3)| - |L_x \cup L_y \cup L_z|
\end{align}
There exists an intersection between the removed $x$-layers and $y$-layers, the $x$-layers and $z$-layers, the $y$-layers and $z$-layers and one intersection between all three. We define:
\begin{align}
    I_{xy} := (L_x \cap L_y) \setminus L_z \nonumber\\
    I_{xz} := (L_x \cap L_z) \setminus L_y \nonumber\\
    I_{yz} := (L_y \cap L_z) \setminus L_x \nonumber\\
    I_{xyz} := L_y \cap L_y \cap L_z
\end{align}
We can now write
\begin{align}
    |Q_s(m,3)| &= |Q_{max}(n,3)| - |L_x \cup L_y \cup L_z| \nonumber\\
    &= n^2 - |L_x \cup L_y \cup L_z| \nonumber\\
    &= n^2 - \big( |L_x| + |L_y| + |L_z| \big) + \big( |I_{xy}| + |I_{xz}| + |I_{yz}| + 2|I_{xyz}| \big) \nonumber\\
    &= n^2 - 3 \cdot k \cdot n + \big( |I_{xy}| + |I_{xz}| + |I_{yz}| + 2|I_{xyz}| \big)
\end{align}
$I_{xyz}$ is contained in a $k \times k \times k$ cube inside the initial $(n,3)$-board. If this subcube of size $k$ contains $p$ queens in the chosen solution to the $(n,3)$-queens problem, i.e. $|I_{xyz}|=p$, we can conclude that the $I_{xy}$, $I_{xz}$ and $I_{yz}$ all must contain $k^2-p$ queens each. They cannot contain less, as they respective rows/columns in their layers have to contain one queen (otherwise this leads to a contradiction as the total cube containing $n^2$ queens). And they cannot contain more due to maximality.
\begin{align}
    |Q_s(m,3)| &= n^2 - 3 \cdot k \cdot n + \big( |I_{xy}| + |I_{xz}| + |I_{yz}| + 2|I_{xyz}| \big) \nonumber \\
    &= n^2 - 3 \cdot k \cdot n + 3 \cdot (k^2 - p) + 2 \cdot p  \nonumber \\
    &= n^2 - 3 \cdot k \cdot n + 3 \cdot k^2 - p
\end{align}
Summarizing we conclude that we have removed $(3k \cdot n - 3k^2- p)$ queens from a solution to the $(n,3)$-queens problem and have obtained a partial solution for the $(m,3)$-board, containing $n^2 - 3k \cdot n + 3k^2- p$ queens.
\end{proof}
\newpage
In order to maximize the size of the derived partial solution, we want to minimize $p$, in particular for $p=0$ we get
\begin{corollary}\label{cor_empty_subcube_3d}
    For any given integer $m$, find suitable $n>m$, $n = m+k$ for which $\text{gcd}(n,210) = 1$. We can construct a partial solution of $n^2 - 3k \cdot n + 3k^2$ queens on the $(m,3)$-board under the following condition: There exists a regular solution to the $(n,3)$-queens problem that contains an empty subcube of size $k$. 
\end{corollary}

Note that the existence of an empty subcube of size $k$ is trivial for $k=1$. An empty subcube of size $2$ fits in any regular solution constructed following Theorem \ref{Klarner} using the choice of parameters $a=3$ and $b=5$. Larger subcubes may fit depending on the choice of $a,b$. However, this is dependent on $n$, as larger $n$ allow for more choices of $a,b$. 

Also note that a slight variant corollary \ref{cor_empty_subcube_3d} can be applied to obtain a bound for the $(m,3)$ board from any partial solution of the $(n,3)$ board, $m < n$, without the latter having to be a maximal partial solution. Here, we additionally require the empty subcube to be adjacent to a corner of the $(n,3)$ board. We can then always construct a subset of a  partial solution on the $(n-1,3)$-board, given a partial solution on the $(n,3)$-board for which 
\begin{align*}
    |Q_{s}(n-1,3)| \geq |Q_{p}(n,3)| - 3 \cdot n + 3 
\end{align*}   
In this case, the result might not be the best lower bound achievable, as it is easily demonstrated by the fact that applying the heuristic twice with $k=1, p=0$ is strictly worse than applying it once with $k=2, p=0$:
\begin{align}
    |Q_{max}(n,3)| &= n^2 \nonumber \\
    |Q_s(n-1,3)| &= n^2 - 3 \cdot n + 3 \nonumber\\
    |Q_s(n-1-1,3)| &= |Q_s(n-1,3)| - 3 \cdot (n-1) + 3 \nonumber\\
    &= n^2 - 3 \cdot n + 3 - 3 \cdot (n-1) + 3 \nonumber\\
    &= n^2 - 6 \cdot n + 9 \nonumber\\
    &< n^2 - 6 \cdot n + 12 \nonumber\\
    &= n^2 - 3 \cdot 2 \cdot n + 3 \cdot 2^2 \nonumber\\
    &= |Q_s(n-2,3)|
\end{align}
This underestimation is due to the fact that the partial solutions to which the heuristic is applied do not contain a queen in each row or column, and thus, the corresponding statement in the construction proof becomes an inequality (lower bound) instead of an equality. 

The following table \ref{table_subcube3d} compares the configurations derived from this heuristic for $p=0$ and $k=1$, $k=2$ with the best-known results to the partial $(n,d)$-queens problem. Entries that are proven maximal are printed in bold. All other entries are taken from the results of \cite{ZimmermannContest}.

\begin{table}[H]
\begin{center}
\noindent
\begin{tabularx}{0.4\textwidth}{*{4}{r}}
    \hline 
    $n$ & current & $k=1$ & $k=2$ \\ \hline%\hline
    9 &  \textbf{67}  &    &67\\ %\hline
    10 &  \textbf{91} & 91 &\\% \hline
    11 & \textbf{121} & &\\% \hline
    12 & \textbf{133} &133 & \\% \hline
    13 & \textbf{169} & &\\% \hline
    14 &  \text{172}  & &\\ %\hline
    15 &  \text{199}  &   &199\\ %\hline
    16 &  \text{241}  &241&\\ %\hline
    17 &  \textbf{289}& &\\ %\hline
    18 &  \text{307}  &307&\\ %\hline
    19 &  \textbf{361}& &\\ %\hline
    20 &  \text{364}  & &\\ %\hline
    21 &  \text{405}  & &403\\ %\hline
    22 &  \text{463}  &463&\\ %\hline
    23 &  \textbf{529}& &\\ \hline
    %... &  & &\\ \hline
\end{tabularx}
\caption{\label{table_subcube3d}lower bounds for $|Q_{max}(n,3)|$}
\end{center}
\end{table}
The heuristic matches the results from \cite{ZimmermannContest} for all bounds obtained with $k=1$ and up to $n=15$ with $k=2$. However, for $n=21$, a partial solution $Q$ with $405$ queens exists. As $405 > 403 = |Q|$, which would be the maximal subset of a solution to the $(23,3)$-queens problem that is a partial solution for the  $(21,3)$-board, we know that the configuration of $405$ queens cannot be such a subset. Notably, this difference of two queens on the $(21,3)$-board turned out to be the deciding difference between 1st and 2nd place in Al Zimmermann's contest \cite{ZimmermannContest}.

    \begin{figure}[H]
        \centering
        \includegraphics[width=0.65\linewidth]{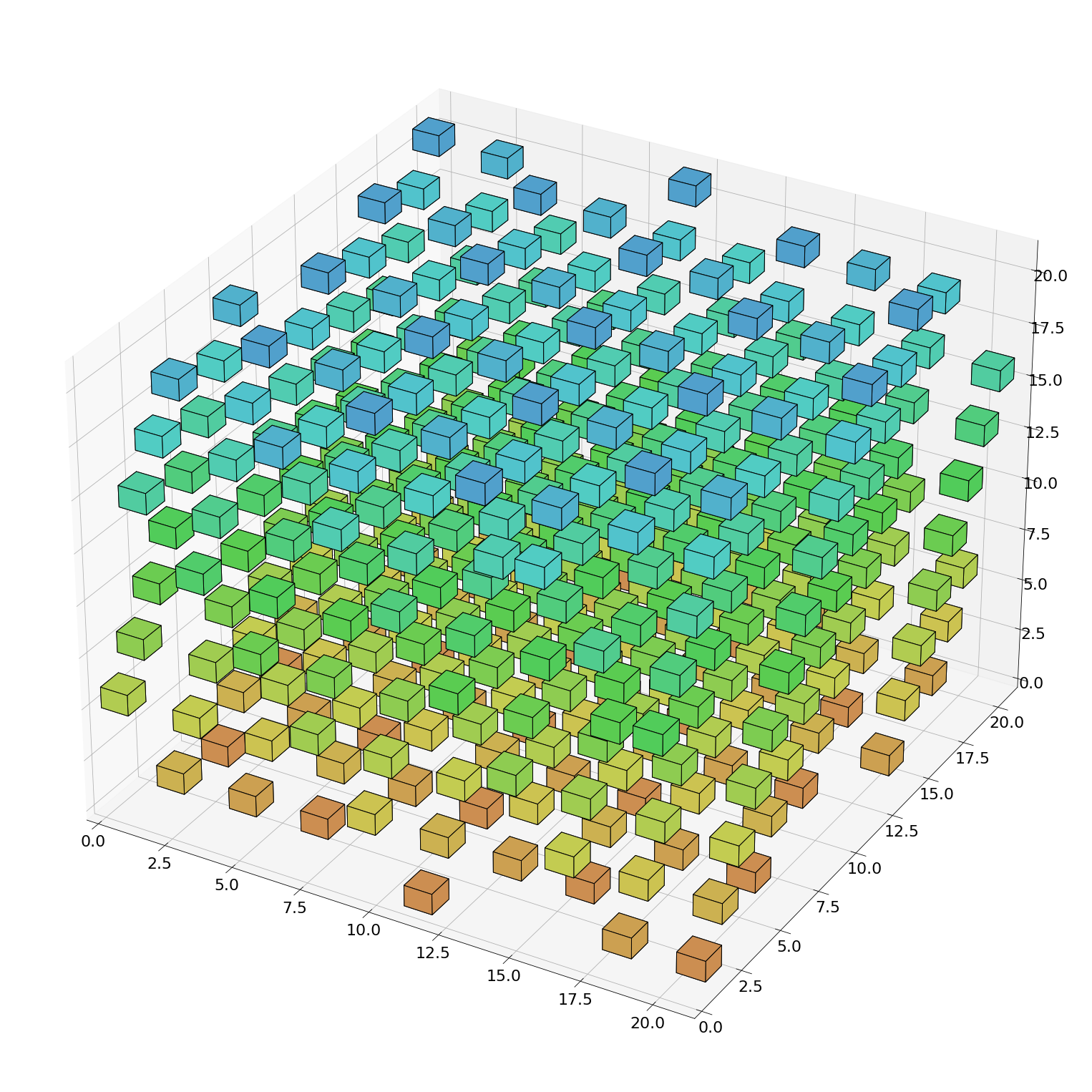}
        \caption{partial solution of $405$ queens on the $(21,3)$-board}
        \label{fig:n21_d3_405queens}
    \end{figure}

Further, we may now state lower bounds for $n$ that exceed the instances that have been tackled so far (see table. \ref{3d_lowerbounds_2}). It is unlikely that configurations derived through this method are also maximal partial solutions for large $n$, as demonstrated by the configuration for the $(21,3)$-board. Instead, we expect the maximal partial solutions for larger $n$ not to be subsets of regular solutions of larger boards. 
\begin{table}[H]
\begin{center}
\noindent
\begin{tabularx}{0.4\textwidth}{*{4}{r}}
    \hline 
    $n$ & current & $k=1$ & $k=2$ \\ \hline%\hline
    27 &  \text{}  &    &697\\ %\hline
    28 &  \text{} & 757 &\\% \hline
    29 & \textbf{841} & &\\% \hline
    30 & \textbf{} &871 & \\% \hline
    31 & \textbf{961} & & \\% \hline
    ... & \textbf{} & &\\% \hline
    35 &  \text{}  & &1159\\ %\hline
    36 &  \text{}  &  1261 &\\ %\hline
    37 &  \textbf{1369}  &   &\\ %\hline
    ... &  \text{}  &&\\ %\hline
    39 &  \textbf{}& &1447\\ %\hline
    40 &  \text{}  &1561&\\ %\hline
    41 &  \textbf{1681}  &&\\ %\hline
    ... &  \text{}  &&\\ \hline
\end{tabularx}

\caption{\label{3d_lowerbounds_2}lower bounds for $|Q_{max}(n,3)|$}
\end{center}
\end{table}

We may now revisit the question inspired by \cite{mccarty1978queen}:

\textit{For given $n_0$, can we find a lower bound $\alpha_{n_0}$ for $R(n)$, such that:}
\begin{align}
    \alpha_{n_0} \leq R(n) = \frac{|Q_{max}(n,3)|}{n^2} \text{ , } \forall n \geq n_0
\end{align}\\
Repeatedly applying the described heuristic to regular solutions to the $(n.3)$-queens problem 
one can also easily check, that $\lim_{n_0 \to \infty} \alpha_n = 1$.\\ 
\begin{corollary}\label{cor_limit_alpha}
For all $n$, $|Q_{max}(n,3)| \geq n^2 -10n - 32$
\end{corollary} 
\begin{proof}
    The next largest $m>n$ with $gcd(m,210)=1$ is at most $m=n+10$. Apply the lower bound from theorem \ref{thm_subcubeheur_3d} once with $k=2, p=0$ and then 8 times with $k=1, p=0$. 
\end{proof}
Note that this bound can be further improved for sufficiently large $n$.\\

\begin{corollary} \label{cor_alpha_limit_3d}
$ \lim_{n_0 \to \infty} \alpha_{n_0} = 1$
\end{corollary}
\begin{proof}
    \begin{align*}
        1 \geq \lim_{n_0 \to \infty} \alpha_{n_0} = \liminf_{n_0 \to \infty} R(n) \geq \lim_{n_0 \to \infty} \frac{n_0^2 -10n_0 - 32}{n^2} = 1
    \end{align*}
\end{proof}

\newpage
\subsection{Enumeration of Solutions} \label{section_enumeration_general_d}
\subsubsection{Lower Bounds} \label{subsubsec_lowerbounds_3d}
Through the construction of regular solutions following theorem \ref{Klarner}, we obtain a lower bound on the number of solutions for the $(n,3)$-queens problem for $n$ with $gcd(n,210)=1$. For given $n$, we may fix one of the shifting classes and observe that there are $12$ possible directions for the stepping pattern (4 in each dimension as for $(\pm a, \pm b)$). There are $n \cdot c(n)$ such regular solutions for all $n$, where $c(n)$ is a function that describes how many choices of $a,b$ there are for given $n$ considering symmetries. 

Table \ref{table_enumeration_d3} shows the currently known exact results for the number of solutions for both full and maximal partial solutions (in cursive). Results are computed through the IP formulation described in section \ref{sectionip_formulation}.\\
\begin{table}[H]
\begin{center}
\noindent
\begin{tabularx}{0.37\textwidth}{*{4}{r}}
    \hline 
    $n$ & $\mathcal{Q}(n,3)$ &$n \cdot {c(n)}$ &  \\ \hline%\hline
    1 &  \textit{1}  &    &\\ %\hline
    2 &  \textit{8} &  &\\% \hline
    3 & \textit{16} & &\\% \hline
    4 & \textit{1344} & & \\% \hline
    5 & \textit{1056} & & \\% \hline
    6 & \textit{912} & &\\% \hline
    7 &  \textit{96}  & &\\ %\hline
    8 &  \text{}  &   &\\ %\hline
    9 &  \textbf{}  &   &\\ %\hline
    10 &  \text{}  &&\\ %\hline
    11 &  264& $= 11 \cdot 12 \cdot 2$&\\ %\hline
    12 &  \text{}  &&\\ %\hline
    13 &  624&  $= 13 \cdot 12 \cdot 4$ &\\ %\hline
    14 &  \text{}  &&\\ \hline
\end{tabularx}
\caption{\label{table_enumeration_d3}Known sequence of $\mathcal{Q}(n,3)$ and partial solutions}
\end{center}
\end{table}
\cite{Lyndenlea} implemented a method that gives a lower bound on  $\mathcal{Q}(n,3)$ by exhaustively computing all regular solutions. The results are listed in table \ref{table_enumeration_d3_2}. Note, however that they may only represent a subset of the solutions.
\begin{table}[H]
\begin{center}
\noindent
\begin{tabularx}{0.35\textwidth}{*{2}{r}*{1}{l}}
    \hline 
    $n$ & lb.& $n \cdot c(n)$   \\ \hline%\hline
    17 & 2,040  & $17 \cdot 12 \cdot 10$ \\%\hline
    19 &  3,192 & $19 \cdot 12 \cdot 14$ \\%\hline
    23 &  6,624 & $23 \cdot 12 \cdot 24$ \\%\hline
    29 & 35,496 & $29 \cdot 12 \cdot 102$  \\% \hline
    31 & 19,344 & $31 \cdot 12 \cdot 52$ \\ \hline
\end{tabularx}
\caption{\label{table_enumeration_d3_2}Lower bound on $\mathcal{Q}(n,3)$ by \cite{Lyndenlea}}
\end{center}
\end{table}
For general statements on upper bounds see section \ref{subsubsec_upperbounds_generald}.

\newpage
\subsection{Open Questions}
In the following, we summarize and connect some of the previously mentioned open questions regarding the $(n,3)$-queens problem.\\
\begin{question}[Ch\`atal \cite{chvatalonline}]
    Does there exists $n_0$, such that for all $n \geq n_0$ the $(n,2)$-queen graph is $n$-colourable?
\end{question}
Recall that this is equivalent to the existence of $n$ superimposable solutions. Thus, the existence of an $n$-colouring is necessary for the existence of a solution to the $(n,3)$-queens problem. \\

\begin{question}[Van Rees]
    Is $\text{gcd}(n,210)=1$ a necessary condition for the $(n,3)$-queens problem to have a solution?
\end{question}

\begin{question}[]
    Are there non-regular solutions to the $(n,3)$-queens problem?
\end{question}
So far there are no certificates of non-regular solutions. Note that answering Question 1 might also help provide insights towards Question 3. We observed that if the only solutions existing are regular, $qc(n,3)$ would be constant and not depend on $n$.\\

\begin{question}
Does there exist an equivalent to the $n$-queens constant (see Theorem \ref{thm_simkin_nobel}) for $d=3$ for all $n$ with $gcd(n.210)=1$. 
Is there a relation between this constant and the $n$-queens constant?
\end{question}
This connects to the question of whether there exist non-regular solutions for sufficiently large $n$ and the enumeration of such solutions.\\

\begin{question}[]
    Does the density of the solutions to the partial $(n,3)$-queens problem converge for $n \to \infty$? 
\end{question}
Recall from corollary \ref{cor_limit_alpha} that for $n \to \infty$, one can place $n^2$ queens (in limit). The density of solutions is of particular interest for $d \geq 3$, as we differentiate between maximal partial solutions with $|Q_{max}(n,3)| < n^2$ who present a density reminiscent of the density for the $(n,2)$-queens problem as described by \cite{Simkin2021} and solutions with $|Q_{max}(n,3)| = n^2$ which so far showcase an equal distribution due to their regularity.

\newpage
\section{Generalization to Higher Dimensions}\label{sec_gereral_dimension}
\begin{theorem}[Nudelman \cite{Nudelman1995}]\label{theorem_nudelmanattackingplanes_higherdim}
    A queen on the $(n,d)$-board attacks in $\frac{1}{2}(3^d-1)$ hyperplanes.
\end{theorem}
\begin{proof}
    Recall the definition of attack lines for the modular $(n,d)$-queens problem (Problem 2). Consider all $\epsilon = (\epsilon_1, ... , \epsilon_d)$, there are $3^d-1$ non-zero possibilities. $\epsilon$ and $-\epsilon$ however define the same hyperplane of attack.
\end{proof}
\begin{corollary} \label{cor_max_attacking_squares}
    A $d$-dimensional queen on the $(n,d)$-board attacks up to $\frac{1}{2}(3^d-1) \cdot (n-1) +1 $ squares (including its own square).
\end{corollary}
\begin{proof}
    For any odd $n$, place a queen on $(\frac{n+1}{2},\frac{n+1}{2},...,\frac{n+1}{2})$. It attacks all $n$ squares in each of its $3^d-1$ hyperplanes of attack. 
\end{proof}
From this we can conclude that for $n=3$, the minimal dominating set problem on the has domination number $\gamma(G_{3,d})=1$ $\forall d$, as one queen can attack the entire board:
\begin{align}
    \frac{1}{2}(3^d-1) \cdot (3-1) +1 = 3^d
\end{align}
One can easily check this by placing a queen on $(2,2,...,2)$, as it is also mentioned by \cite{langlois2022complexity}.
\begin{corollary}
    A $d$-dimensional queen on the $(n,d)$-board attacks at least $(2^d-1)\cdot(n-1)+1$ squares (including its own square).
\end{corollary}
\begin{proof}
    Place a queen on $(1,1,...,1)$. This minimizes the total squares it attacks for each pair of two orthogonal diagonal hyperplanes of attack. Its hyperplanes of attack are described by all non-zero vectors $\epsilon$ with positive entries. There are $2^d-1$ such vectors.
\end{proof}
From this corollary it trivially follows that any single queen on the $(2,d)$-board dominates the entire board:
\begin{corollary} \label{cor_n2_qmax}
    $|Q_{max}(2,d)| = 1$ for all $d$.
\end{corollary}
\begin{proof}
    $(2^d-1)\cdot(2-1)+1 = 2^d$
\end{proof}

\subsection{Existence and Construction of Solutions}
\begin{proposition}[]
\label{cor_superimpo_3d}
Any solution $Q$ to the $(n,d)$-queens problem also yields $n$ superimposable solutions to the $(n,d-1)$-queens problem in each of the dimensions of the d-dimensional hypercube.  
\end{proposition}
\begin{proof}
    Analog to proposition \ref{cor_superimpo}.
\end{proof}

\begin{proposition}
\label{cor_queensgraph}
    $n^{(d-1)}$-colouring the $(n,d)$-queen graph is equivalent to finding $n$ superimposable solutions to the $(n,d)$-queens problem.
\end{proposition}

\begin{theorem}[Van Rees \cite{van1981latin}] \label{thm_vanrees_general_d}
    If $\text{gcd}(n,(2^d-1)!) = 1$ then $n^{d-1}$ non-attacking queens can be placed on the $(n,d)$-board, i.e. the $(n,d)$-queens problem has a solution.
\end{theorem}
\newpage
\subsection{Lower Bounds}
\begin{proposition}[trivial lower bound] \label{prop_trivial_upper_bd}
    For all $n,d$, $k \geq 1$
    \begin{align}
    |Q_{max}(n+k,d)| \geq |Q_{max}(n,d)| \nonumber \\
    |Q_{max}(n,d+k)| \geq |Q_{max}(n,d)|
    \end{align}
\end{proposition}

\subsubsection{Subscube Heuristic}
Recall the subcube heuristic from theorem \ref{thm_subcubeheur_3d}. For $d=4$, it appears that we are looking for a $k^4$-subhypercube containing as many queens as possible, contrary to $d=3$, for which we wanted to minimize the number of queens in the $k^3$-cube. This is also the case for the trivial case $d=2$ as one can verify:\\
\begin{theorem} \label{thm_subcubeheur_2d}
    For any given integer $m$, $n>m$, $n = m+k$, we can construct a partial solution of $n - 2k + p$ queens on the $(m,2)$-board under the following condition: There exists a regular solution to the $(n,2)$-queens problem that contains a subset of size $k^2$ with $p$ queens. 
\end{theorem}
\begin{proof}
Analog to theorem \ref{thm_subcubeheur_3d}. 
We obtain $Q_s(m,2)$, a partial solution for the $(m,2)$-board by removing $k$ layers in both dimensions of the solution. The $k$ layers in both dimension have to be outer layers of the cube or adjacent layers, as we can shift the solution wlog. Let us call the set of queens in the removed $x$-layers $L_x$ and the set of queens in the removed $y$-layers $L_y$.
\begin{align}
    |L_x| = |L_y| = k 
\end{align}
\begin{align}
    |Q_s(m,2)| = |Q_{max}(n,2)| - |L_x \cup L_y |
\end{align}
\begin{align}
    I_{xy} := (L_x \cap L_y)  \nonumber\\
    I_{x} := I_{xy} \setminus L_y \nonumber\\
    I_{y} := I_{xy} \setminus L_x 
\end{align}
\begin{align}
    |Q_s(m,2)| &= |Q_{max}(n,2)| - |L_x \cup L_y | \nonumber\\
    &= n - \big( |I_{xy}| + |I_{x}| + |I_{y}| \big)  \nonumber\\
    &= n - \big( p + (k-p) + (k-p) \big) \nonumber\\
    &= n - 2 \cdot k + p
\end{align}
\end{proof}

This result is apparent without proof from just looking at the $(n,2)$-board, and indeed the $p$ term is positive, meaning that we want to maximize the amount of queens in the subset that is the intersection of all cut layers. We believe that this is the case for all even $d$, while for odd $d$ we want that intersection to contain as few queens as possible. A generalization of theorem \ref{thm_subcubeheur_3d} for all $d$ remains to be shown. For $d=4$ we get\\
\begin{theorem} \label{thm_subcubeheur_4d}
    For any given integer $m$, $n>m$, $n = m+k$, we can construct a partial solution of $n^3 - 4kn^2 + 6k^2n + 8k^2 -12k^3 + p $ queens on the $(m,4)$-board under the following condition: There exists a regular solution to the $(n,4)$-queens problem that contains a subset of size $k^4$ with $p$ queens. 
\end{theorem}
\begin{proof}
    Analog to theorem \ref{thm_subcubeheur_3d}, by construction.\\
\end{proof}

\begin{conjecture} \label{conj_subcube_thm}
For all $d$ and given integer $m$, $n>m$, $n = m+k$, we can construct a partial solution of 
\begin{align}
    n^{d-1}-dkn^{d-2}+\frac{(d-1)(d-2)}{2}k^2n^{d-3}+ \dots + (-1)^{d-1} a \cdot p
\end{align}
queens on the $(m,d)$-board under the following condition: There exists a regular solution to the $(n,d)$-queens problem that contains a subset of size $k^d$ with $p$ queens.\\
\end{conjecture}
The last term of integer $a$ is negative for even $d$ and positve for odd $d$.\\

\begin{conjecture} \label{conj_subcube_heur_gend}
Consider the lower bound for $R(n,d)$ (see corollary \ref{cor_alpha_limit_3d}) 
\begin{align*}
    \alpha_{n_0,d} \leq R(n,d) = \frac{|Q_{max}(n,d)|}{n^{d-1}} \text{ , } \forall n \geq n_0
\end{align*}
Then for all $d$
\begin{align*}
    \lim_{n_0 \to \infty} \alpha_{n_0,d} = 1
\end{align*}
\end{conjecture}
Conjecture \ref{conj_subcube_heur_gend} would follow from the subhypercube theorem, i.e. conjecture \ref{conj_subcube_thm}.

\newpage
\subsection{Upper Bounds} \label{subsec_upperbounds}
\begin{proposition}[trivial upper bound] \label{prop_trivial_upper_bd}
    For all $n,d$, $k \geq 1$
    \begin{align}
    |Q_{max}(n,d)|   &\leq n^{d-1} \nonumber \\
    |Q_{max}(n,d+k)| &\leq |Q_{max}(n,d)| \cdot n^k
    \end{align}
\end{proposition}

\subsubsection{Subsolution Inequalities} \label{subsubsec_subsolution_ineq_gend}
For any subset of the $(n,d)$-board, we may apply our knowledge of maximal partial solutions for $(m,d)$-boards, $m < n$. Let $S^{n,d}_m$ denote the set of all subsets of the $(n,d)$-board that correspond to $(m,d)$-boards, i.e., $m^d$ hypercubes. We obtain the following additional valid inequalities:
\begin{align}
     \sum_{s\in S} x_s \leq  |Q_{max}(m,d)| \;\;\;\;\; \forall S \in S^{n,d}_m
\end{align}
There are $|S^{n,d}_m| = (n-m+1)^d$ such subsets for each $m$. 
Implementation of this approach needs to be done iteratively, as it requires knowledge (or tight upper bounds) on maximal partial solutions for $m < n$.

\begin{figure}[H]
    \centering
    \includegraphics[width=0.5\linewidth]{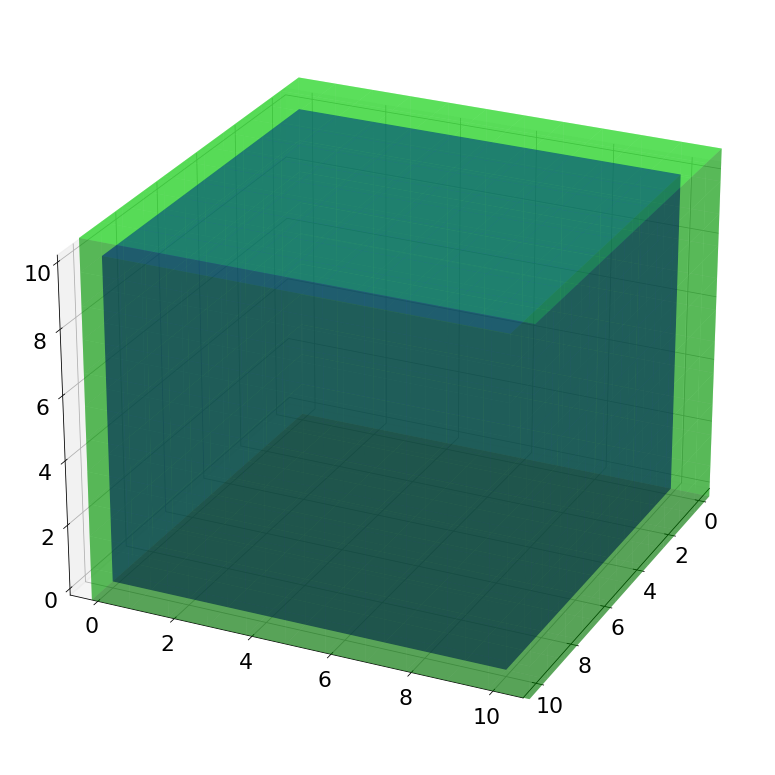}
    \caption{One of the 8 $S^{11,3}_{10}$ subsolutions}
    \label{fig:n21_d3_405queens}
\end{figure}

These inequalities are impactful for smaller $n$. They may be used to derive upper bounds for smaller $n$:\\
\begin{theorem} \label{thm_upperbound_general}
For $n,m$ with $n=k \cdot m$
    \begin{align}
        |Q_{max}(n,d)| \leq |Q_{max}(m,d)| \cdot k^d
    \end{align}
\end{theorem}
\begin{proof}
    The $(n,d)$-board can be considered as disjoint union of $k^d$ many $(m,d)$-boards. 
\end{proof}
\medskip

\begin{corollary}
    $|Q_{max}(2k,d)| \leq  k^{d} $.
\end{corollary}
\begin{proof}
    $|Q_{max}(2k,d)| \leq |Q_{max}(2,d)| \cdot k^d = k^{d} $  by corollary \ref{cor_n2_qmax}.
\end{proof}

\newpage
Regarding the observation from the last corollary, see also the special case for $k=2$ as discussed by \cite{langlois2022complexity}
\begin{conjecture}[Langlois-R{\'e}millard, M{\"u}{\ss}ig, R{\'o}ldan \cite{langlois2022complexity}]
    $|Q_{max}(4,d)| = 2^d$ for $d \geq 4$.
\end{conjecture}

It remains to be shown if the bound is tight for all $d \geq 4$.

The following table lists new upper bounds derived from theorem \ref{thm_upperbound_general} for unsolved instances for $n=4,6,8$ as a function of $d$. The size of maximal partial solutions is given in brackets, if known.\\

\begin{table}[H]
    \centering
    \begin{tabular}{c|cccccccc}  \hline 
         d/n&  &    &  4&  & 6  & &8&9\\ \hline 
         2&  &  &   4 \textit{(4)}&  & 9 \textit{(6)} & &16 \textit{(8)} &18 \textit{(9)}\\
         3&  &  &  8 \textit{(7)}&  & 27 \textit{(21)} & &56 \textit{(48)} &108 \textit{(67)}\\
         4&  &  &  16 \textit{(16)}&  & 81 \textit{(80)} & & 256 &486  \\
         5&  &  &  32 \textit{(32)}&  & 243 & &1024 &2673\\
         6&  &  &  64 \textit{(64)}&  & 729 & &4096 &13851\\
         7&  &  &  128 \textit{(128)}&  & 2187 & &16384 &69984\\ 
         8& &   & 256& &6561 & &65536 &341172\\ 
         ... &  && ... && ...&&...&...\\
         d& & & $2^d$& &$3^d$ & &$4^d$& $|Q_{max}(3,d)| \cdot 3^d $\\  \hline 
    \end{tabular}
    \caption{Upper bounds for $|Q_{max}(n,d)|$}
    \label{tab:upperbounds_general_d}
\end{table}

Applying theorem \ref{thm_upperbound_general} for unsolved instances at smaller $n$ and increasing $d$ yields upper bounds significantly lower than the trivial upper bound $n^{d-1}$.
\begin{corollary}
    For all $n$ with $n < 2^d$ it follows that $|Q_{max}(n,d)|<n^{(d-1)}$
\end{corollary}
\begin{proof}
Compare the size of full solutions $n^{d-1}$ to the upper bound 
    \begin{align}
    \frac{(2k)^{d-1}}{k^{d}} = \frac{2^{d-1}}{k} = \frac{2^{d}}{n}
\end{align}
\end{proof}

\newpage
\subsection{Enumeration of Solutions}
\begin{table}[H]
\begin{center}
\noindent
\begin{tabularx}{0.6\textwidth}{*{5}{r}}
    \hline 
    $n$ & $\mathcal{Q}(n,2)$ & $\mathcal{Q}(n,3)$ & $\mathcal{Q}(n,4)$& $\mathcal{Q}(n,5)$ \\ \hline%\hline
    1   & 1     & 1   & 1& 1  \\ %\hline
    2   & \textit{4}     & \textit{8}   & \textit{16} & \textit{32} \\% \hline
    3   & \textit{8}     & \textit{16}& \textit{4992} & \textit{71154} \\% \hline
    4   & 2     & \textit{1344}& \textit{404564}&\\% \hline
    5   & 10    & \textit{1056} & & \\% \hline
    6   & 4     & \textit{912}  & & \\% \hline
    7   & 40    & \textit{96} & & \\% \hline
    8   & 92    & & & \\% \hline
    9   & 352   & & & \\% \hline
    10  & 724   & & & \\% \hline
    11  & 2680  & 264 & & \\% \hline
    12  & 14200 & & & \\% \hline
    13  & 73712 & 624 & & \\% \hline
    14  & 365596& & &\\ \hline
\end{tabularx}
\caption{\label{table_enumeration_d3}Known sequence of $\mathcal{Q}(n,d)$ and enumeration of maximal partial solutions with $|Q(n,d)| < n^{d-1}$ (cursive). Solutions were computed using the methods described in section \ref{sec_comp_results}.}
\end{center}
\end{table}

\begin{proposition} \label{cor_2d}
     $\mathcal{Q}(2,d) = 2^d$ 
\end{proposition}
\begin{proof}
    Follows from corollary \ref{cor_n2_qmax}.
\end{proof}

\subsubsection{Lower Bounds}
A lower bound for all $n,d$ with $gcd(n,(2^d-1)!)=1$ can be formulated given $c(n)$ (see \ref{subsubsec_lowerbounds_3d}) for general $d$, i.e. through a function $c(n,d)$.

\subsubsection{Upper Bounds} \label{subsubsec_upperbounds_generald}
The recently described enumeration of higher dimensional permutations \cite{keevash2018existence} gives a first weak upper bound for $\mathcal{Q}(n,3)$, as it corresponds to the number of solutions to the $(n,d)$-rooks problem (c.f. \cite{OEISA002860}).
\begin{theorem}[Keevash \cite{keevash2018existence}]
    The number of $d$-dimensional permutations of order $n$ is $(n/e^d + o(n))^{n^d}$.
\end{theorem}
\begin{corollary}
    $\mathcal{Q}(n,d) < (n/e^3 + o(n))^{n^d}$.
\end{corollary}
\begin{proof}
    By proposition \ref{prop_rooks_contained_in_queens}; the set of solutions to the $(n,d$)-queens problem is contained in the set of solutions to the $(n,d$)-rooks problem.
\end{proof}

Such counting results for combinatorial designs yield bounds for further variants or generalizations of the $(n,d)$-queens problem.

Recall theorem \ref{thm_simkin_nobel} and that for all $n$ there exists $\alpha$ such that
\begin{align}
    \mathcal{Q}(n,2) =  \left((1+o(1))n \cdot \frac{n}{e^\alpha} \right) ^n \nonumber
\end{align}
which yields a first tighter bound as
\begin{align}
    \mathcal{Q}(n,d) <  \left((1+o(1))n \cdot \frac{n}{e^\alpha} \right) ^{n^{d-1}}
\end{align}

\newpage
\subsection{Related Problems}
\subsubsection{$(n,d)$-Queens Completion}
Recall the $(n,2)$-queens completion problem concerned with whether a non-attacking set of queens can be completed to a set that is a solution to the $(n,2)$-queens problem. We generalize the following definition by \cite{glock2022n}:
\begin{definition}[$(n,d)$-queens completion threshold]
    Define $qc(n,d)$ as the maximum integer with the property that any partial $(n,d)$-queens configuration $Q$ of size at most $|Q| = qc(n,d)$ can be completed to a maximal partial configuration. We call $qc(n,d)$ the $(n,d)$-queens completion threshold.
\end{definition}
\begin{theorem}[Glock, Munh\'{a} Correia, Sudakov \cite{glock2022n}] \label{thm_completion_bound}
    For all sufficiently large $n$, we have $n/60 \leq qc(n,2) \leq n/4$.
\end{theorem}
\begin{theorem}[Gent \cite{gent2017complexity}]
    The $(n,2)$-queens completion problem is NP-complete and \#P-complete.
\end{theorem}

A first observation regarding $(n,3)$-queens completion is that there exist squares for $n=3,5,6,7$ that may not be occupied, i.e., there exists no maximal partial solution placing a queen on that square (see section \ref{subsec:comp_density}). For increasing $d$, we expect this phenomenon to extend to higher $n$ while also noting that it likely vanishes for sufficiently large $n$.

Second, for $n$, for which only regular solutions exist, we can construct sets of just three queens that cannot be completed to a maximal solution. This is due to their linearity, i.e., the placement of all solutions described by piecewise linear functions.\\

\begin{table}[H]
\begin{center}
\noindent
\begin{tabularx}{0.45\textwidth}{*{4}{r}}
    \hline 
    $n$ & $qc(n,3)$ & $qc(n,4)$ & $qc(n,5)$   \\ \hline%\hline
    1 & 1 & 1 & 1  \\%\hline
    2 & 1 & 1 & 1 \\%\hline
    3 & 0 & 0 & 0\\%\hline
    4 & 1 & 0  &  \\% \hline
    5 & 0 &   & \\% \hline
    6 & 0 &   & \\% \hline
    7 & 0 &   & \\% \hline
    ...&  &   &  \\% \hline
    11 & 2&   & \\% \hline
    12 & -&   & \\% \hline
    13 & $\leq 2$&   & \\ \hline
\end{tabularx}
\caption{\label{table:completion_3d}$qc(n,d)$}
\end{center}
\end{table}

\begin{proposition}
 $qc(n,d) = 1$ for all $n \leq 2$
\end{proposition}
\begin{proof}
    Follows directly from corollary \ref{cor_n2_qmax}.
\end{proof}
\begin{proposition}
 $qc(3,d) = 0$ 
\end{proposition}
\begin{proof}
    Follows from corollary \ref{cor_max_attacking_squares}. Let $Q^* = \{ (2,2,...,2) \}$, then the entire board is dominated. It remains to be shown that $|Q_{max}(3,d)|>1$. To prove this place one queen on $(1,1,...,1,1,1)$ and observe that is does not attack $(1,1,...1,2,3)$, so there exists a partial solution of size $2$.
\end{proof}
\newpage
\subsection{Open Questions}

\begin{question}[]
    Does there exists $n_d$, such that for all $n \geq n_d$ the $(n,d)$-queen graph is $n^{d-1}$-colourable?
\end{question} 
Again, the existence of an $n$-colouring of the $(n,d)$-queen graph is necessary for the existence of a solution to the $(n,d)$-queens problem. \\

\begin{question}[Van Rees] \label{question_d_vanress}
    Is $\text{gcd}(n,(2^d-1)!)=1$ a necessary condition for the $(n,d)$-queens problem to have a solution?
\end{question}

We know that there exist sufficiently many superimposable solutions for $n$ with $\text{gcd}(n,(2^d-1)!)>1$ by theorem \ref{thm_iyer} and theorem \ref{thm_vasquez}. However, there is no known certificate for such $n$.\\

\begin{question}[]
    Are there non-regular solutions to the $(n,d)$-queens problem?
\end{question}
All known certificates are regular; a certificate as described in question \ref{question_d_vanress} (if it exists) would likely be non-regular.\\

\begin{question}
Does there exist an equivalent to the $n$-queens constant (see Theorem \ref{thm_simkin_nobel}) for given $d$ for all $n$ with $gcd(n,(2^d-1)!)=1$. 
\end{question}
 
This connects to whether non-regular solutions for sufficiently large $n$ exist and the enumeration of such solutions. If such constants exist for given $d$, one may ask if there is a relation between them and the $n$-queens constant and how it may be described.\\

\begin{question}[]
    Does the density of the solutions to the partial $(n,d)$-queens problem converge for given $d$ and $n \to \infty$? 
\end{question}
First, we may note conjecture \ref{conj_subcube_heur_gend} is necessary for the distribution to converge. Second, we have observed a uniform distribution for $n$ for which only regular solutions exist and a (empirical) distributions reminiscent of Fig. \ref{fig:dens_simkin} for all other $n$. It is possible that for $n \to \infty$ and provided non-regular solutions to the $(n,d)$-queens problem exist, the distribution in question is shaped as hinted by the empirical distributions for maximal partial solutions and that it can be described for general $d$ in closed form.\\
\begin{question}
    Does there exist a completion bound for $d > 2$ (c.f. theorem \ref{thm_completion_bound})? In other words, does there exist $c_d$, such that for sufficiently large $n$
    \begin{align}
        \frac{n^{d-1}}{c_d} \leq qc(n,d) 
    \end{align}
\end{question}
For a lower bound to $qc(n,d)$ as a function of $n$, the existence of non-regular solutions is necessary.

\newpage
\section{Integer Programming Formulation} \label{sectionip_formulation}
\subsection{Base Model}
This integer programming formulation considers the $\frac{(3^d-1)}{2}$ hyperplanes of attack of a queen on the $(n,d)$-board. It restricts the sum over all squares in each of those lines to be less or equal to $1$, as two or more queens in one such line would attack each other. Thus, the proposed inequalities are sufficient to describe non-attacking configurations.

Recall the hyperplanes of attack identified by the vectors $\epsilon = (\epsilon_1, ..., \epsilon_d)$. For each pair of nonzero $\epsilon$, $\epsilon^{'}$ with $\epsilon = -\epsilon^{'}$, we require one set of inequalities summing over the squares of the board in the direction of $\epsilon, \epsilon^{'}$.

\subsubsection{The Partial $(n,d)$-Queens Problem}
Denote $S$ the set of all squares on the $(n,d)$-board and $H$ the set of all $\frac{(3^d-1)}{2}$ vectors of attack.
Let $L_h$ be the superset of all sets of squares that lie in one of the hyperplanes corresponding to vector $h \in H$. Then the partial $(n,d)$-queens problem can be formulated as follows
\begin{alignat*}{3}
  & \text{max } &       & \sum_{s \in S}^{n} x_s \\[2ex]
  & \text{s.t } & \quad &\; \sum_{s \in L }^{n} x_s  \leq 1,  && \quad \forall L \in L_{h}, \quad  \forall h \in H 
\end{alignat*}
In the following, we will write out the formulation in detail for the classical case $d=2$ and $d=3$. 

\subsubsection{The $(n,2)$-Queens Problem}
$x_{i,j}\in \{\,0, 1 \,\}$ denotes a queen at position $(i,j)$. For $n \geq 4$:\\
\begin{alignat*}{3}
  & \text{max } &       & \sum_{i=1, j=1}^{n} x_{i,j} \\[2ex]
  & \text{s.t } & \quad &\; \sum_{i=1}^{n} x_{i,j} = 1,  && \forall j =1, \dots, n\\
  &             &       &\; \sum_{j=1}^{n} x_{i,j} = 1,  && \forall i =1, \dots, n\\
  &             &       & \sum_{i+j=k}^{n} x_{i,j} \leq 1,&& \forall i,j=1, \dots, n \text{ and } 2 \leq k \leq 2n\\
  &             &       & \sum_{i-j=k}^{n} x_{i,j} \leq 1,&& \forall i,j=1, \dots, n \text{ and } -n+1 \leq k \leq n-1\\[2ex]
  &             &       & x_{i,j} \in \{0,1\},            &\quad \quad \quad& \forall i,j =1 ,\dots, n
\end{alignat*}

The first two equalities constrain rows and columns of the board; they correspond to the vectors $(1,0)$ and $(0,1)$. The third and fourth inequality correspond to the vectors $(1,-1)$  and $(1,1)$, the two diagonal hyperplanes of attack on the $(n,2)$-board.

It is important to note that the first two constraints are, in fact, equalities, as we know that solutions to the $(n,2)$-queens problem exist for all $n \geq 4$. 

\subsubsection{The Partial $(n,3)$-Queens Problem}
$x_{i,j,k}\in \{\,0, 1 \,\}$ denotes a queen at position $(i,j,k)$. For all $n$:\\
\begin{alignat*}{3}
  & \text{max } &       & \sum_{i=1, j=1, k=1}^{n} x_{i,j,k} \\[2ex]
  & \text{s.t } & \quad &\;\;\;\;\;\: \sum_{i=1}^{n} x_{i,j,k} \leq 1,  && \forall j,k =1, \dots, n\\
  &             &       &\;\;\;\;\;\: \sum_{j=1}^{n} x_{i,j,k} \leq 1,  && \forall i,k =1, \dots, n\\
  &             &       &\;\;\;\;\;\: \sum_{k=1}^{n} x_{i,j,k} \leq 1,  && \forall i,j =1, \dots, n\\
  &             &       &\;\;\;\; \sum_{i+j=m}^{n} x_{i,j,k} \leq 1,&& \forall i,j,k=1, \dots, n \text{ and } 2 \leq m \leq 2n\\
  &             &       &\;\;\;\; \sum_{i-j=m}^{n} x_{i,j,k} \leq 1,&& \forall i,j,k=1, \dots, n \text{ and } -n+1 \leq m \leq n-1\\
  &             &       &\;\;\;\; \sum_{i+k=m}^{n} x_{i,j,k} \leq 1,&& \forall i,j,k=1, \dots, n \text{ and } 2 \leq m \leq 2n\\
  &             &       &\;\;\;\; \sum_{i-k=m}^{n} x_{i,j,k} \leq 1,&& \forall i,j,k=1, \dots, n \text{ and } -n+1 \leq m \leq n-1\\
  &             &       &\;\;\;\; \sum_{j+k=m}^{n} x_{i,j,k} \leq 1,&& \forall i,j,k=1, \dots, n \text{ and } 2 \leq m \leq 2n\\
  &             &       &\;\;\;\; \sum_{j-k=m}^{n} x_{i,j,k} \leq 1,&& \forall i,j,k=1, \dots, n \text{ and } -n+1 \leq m \leq n-1\\
  &             &       & \sum_{i+j=m, j+k=p}^{n} x_{i,j,k} \leq 1,&& \text{* and } 2 \leq m \leq 2n, 2 \leq p \leq 2n\\
  &             &       & \sum_{i-j=m, j+k=p}^{n} x_{i,j,k} \leq 1,&& \text{* and } -n+1 \leq m \leq n-1, \leq p \leq 2n\\
  &             &       & \sum_{j+k=m, j-k=p}^{n} x_{i,j,k} \leq 1,&& \text{* and } 2 \leq m \leq 2n, -n+1 \leq p \leq n-1\\
  &             &       & \sum_{j-k=m, j-k=p}^{n} x_{i,j,k} \leq 1,&& \text{* and } -n+1 \leq m \leq n-1, -n+1 \leq p \leq n-1\\[2ex]
  &             &       &                                 &\quad \quad \quad& \text{*} \forall i,j,k, n \\
  &             &       & x_{i,j,k} \in \{0,1\},            &\quad \quad \quad& \forall i,j,k =1 ,\dots, n 
\end{alignat*}
\newpage
The first three inequalities correspond to the vectors $(1,0,0), (0,1,0), (0,0,1)$. Inequalities 4 to 9 correspond to all vectors with two non-zero entries, the diagonals. Inequalities 10 to 12 correspond to all vectors with three non-zero entries, called triagonals. It is intuitive to see how this concept scales for larger $d$, as the vectors of attack are characterized by their number of non-zero entries.

Note that in the particular case of the $(n,3)$-queens problem with $|Q_{max}(m,3)|=n^2$, the first three inequalities are equalities, as we fit exactly one queen in each row. This strengthens the IP formulation, if applicable.

\newpage
\subsection{Strengthening the IP}
We present a selection of valid inequalities for the partial $(n,d)$-queens problem. 
The inequalities strengthen the IP by restricting the LP relaxation and may be used to derive upper bounds. The trivial upper bound for the partial $(n,d)$-queens problem of $n^{d-1}$ is also the solution to the LP relaxation, for example, by placing $\frac{1}{n}$ queens on each square. The proposed valid inequalities significantly improve the LP and, thereby, the solving process.

Note that while well-defined for all $d$, some of the following inequalities improve the upper bound of maximal partial configurations, particularly for $d>2$.

\subsubsection{Subsolution Inequalities}
Recall from \ref{subsubsec_subsolution_ineq_gend}: $S^{n,d}_m$ denote the set of all subsets of the $(n,d)$-board that correspond to $(m,d)$-boards, then
\begin{align}
     \sum_{s\in S} x_s \leq  |Q_{max}(m,d)| \;\;\;\;\; \forall S \in S^{n,d}_m
\end{align}
We may directly translate these inequalities into additional constraints, obtaining $\sum_{m=1}^{n-1} (n-m+1)^d$ inequalities for given $n,d$. 

\subsubsection{Layer Inequalities}
For all layer subsets of the $(n,d)$-board, we may apply our knowledge of maximal partial solutions for $(n,d-1)$-boards. Let $L^{n,d}$ denote the set of all layers of the $(n,d)$-board, meaning all subsets that correspond to $(n,d-1)$-boards (recall \ref{subsubsection_board}). We obtain the following additional valid inequalities:
\begin{align}
     \sum_{s\in L} x_s \leq  |Q_{max}(n,d-1)| \;\;\;\;\; \forall L \in L^{n,d}
\end{align}
These inequalities may be applied recursively for layers of layers. For layers of dimension $2$ or less, they can be discarded due to theorem \ref{thm_pauls}.

\newpage
\subsubsection{Cube and Star Cliques}
\cite{fischetti2019finding} discuss a set of clique inequalities for the $(n,2)$-queens problem. These generalize nicely to higher dimensions. For $d=3$ and for integer $h$ odd and $i+h, j+h, k+h \leq n$, we get a clique in the shape of the corners of a cube.
\begin{align}
    & x_{i,j,k} + x_{i+h,j,k} + x_{i,j+h,k} + x_{i+h,j+h,k}\nonumber \\ + & x_{i,j,k+h} + x_{i+h,j,k+h} + x_{i,j+h,k+h} + x_{i+h,j+h,k+h} \leq 1
\end{align}
For even $h$ we further get
\begin{align}
    & x_{i,j,k} + x_{i+h,j,k} + x_{i,j+h,k} + x_{i+h,j+h,k} \nonumber \\ + & x_{i,j,k+h} + x_{i+h,j,k+h} + x_{i,j+h,k+h} + x_{i+h,j+h,k+h} \nonumber \\
    + & x_{i+\frac{h}{2},j+\frac{h}{2},k+\frac{h}{2}} \leq 1
\end{align}
\\
And for $i+h, j+h, k+h \leq n$ and $i-h, j-h, k-h \geq 1$, a third family of star-shaped cliques (corresponding to cliques in the shape of a cross-polytope in the respective dimension) is given by
\begin{align}
    & x_{i,j,k} + x_{i+h,j,k} + x_{i-h,j,k} + x_{i,j+h,k} + x_{i,j-h,k} \nonumber \\ + & x_{i,j,k+h} + x_{i,j,k-h} \leq 1
\end{align}

\begin{figure}[H]
\setkeys{Gin}{width=0.5\linewidth}
\includegraphics{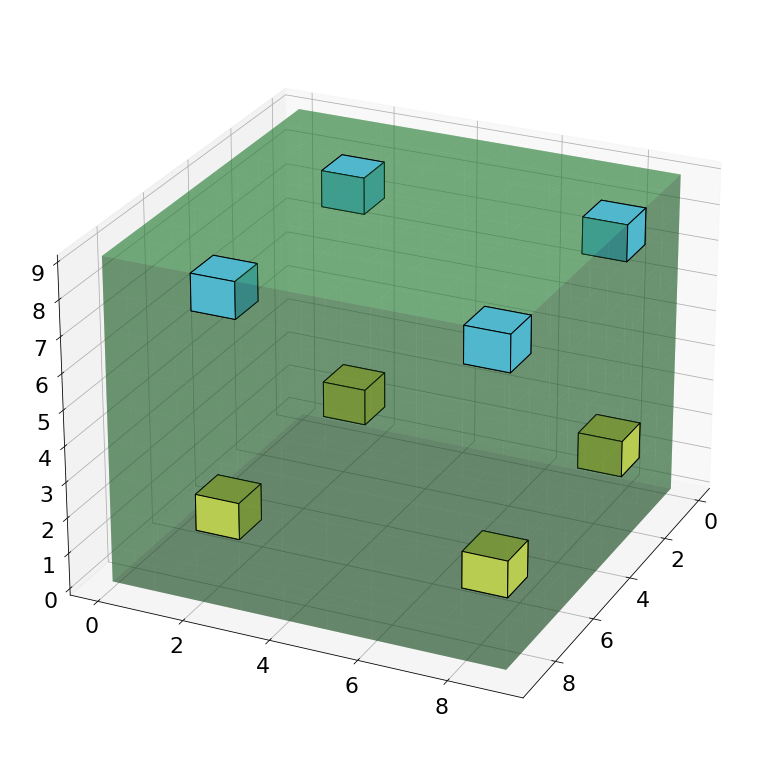}
\hfill
\includegraphics{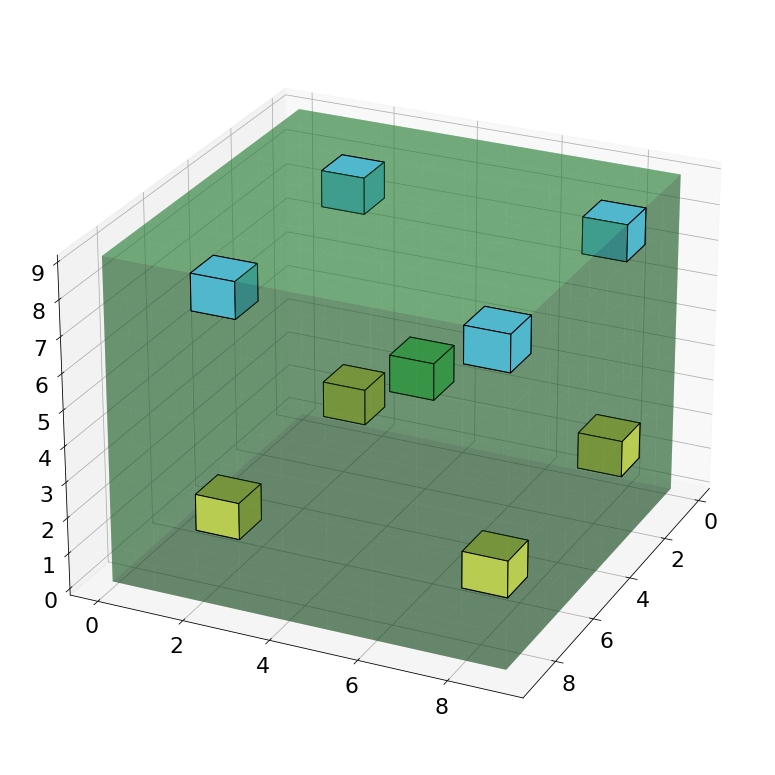}
\caption{Cube cliques for $h=6$ on the $(9,3)$-board}
\label{tab:vis_clique_12}
\end{figure}

\begin{figure}[H]
\begin{center}
    \setkeys{Gin}{width=0.5\linewidth}
\includegraphics{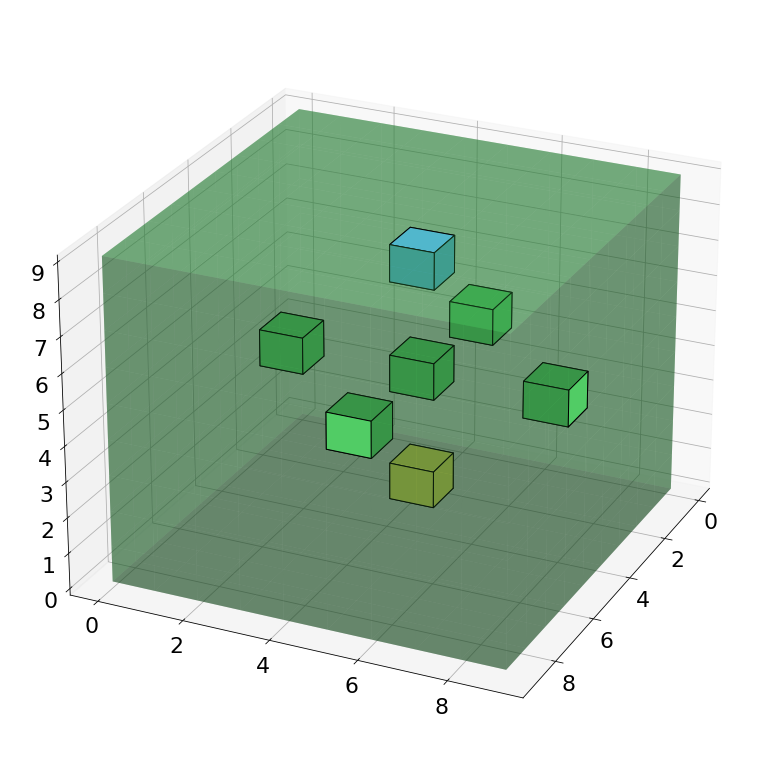}
\caption{Star clique for $h=3$ on the $(9,3)$-board}
\label{tab:vis_clique_3}
\end{center}
\end{figure}

For general $d$, we may describe the cube inequalities as follows. For integer $h$ odd and $s_i+h \leq n$ for all $i \leq d$
\begin{align}
    & \bigg( \sum_{a_1 =0}^1 \sum_{a_2 =0}^1 \dots \sum_{a_d =0}^1 x_{(s_1+a_1 \cdot h,s_2+a_2 \cdot h,...,s_d+a_d \cdot h)} \bigg) \leq 1
\end{align}
And further for even $h$ 
\begin{align}
    & \Bigg( \bigg( \sum_{a_1 =0}^1 \sum_{a_2 =0}^1 \dots \sum_{a_d =0}^1 x_{(s_1+a_1 \cdot h,s_2+a_2 \cdot h,...,s_d+a_d \cdot h)} \bigg) \\
    & + x_{(s_1+\cdot \frac{h}{2},s_2+ \cdot \frac{h}{2},...,s_d+ \cdot \frac{h}{2})} \Bigg) \leq 1
\end{align}
For the star cliques we get for for $s_i+h \leq n$ and $s_i-h \geq n$ for all $i \leq d$
\begin{align}
    & x_{(s_1,s_2,...,s_d)} \nonumber \\
    & + x_{(s_1+h,s_2,...,s_d)} + x_{(s_1-h,s_2,...,s_d)} \nonumber \\
    & + x_{(s_1,s_2+h,...,s_d)} + x_{(s_1,s_2-h,...,s_d)} \nonumber \\
    & + ... \nonumber \\
    & + x_{(s_1,...,s_i+h,...,s_d)} + x_{(s_1,...,s_i-h,...,s_d)} \nonumber \\
    & + ... \nonumber \\
    & + x_{(s_1,s_2,...,s_d+h)} + x_{(s_1,s_2,...,s_d-h)} \leq 1
\end{align}
\\
There are $\sum_{m=2}^n (n-m+1)^d$ cube cliques and $\sum_{m=(2k+1)}^n (n-m+1)^d$ star cliques. The family of cliques is polynomial in $n$ and exponential in $d$.
\newpage

\subsubsection{Odd Cycle Inequalities}
Given any cycle of nodes $O$ on the $(n,d)$-queen graph, it holds that:
\begin{align}
    \sum_{s \in O} x_s \leq \frac{|O|}{2}
\end{align}
In particular for odd cycles, $|O|$ odd, we get:
\begin{align}
    \sum_{s \in O} x_s \leq \frac{|O|-1}{2}
\end{align}
Odd cycles of $|O|=3$ are all contained in the previously described clique inequalities. Implementing odd cycle inequalities with the goal of decreasing computational time, it may be best to consider larger cycles of queens that only attack their respective neighbors in the cycle.

\newpage

\section{Computational Results} \label{sec_comp_results}
\subsection{The Partial $(n,3)$-Queens Problem}
All instances were solved using Gurobi version 9.5.2 with 10 randomly generated seeds on a single Intel Xeon Gold 6246 with 24 cores, 48 threads, and 384GB RAM. We compare the following different variants of the IP regarding runtime, gurobi work\footnote{Gurobi work is a deterministic measure that roughly corresponds to one second of computational time on a single thread.} and nodes:
\begin{itemize}
    \item \textbf{Base Model} \\
    This IP formulation only contains the necessary inequalities as laid out in section \ref{sectionip_formulation}.
    \item \textbf{LMR} \\
    The IP formulation by \cite{langlois2022complexity}, see \cite{Langlois2023}. 
    \item \textbf{Cube Cliques} \\
    Base model with the additional cube clique inequalities, that correspond to cliques in the shape of a hypercube in the respective dimension.
    \item \textbf{Star Cliques} \\
    Base model with the additional star clique inequalities, that correspond to cliques in the shape of a cross-polytope in the respective dimension.
    \item \textbf{Cube and Star Cliques} \\
    Base model with the additional cube and star clique inequalities
    \item \textbf{Infeasibility for $|Q_{max}|+1$} \\
    This formulation adds the additional constraint to set the sum over all squares to be $|Q_{max}|+1$. Provided we know a certificate that we believe to be maximal, this allows us to prove that it is indeed a maximal partial solution. Additional variants add the clique as mentioned above inequalities.
    \item \textbf{Warmstarts} \\
    The solver is warmstarted with a certificate of a maximal partial solution. Recall that we may construct candidates for such certificates with the mentioned heuristic methods. Additional variants add the clique as mentioned above inequalities. As all previous variants solve $n=1,2,3$ in the root node and under one second, the comparison to warmstarts is ignored for $n<4$.
\end{itemize}

Runtime is given in seconds, work in gurobi work units. 

Additionally, we list metrics for larger instances that exceeded runtimes feasible for the above setup and were run on a PowerEdge C6520 with two Intel Xeon Gold 6338 with 64 cores, 128 threads, and 1TB RAM in total. 

\newpage
\begin{landscape}
\subsubsection{Base Model} 
\begin{table}[H]
\noindent
\begin{tabular}{|rr|rrr|rrr|rrr|}
    \hline 
    $n$ & $|Q_{max}|$& avg. time & \hphantom{avg. }min & \hphantom{avg. }max & avg. work & \hphantom{avg. }min & \hphantom{avg. }max & avg. nodes & \hphantom{avg. }min & \hphantom{avg. }max\\ \hline%\hline
    1 &  1 & $0.00$ &$0.00$&$0.00$  & 0.00 &$0.00$&$0.00$ &  0.0 & 0 & 0\\ %\hline
    2 &  1 & $0.00$ &$0.00$&$0.00$  & 0.00 &$0.00$&$0.00$ &  0.0 & 0 & 0\\ %\hline
    3 &  4 & $0.00$ &$0.00$&$0.00$  & 0.00 &$0.00$&$0.00$ &  1.0 & 1 & 1\\ %\hline
    4 &  7 & $1.03$ &$0.40$&$1.42$  & 0.05 &$0.03$&$0.07$ &   7.1 & 4 & 12\\ %\hline
    5 & 13 & $3.66$ &$2.51$&$5.42$  & 0.53 &$0.41$&$0.72$ &  728.5 & 630 & 886\\ %\hline
    6 & 21 &$135.71$&$63.11$&$155.28$  & 22.66 &$10.32$&$25.30$ & 41885.7& 26876 & 47010 \\ %\hline
    7 & 32 &$3351.14$&$2830.59$&$3612.67$  & 623.63 &$524.40$&$664.00$ &616683.6&  509878 & 686108 \\ %\hline
    8 & 48 & \hphantom{0000.00}-&$>36000.00$& \hphantom{0000.00}-&\hphantom{0000.00}-&\hphantom{0000.00}-&\hphantom{0000.00}- &\hphantom{0000.00}-& \hphantom{0000.00}- & \hphantom{0000.00}-\\ %\hline
    9 & 67 & \hphantom{0000.00}-&$>36000.00$& \hphantom{0000.00}-&\hphantom{0000.00}-&\hphantom{0000.00}-&\hphantom{0000.00}- &\hphantom{0000.00}-& \hphantom{0000.00}- & \hphantom{0000.00}-\\ %\hline
    10 & 91 & \hphantom{0000.00}-&$>36000.00$& \hphantom{0000.00}-&\hphantom{0000.00}-&\hphantom{0000.00}-&\hphantom{0000.00}- &\hphantom{0000.00}-& \hphantom{0000.00}- & \hphantom{0000.00}-\\ %\hline
    11 & 121 & $894.79$&$41.78$& $1469.94$&$159.29$&$16.96$ &$232.51$& $663.1 $ & $1$&$2590$\\ %\hline
    12 & 133 & \hphantom{0000.00}-&$>36000.00$& \hphantom{0000.00}-&\hphantom{0000.00}-&\hphantom{0000.00}-&\hphantom{0000.00}- &\hphantom{0000.00}-& \hphantom{0000.00}- & \hphantom{0000.00}-\\ %\hline
    13 & 169 & \hphantom{0000.00}-&$>36000.00$& \hphantom{0000.00}-&\hphantom{0000.00}-&\hphantom{0000.00}-&\hphantom{0000.00}- &\hphantom{0000.00}-& \hphantom{0000.00}- & \hphantom{0000.00}-\\ %\hline
    14 & $\geq 172$ & \hphantom{0000.00}-&$>36000.00$& \hphantom{0000.00}-&\hphantom{0000.00}-&\hphantom{0000.00}-&\hphantom{0000.00}- &\hphantom{0000.00}-& \hphantom{0000.00}- & \hphantom{0000.00}-\\ \hline
\end{tabular}
\caption{\label{table-d3-baseIP}Base IP for $(n,3)$}
\end{table}

\subsubsection{LMR}
\begin{table}[H]
\noindent
\begin{tabular}{|rr|rrr|rrr|rrr|}
    \hline 
    $n$ & $|Q_{max}|$& avg. time & \hphantom{avg. }min & \hphantom{avg. }max & avg. work & \hphantom{avg. }min & \hphantom{avg. }max & avg. nodes & \hphantom{avg. }min & \hphantom{avg. }max\\ \hline%\hline
    1 &  1 & $0.00$ &$0.00$&$0.00$  & 0.00 &$0.00$&$0.00$ &  0.0 & 0 & 0\\ %\hline
    2 &  1 & $0.00$ &$0.00$&$0.00$  & 0.00 &$0.00$&$0.00$ &  0.0 & 0 & 0\\ %\hline
    3 &  4 & $0.00$ &$0.00$&$0.00$  & 0.00 &$0.00$&$0.00$ &  1.0 & 1 & 1\\ %\hline
    4 &  7 & $1.20$ &$1.09$&$1.27$  & $0.04$ &$0.04$&$0.04$ &   $10.9$ & $7$ & $13$\\ %\hline
    5 & 13 & $3.25$ &$2.50$&$3.87$  & $0.42$ &$0.34$&$0.51$ &  $767.2$ & $698$ & $822$\\ %\hline
    6 & 21 &$101.12$&$73.44$&$183.16$  & $14.41$ &$10.25$&$28.24$ & $31556.6$& $23033$ & $62215$\\ %\hline
    7 & 32 &$6225.06$&$5529.80$&$7516.79$  & $913.89$ &$843.55$&$1084.02$ &$870128.1$&  $774062$ & $1076920$\\ %\hline
    8 & 48 & \hphantom{0000.00}-&$>36000.00$& \hphantom{0000.00}-&\hphantom{0000.00}-&\hphantom{0000.00}-&\hphantom{0000.00}- &\hphantom{0000.00}-& \hphantom{0000.00}- & \hphantom{0000.00}-\\ %\hline
    9 & 67 & \hphantom{0000.00}-&$>36000.00$& \hphantom{0000.00}-&\hphantom{0000.00}-&\hphantom{0000.00}-&\hphantom{0000.00}- &\hphantom{0000.00}-& \hphantom{0000.00}- & \hphantom{0000.00}-\\ %\hline
    10 & 91 & \hphantom{0000.00}-&$>36000.00$& \hphantom{0000.00}-&\hphantom{0000.00}-&\hphantom{0000.00}-&\hphantom{0000.00}- &\hphantom{0000.00}-& \hphantom{0000.00}- & \hphantom{0000.00}-\\ %\hline
    11 & 121 & $2029.97$&$399.52$& $3819.09$&$377.67$&$108.57$ &$560.04$& $659.3$ & $1$&$2520$\\ %\hline
    12 & 133 & \hphantom{0000.00}-&$>36000.00$& \hphantom{0000.00}-&\hphantom{0000.00}-&\hphantom{0000.00}-&\hphantom{0000.00}- &\hphantom{0000.00}-& \hphantom{0000.00}- & \hphantom{0000.00}-\\ %\hline
    13 & 169 & \hphantom{0000.00}-&$9656.33$& \hphantom{0000.00}-&\hphantom{0000.00}-&$1705.20$&\hphantom{0000.00}- &\hphantom{0000.00}-& $2420$ & \hphantom{0000.00}-\\ %\hline
    14 & $\geq 172$ & \hphantom{0000.00}-&$>36000.00$& \hphantom{0000.00}-&\hphantom{0000.00}-&\hphantom{0000.00}-&\hphantom{0000.00}- &\hphantom{0000.00}-& \hphantom{0000.00}- & \hphantom{0000.00}-\\ \hline
\end{tabular}
\caption{\label{table-d3-LMR}\cite{Langlois2023} for $(n,3)$}
\end{table}

\subsubsection{Cube Cliques} 
\begin{table}[H]
\noindent
\begin{tabular}{|rr|rrr|rrr|rrr|}
    \hline 
    $n$ & $|Q_{max}|$& avg. time & \hphantom{avg. }min & \hphantom{avg. }max & avg. work & \hphantom{avg. }min & \hphantom{avg. }max & avg. nodes & \hphantom{avg. }min & \hphantom{avg. }max\\ \hline%\hline
    1 &  1 & $0.00$ &$0.00$&$0.00$  & 0.00 &$0.00$&$0.00$ &  0.0 & 0 & 0\\ %\hline
    2 &  1 & $0.00$ &$0.00$&$0.00$  & 0.00 &$0.00$&$0.00$ &  0.0 & 0 & 0\\ %\hline
    3 &  4 & $0.00$ &$0.00$&$0.00$  & 0.00 &$0.00$&$0.00$ &  1.0 & 1 & 1\\ %\hline
    4 &  7 & $0.81$ &$0.73$&$1.16$  & 0.02 &$0.02$&$0.05$ &   6.2 & 4 & 11\\ %\hline
    5 & 13 & $1.61$ &$1.37$&$1.88$  & 0.34 &$0.29$&$0.42$ &  37.0 & 25 & 49\\ %\hline
    6 & 21 &$15.44 $&$9.38$&$20.74$  & 2.98 &$3.64$&$2.06$ & 1091.6& 678 & 1688\\ %\hline
    7 & 32 &$574.01 $&$346.55$&$1063.55$  & 138.91 &$77.66$&$308.37$ &62032.8&  27419 & 106960\\ %\hline
    8 & 48 &$11343.43$ &$4532.57$&$20103.97$  & $2029.25$ &$918.91$&$3855.79$ &$519285.6$& $166974$ & $1090404$\\ 
    9 & 67 & \hphantom{0000.00}-&$>36000.00$& \hphantom{0000.00}-&\hphantom{0000.00}-&\hphantom{0000.00}-&\hphantom{0000.00}- &\hphantom{0000.00}-& \hphantom{0000.00}- & \hphantom{0000.00}-\\ %\hline
    10 & 91 & \hphantom{0000.00}-&$>36000.00$& \hphantom{0000.00}-&\hphantom{0000.00}-&\hphantom{0000.00}-&\hphantom{0000.00}- &\hphantom{0000.00}-& \hphantom{0000.00}- & \hphantom{0000.00}-\\ %\hline
    11 & 121 & $ 1164.61$&$158.65$& $2159.88$&$277.87$&$63.24$ &$416.13$& $169.0$ & $1$&$564$\\ %\hline
    12 & 133 & \hphantom{0000.00}-&$>36000.00$& \hphantom{0000.00}-&\hphantom{0000.00}-&\hphantom{0000.00}-&\hphantom{0000.00}- &\hphantom{0000.00}-& \hphantom{0000.00}- & \hphantom{0000.00}-\\ %\hline
    13 & 169 & \hphantom{0000.00}-&$>36000.00$& \hphantom{0000.00}-&\hphantom{0000.00}-&\hphantom{0000.00}-&\hphantom{0000.00}- &\hphantom{0000.00}-& \hphantom{0000.00}- & \hphantom{0000.00}-\\ %\hline
    14 & $\geq 172$ & \hphantom{0000.00}-&$>36000.00$& \hphantom{0000.00}-&\hphantom{0000.00}-&\hphantom{0000.00}-&\hphantom{0000.00}- &\hphantom{0000.00}-& \hphantom{0000.00}- & \hphantom{0000.00}-\\ \hline
\end{tabular}
\caption{\label{table-d3-CubeCliques}IP with cube cliques for $(n,3)$}
\end{table}

\subsubsection{Star Cliques} 
\begin{table}[H]
\noindent
\begin{tabular}{|rr|rrr|rrr|rrr|}
    \hline 
    $n$ & $|Q_{max}|$& avg. time & \hphantom{avg. }min & \hphantom{avg. }max & avg. work & \hphantom{avg. }min & \hphantom{avg. }max & avg. nodes & \hphantom{avg. }min & \hphantom{avg. }max\\ \hline%\hline
    1 &  1 & $0.00$ &$0.00$&$0.00$  & 0.00 &$0.00$&$0.00$ &  0.0 & 0 & 0\\ %\hline
    2 &  1 & $0.00$ &$0.00$&$0.00$  & 0.00 &$0.00$&$0.00$ &  0.0 & 0 & 0\\ %\hline
    3 &  4 & $0.00$ &$0.00$&$0.00$  & 0.00 &$0.00$&$0.00$ &  1.0 & 1 & 1\\ %\hline
    4 &  7 & $0.76$ &$0.67$&$0.85$  & $0.02$ &$0.02$&$0.02$ &   6.6 & 5 & 9\\ %\hline
    5 & 13 & $3.48$ &$2.47$&$6.25$  & $0.50$&$0.36$&$0.84$ &  $675.2$ & $621$ & $750$\\ %\hline
    6 & 21 &$163.24 $&$149.44$&$176.04$  & $26.68$ &$22.78$&$30.11$ & $41836.9$& $37195$ & $46049$\\ %\hline
    7 & 32 &$3714.34$&$3198.91$&$4064.65$  & $667.38$ &$576.75$&$721.11$ &$647698.5$&  $535591$ & $614444$\\ %\hline
    8 & 48 & \hphantom{0000.00}-&$>36000.00$& \hphantom{0000.00}-&\hphantom{0000.00}-&\hphantom{0000.00}-&\hphantom{0000.00}- &\hphantom{0000.00}-& \hphantom{0000.00}- & \hphantom{0000.00}-\\ %\hline
    9 & 67 & \hphantom{0000.00}-&$>36000.00$& \hphantom{0000.00}-&\hphantom{0000.00}-&\hphantom{0000.00}-&\hphantom{0000.00}- &\hphantom{0000.00}-& \hphantom{0000.00}- & \hphantom{0000.00}-\\ %\hline
    10 & 91 & \hphantom{0000.00}-&$>36000.00$& \hphantom{0000.00}-&\hphantom{0000.00}-&\hphantom{0000.00}-&\hphantom{0000.00}- &\hphantom{0000.00}-& \hphantom{0000.00}- & \hphantom{0000.00}-\\ %\hline
    11 & 121 & $1246.48$&$62.50$& $3011.27$&$295.68$&$30.98$ &$614.68$& $800.7$ & $1$&$2401$\\ %\hline
    12 & 133 & \hphantom{0000.00}-&$>36000.00$& \hphantom{0000.00}-&\hphantom{0000.00}-&\hphantom{0000.00}-&\hphantom{0000.00}- &\hphantom{0000.00}-& \hphantom{0000.00}- & \hphantom{0000.00}-\\ %\hline
    13 & 169 & \hphantom{0000.00}-&$>36000.00$& \hphantom{0000.00}-&\hphantom{0000.00}-&\hphantom{0000.00}-&\hphantom{0000.00}- &\hphantom{0000.00}-& \hphantom{0000.00}- & \hphantom{0000.00}-\\ %\hline
    14 & $\geq 172$ & \hphantom{0000.00}-&$>36000.00$& \hphantom{0000.00}-&\hphantom{0000.00}-&\hphantom{0000.00}-&\hphantom{0000.00}- &\hphantom{0000.00}-& \hphantom{0000.00}- & \hphantom{0000.00}-\\ \hline
\end{tabular}
\caption{\label{table-d3-StarCliques}IP with star cliques for $(n,3)$}
\end{table}

\subsubsection{Cube and Star Cliques}
\begin{table}[H]
\noindent
\begin{tabular}{|rr|rrr|rrr|rrr|}
    \hline 
    $n$ & $|Q_{max}|$& avg. time & \hphantom{avg. }min & \hphantom{avg. }max & avg. work & \hphantom{avg. }min & \hphantom{avg. }max & avg. nodes & \hphantom{avg. }min & \hphantom{avg. }max\\ \hline%\hline
    1 &  1 & $0.00$ &$0.00$&$0.00$  & 0.00 &$0.00$&$0.00$ &  0.0 & 0 & 0\\ %\hline
    2 &  1 & $0.00$ &$0.00$&$0.00$  & 0.00 &$0.00$&$0.00$ &  0.0 & 0 & 0\\ %\hline
    3 &  4 & $0.00$ &$0.00$&$0.00$  & 0.00 &$0.00$&$0.00$  & 1.0 & 1 & 1\\ %\hline
    4 &  7 & $0.81$ &$0.78$&$0.85$  & $0.03$ &$0.02$&$0.03$ &   6.2 & 4 & 9\\ %\hline
    5 & 13 & $1.58$ &$1.36$&$1.94$  & $0.35$&$0.25$&$0.45$ &  $30.8$ & $25$ & $36$\\ %\hline
    6 & 21 &$16.43$&$10.41$&$23.48$  & $3.18$ &$2.23$&$4.24$ & $973.2$& $775$ & $1259$\\ %\hline
    7 & 32 &$531.74$&$399.99$&$826.12$  & $125.17$ &$85.50$&$185.09$ &$43436.3$&  $20449$ & $85804$\\ %\hline
    8 & 48 &$11745.36$ &$4192.13$&$18865.56$  & $2124.57$ &$838.47$&$3483.81$ &$449242.0$& $136496$ & $660154$\\ 
    9 & 67 & \hphantom{0000.00}-&$>36000.00$& \hphantom{0000.00}-&\hphantom{0000.00}-&\hphantom{0000.00}-&\hphantom{0000.00}- &\hphantom{0000.00}-& \hphantom{0000.00}- & \hphantom{0000.00}-\\ %\hline
    10 & 91 & \hphantom{0000.00}-&$>36000.00$& \hphantom{0000.00}-&\hphantom{0000.00}-&\hphantom{0000.00}-&\hphantom{0000.00}- &\hphantom{0000.00}-& \hphantom{0000.00}- & \hphantom{0000.00}-\\ %\hline
    11 & 121 & $2001.73$&$318.61$& $4217.57$&$401.27$&$124.33$ &$614.68$& $515.3$ & $1$&$1591$\\ %\hline
    12 & 133 & \hphantom{0000.00}-&$>36000.00$& \hphantom{0000.00}-&\hphantom{0000.00}-&\hphantom{0000.00}-&\hphantom{0000.00}- &\hphantom{0000.00}-& \hphantom{0000.00}- & \hphantom{0000.00}-\\ %\hline
    13 & 169 & \hphantom{0000.00}-&$>36000.00$& \hphantom{0000.00}-&\hphantom{0000.00}-&\hphantom{0000.00}-&\hphantom{0000.00}- &\hphantom{0000.00}-& \hphantom{0000.00}- & \hphantom{0000.00}-\\ %\hline
    14 & $\geq 172$ & \hphantom{0000.00}-&$>36000.00$& \hphantom{0000.00}-&\hphantom{0000.00}-&\hphantom{0000.00}-&\hphantom{0000.00}- &\hphantom{0000.00}-& \hphantom{0000.00}- & \hphantom{0000.00}-\\ \hline
\end{tabular}
\caption{\label{table-d3-CubeStarCliques}IP with cube and star cliques for $(n,3)$}
\end{table}

\subsubsection{Infeasibility for $|Q_{max}|+1$}
\begin{table}[H]
\noindent
\begin{tabular}{|rr|rrr|rrr|rrr|}
    \hline 
    $n$ & $|Q_{max}|$& avg. time & \hphantom{avg. }min & \hphantom{avg. }max & avg. work & \hphantom{avg. }min & \hphantom{avg. }max & avg. nodes & \hphantom{avg. }min & \hphantom{avg. }max\\ \hline%\hline
    1 &  1 & $0.00$ &$0.00$&$0.00$  & 0.00 &$0.00$&$0.00$ &  $0.0$ & $0$ & $0$\\ %\hline
    2 &  1 & $0.00$ &$0.00$&$0.00$  & 0.00 &$0.00$&$0.00$ &  $0.0$ & $0$ & $0$\\ %\hline
    3 &  4 & $0.00$ &$0.00$&$0.00$  & 0.00 &$0.00$&$0.00$ &  $0.0$ & $0$ & $0$\\ %\hline
    4 &  7 & $0.03$ &$0.02$&$0.09$  & $0.01$ &$0.01$&$0.02$ &   $1$ & $1$ & $1$\\ %\hline
    5 & 13 & $2.48$ &$2.35$&$2.59$  & $0.32$&$0.29$&$0.35$ &  $1236.5$ & $1003$ & $1493$\\ %\hline
    6 & 21 &$349.06$&$245.60$&$486.80$  & $41.77$ &$29.52$&$58.22$ & $254049.9$& $148023$ & $402077$\\ %\hline
    7 & 32 & \hphantom{0000.00}-&$>36000.00$& \hphantom{0000.00}-&\hphantom{0000.00}-&\hphantom{0000.00}-&\hphantom{0000.00}- &\hphantom{0000.00}-& \hphantom{0000.00}- & \hphantom{0000.00}-\\ %\hline
    8 & 48 & \hphantom{0000.00}-&$>36000.00$& \hphantom{0000.00}-&\hphantom{0000.00}-&\hphantom{0000.00}-&\hphantom{0000.00}- &\hphantom{0000.00}-& \hphantom{0000.00}- & \hphantom{0000.00}-\\ %\hline
    9 & 67 & \hphantom{0000.00}-&$>36000.00$& \hphantom{0000.00}-&\hphantom{0000.00}-&\hphantom{0000.00}-&\hphantom{0000.00}- &\hphantom{0000.00}-& \hphantom{0000.00}- & \hphantom{0000.00}-\\ %\hline
    10 & 91 & \hphantom{0000.00}-&$>36000.00$& \hphantom{0000.00}-&\hphantom{0000.00}-&\hphantom{0000.00}-&\hphantom{0000.00}- &\hphantom{0000.00}-& \hphantom{0000.00}- & \hphantom{0000.00}-\\ %\hline
    11 & 121 & $0.00$&$0.00$& $0.01$&$0.01$ &$0.01$& $0.01$ &$0.0$& $0$&$0$\\ %\hline
    12 & 133 & \hphantom{0000.00}-&$>36000.00$& \hphantom{0000.00}-&\hphantom{0000.00}-&\hphantom{0000.00}-&\hphantom{0000.00}- &\hphantom{0000.00}-& \hphantom{0000.00}- & \hphantom{0000.00}-\\ %\hline
    13 & 169 & $0.01$&$0.01$& $0.01$&$0.01$ &$0.01$& $0.01$ &$0.0$& $0$&$0$\\ %\hline
    14 & $\geq 172$ & \hphantom{0000.00}-&$>36000.00$& \hphantom{0000.00}-&\hphantom{0000.00}-&\hphantom{0000.00}-&\hphantom{0000.00}- &\hphantom{0000.00}-& \hphantom{0000.00}- & \hphantom{0000.00}-\\ \hline
\end{tabular}
\caption{\label{table-d3-Infeasibility}Proving infeasibility for $k+1$, Base IP for $(n,3)$}
\end{table}

\subsubsection{Infeasibility for $|Q_{max}|+1$, Cube Cliques}
\begin{table}[H]
\noindent
\begin{tabular}{|rr|rrr|rrr|rrr|}
    \hline 
    $n$ & $|Q_{max}|$& avg. time & \hphantom{avg. }min & \hphantom{avg. }max & avg. work & \hphantom{avg. }min & \hphantom{avg. }max & avg. nodes & \hphantom{avg. }min & \hphantom{avg. }max\\ \hline%\hline
    1 &  1 & $0.00$ &$0.00$&$0.00$  & 0.00 &$0.00$&$0.00$ &  0.0 & 0 & 0\\ %\hline
    2 &  1 & $0.00$ &$0.00$&$0.00$  & 0.00 &$0.00$&$0.00$ &  0.0 & 0 & 0\\ %\hline
    3 &  4 & $0.00$ &$0.00$&$0.00$  & 0.00 &$0.00$&$0.00$ &  0.0 & 0 & 0\\ %\hline
    4 &  7 & $0.52$ &$0.02$&$1.01$  & $0.02$ &$0.01$&$0.04$ &   $12.6$ & $1$ & $95$\\ %\hline
    5 & 13 & $1.40$ &$1.21$&$1.73$  & $0.26$&$0.22$&$0.39$ &  $42.6$ & $27$ & $63$\\ %\hline
    6 & 21 &$6.26$&$5.50$&$8.82$  & $1.34$ &$1.21$&$1.62$ & $1051.8$& $687$ & $2233$\\ %\hline
    7 & 32 &$168.31$&$116.05$&$403.87$  & $24.86$ &$17.14$&$60.85$ &$20909.3$&  $14391$ & $45627$\\ %\hline
    8 & 48 &$4559.43$ &$1730.25$&$10204.63$  & $714.30$ &$290.54$&$1548.51$ &$162308.4$& $60534$ & $386932$\\ %\hline
    9 & 67 & \hphantom{0000.00}-&$>36000.00$& \hphantom{0000.00}-&\hphantom{0000.00}-&\hphantom{0000.00}-&\hphantom{0000.00}- &\hphantom{0000.00}-& \hphantom{0000.00}- & \hphantom{0000.00}-\\ %\hline
    10 & 91 & \hphantom{0000.00}-&$>36000.00$& \hphantom{0000.00}-&\hphantom{0000.00}-&\hphantom{0000.00}-&\hphantom{0000.00}- &\hphantom{0000.00}-& \hphantom{0000.00}- & \hphantom{0000.00}-\\ %\hline
    11 & 121 & $0.02$&$0.02$& $0.02$&$0.02$ &$0.02$& $0.02$ &$0.0$& $0$&$0$\\ %\hline
    12 & 133 & \hphantom{0000.00}-&$>36000.00$& \hphantom{0000.00}-&\hphantom{0000.00}-&\hphantom{0000.00}-&\hphantom{0000.00}- &\hphantom{0000.00}-& \hphantom{0000.00}- & \hphantom{0000.00}-\\ %\hline
    13 & 169 & $0.04$&$0.04$& $0.04$&$0.03$&$0.03$& $0.03$ &$0.0$& $0$&$0$\\ %\hline
    14 & $\geq 172$ & \hphantom{0000.00}-&$>36000.00$& \hphantom{0000.00}-&\hphantom{0000.00}-&\hphantom{0000.00}-&\hphantom{0000.00}- &\hphantom{0000.00}-& \hphantom{0000.00}- & \hphantom{0000.00}-\\ \hline
\end{tabular}
\caption{\label{table-d3-Infeasibility-Cube}Proving infeasibility for $k+1$, IP with cube cliques for $(n,3)$}
\end{table}

\subsubsection{Infeasibility for $|Q_{max}|+1$, Cube and Star Cliques}
\begin{table}[H]
\noindent
\begin{tabular}{|rr|rrr|rrr|rrr|}
    \hline 
    $n$ & $|Q_{max}|$& avg. time & \hphantom{avg. }min & \hphantom{avg. }max & avg. work & \hphantom{avg. }min & \hphantom{avg. }max & avg. nodes & \hphantom{avg. }min & \hphantom{avg. }max\\ \hline%\hline
    1 &  1 & $0.00$ &$0.00$&$0.00$  & 0.00 &$0.00$&$0.00$ &  0.0 & 0 & 0\\ %\hline
    2 &  1 & $0.00$ &$0.00$&$0.00$  & 0.00 &$0.00$&$0.00$ &  0.0 & 0 & 0\\ %\hline
    3 &  4 & $0.00$ &$0.00$&$0.00$  & 0.00 &$0.00$&$0.00$ &  0.0 & 0 & 0\\ %\hline
    4 &  7 & $0.43$ &$0.02$&$0.95$  & $0.02$ &$0.01$&$0.03$ &   $2.4$ & $1$ & $6$\\ %\hline
    5 & 13 & $1.42$ &$1.28$&$1.60$  & $0.28$&$0.23$&$0.32$ &  $42.5$ & $35$ & $63$\\ %\hline
    6 & 21 &$6.52$&$5.56$&$10.41$  & $1.44$ &$1.28$&$2.07$ & $1009.9$& $741$ & $2003$\\ %\hline
    7 & 32 &$267.05$&$228.21$&$342.88$  & $43.11$ &$36.05$&$53.36$ &$27297.0$&  $21729$ & $39542$\\ %\hline
    8 & 48 &$4127.32$ &$1892.48$&$5808.65$  & $571.04$ &$317.82$&$760.54$ &$112814.0$& $57872$ & $160581$\\ %\hline
    9 & 67 & \hphantom{0000.00}-&$>36000.00$& \hphantom{0000.00}-&\hphantom{0000.00}-&\hphantom{0000.00}-&\hphantom{0000.00}- &\hphantom{0000.00}-& \hphantom{0000.00}- & \hphantom{0000.00}-\\ %\hline
    10 & 91 & \hphantom{0000.00}-&$>36000.00$& \hphantom{0000.00}-&\hphantom{0000.00}-&\hphantom{0000.00}-&\hphantom{0000.00}- &\hphantom{0000.00}-& \hphantom{0000.00}- & \hphantom{0000.00}-\\ %\hline
    11 & 121 & $0.02$&$0.02$& $0.02$&$0.02$ &$0.02$& $0.02$ &$0.0$& $0$&$0$\\ %\hline
    12 & 133 & \hphantom{0000.00}-&$>36000.00$& \hphantom{0000.00}-&\hphantom{0000.00}-&\hphantom{0000.00}-&\hphantom{0000.00}- &\hphantom{0000.00}-& \hphantom{0000.00}- & \hphantom{0000.00}-\\ %\hline
    13 & 169 & $0.06$&$0.05$& $0.06$&$0.04$ &$0.04$& $0.04$ &$0.0$& $0$&$0$\\ %\hline
    14 & $\geq 172$ & \hphantom{0000.00}-&$>36000.00$& \hphantom{0000.00}-&\hphantom{0000.00}-&\hphantom{0000.00}-&\hphantom{0000.00}- &\hphantom{0000.00}-& \hphantom{0000.00}- & \hphantom{0000.00}-\\ \hline
\end{tabular}
\caption{\label{table-d3-Infeasibility-CubeStar}Proving infeasibility for $k+1$, IP with cube and star cliques for $(n,3)$}
\end{table}

\subsubsection{Warmstarts with Base Model}
\begin{table}[H]
\noindent
\begin{tabular}{|rr|rrr|rrr|rrr|}
    \hline 
    $n$ & $|Q_{max}|$& avg. time & \hphantom{avg. }min & \hphantom{avg. }max & avg. work & \hphantom{avg. }min & \hphantom{avg. }max & avg. nodes & \hphantom{avg. }min & \hphantom{avg. }max\\ \hline%\hline
    4 &  7 &$1.08$ &$0.72$&$1.45$  & $0.05$ &$0.03$&$0.06$ &   $7.0$ & $4$ & $11$\\ %\hline
    5 & 13 &$4.30$ &$2.95$&$6.07$  & $0.56$&$0.41$&$0.78$ &  $759.0$ & $673$ & $871$\\ %\hline
    6 & 21 &$143.08$&$127.75$&$151.12$  & $22.72$ &$19.97$&$24.22$ & $44150.3$& $38559$ & $46817$\\ %\hline
    7 & 32 &$3839.39$&$3092.87$&$4389.22$  & $683.22$ &$571.67$&$7782.13$ &$666381.1$&  $518615$ & $765314$\\ %\hline
    8 & 48 & \hphantom{0000.00}-&$>36000.00$& \hphantom{0000.00}-&\hphantom{0000.00}-&\hphantom{0000.00}-&\hphantom{0000.00}- &\hphantom{0000.00}-& \hphantom{0000.00}- & \hphantom{0000.00}-\\ %\hline
    9 & 67 & \hphantom{0000.00}-&$>36000.00$& \hphantom{0000.00}-&\hphantom{0000.00}-&\hphantom{0000.00}-&\hphantom{0000.00}- &\hphantom{0000.00}-& \hphantom{0000.00}- & \hphantom{0000.00}-\\ %\hline
    10 & 91 & \hphantom{0000.00}-&$>36000.00$& \hphantom{0000.00}-&\hphantom{0000.00}-&\hphantom{0000.00}-&\hphantom{0000.00}- &\hphantom{0000.00}-& \hphantom{0000.00}- & \hphantom{0000.00}-\\ %\hline
    11 & 121 & $1.10$&$1.03$& $1.16$&$1.64$ &$1.55$& $1.73$ &$1.0$& $1$&$1$\\ %\hline
    12 & 133 & \hphantom{0000.00}-&$>36000.00$& \hphantom{0000.00}-&\hphantom{0000.00}-&\hphantom{0000.00}-&\hphantom{0000.00}- &\hphantom{0000.00}-& \hphantom{0000.00}- & \hphantom{0000.00}-\\ %\hline
    13 & 169 & $2.74$&$2.55$& $2.90$&$3.96$ &$3.65$& $4.22$ &$1.0$& $1$&$1$\\ %\hline
    14 & $\geq 172$ & \hphantom{0000.00}-&$>36000.00$& \hphantom{0000.00}-&\hphantom{0000.00}-&\hphantom{0000.00}-&\hphantom{0000.00}- &\hphantom{0000.00}-& \hphantom{0000.00}- & \hphantom{0000.00}-\\ \hline
\end{tabular}
\caption{\label{table-d3-warmstart-baseIP}Base IP with warmstarts for $(n,3)$}
\end{table}

\subsubsection{Warmstarts with Cube and Star Cliques} 
\begin{table}[H]
\noindent
\begin{tabular}{|rr|rrr|rrr|rrr|}
    \hline 
    $n$ & $|Q_{max}|$& avg. time & \hphantom{avg. }min & \hphantom{avg. }max & avg. work & \hphantom{avg. }min & \hphantom{avg. }max & avg. nodes & \hphantom{avg. }min & \hphantom{avg. }max\\ \hline%\hline
    4 &  7 & $0.81$ &$0.75$&$0.87$  & $0.03$ &$0.03$&$0.03$ &   $8.4$ & $7$ & $14$\\ %\hline
    5 & 13 &$1.77$ &$1.49$&$2.19$  & $0.37$&$0.30$&$0.49$ &  $32.4$ & $23$ & $37$\\ %\hline
    6 & 21 &$18.14$&$15.52$&$20.63$  & $3.48$ &$3.04$&$3.89$ & $1177.9$& $765$ & $1700$\\ %\hline
    7 & 32 &$498.39$&$ 401.58$&$603.80$  & $107.35$ &$79.57$&$144.34$ &$36634.9$&  $26126$ & $45966$\\ %\hline
    8 & 48 &$7324.733$ &$4433.61$&$9261.85$  & $1465.36$ &$931.14$&$1844.42$ &$280155.8$& $170989$ & $382164$\\ %\hline
    9 & 67 & \hphantom{0000.00}-&$>36000.00$& \hphantom{0000.00}-&\hphantom{0000.00}-&\hphantom{0000.00}-&\hphantom{0000.00}- &\hphantom{0000.00}-& \hphantom{0000.00}- & \hphantom{0000.00}-\\ %\hline
    10 & 91 & \hphantom{0000.00}-&$>36000.00$& \hphantom{0000.00}-&\hphantom{0000.00}-&\hphantom{0000.00}-&\hphantom{0000.00}- &\hphantom{0000.00}-& \hphantom{0000.00}- & \hphantom{0000.00}-\\ %\hline
    11 & 121 & $6.79$&$6.38$& $7.79$&$10.23$ &$9.58$&$11.658$ &$1.0$& $1$&$1$\\ %\hline
    12 & 133 & \hphantom{0000.00}-&$>36000.00$& \hphantom{0000.00}-&\hphantom{0000.00}-&\hphantom{0000.00}-&\hphantom{0000.00}- &\hphantom{0000.00}-& \hphantom{0000.00}- & \hphantom{0000.00}-\\ %\hline
    13 & 169 & $30.66$&$27.43$& $34.63$&$28.13$ &$26.45$& $29.93$ &$1.0$& $1$&$1$\\ %\hline
    14 & $\geq 172$ & \hphantom{0000.00}-&$>36000.00$& \hphantom{0000.00}-&\hphantom{0000.00}-&\hphantom{0000.00}-&\hphantom{0000.00}- &\hphantom{0000.00}-& \hphantom{0000.00}- & \hphantom{0000.00}-\\ \hline
\end{tabular}
\caption{\label{table-d3-warmstart-CubeStar}IP with cube and star cliques with warmstarts for $(n,3)$}
%\end{center}
\end{table}

\newpage
\subsubsection{Comparison of IP Formulations} 
\begin{table}[H]
\noindent
\begin{tabular}{|rr|rrrrrrrrr|rr|}
    \hline 
    $n$ & $|Q_{max}|$& Base   & LMR     & Cube     & Star    & C+S    & Inf    & Inf C+S & ws & ws C+S & Cube$^*$ & Inf C$^*$\\ \hline%\hline
    1 &  1           &$0.00$  &$0.00$   &$0.00$    &$0.00$   &$0.00$  &$0.00$  &$0.00$   & $-$ & $-$ &$0.00$ &$0.00$ \\ %\hline
    2 &  1           &$0.00$  &$0.00$   &$0.00$    &$0.00$   &$0.00$  &$0.00$  &$0.00$   & $-$ & $-$ &$0.00$ &$0.00$\\ %\hline
    3 &  4           &$0.00$  &$0.00$   &$0.00$    &$0.00$   &$0.00$  &$0.00$  &$0.00$   & $-$ & $-$ &$0.01$ &$0.00$\\ %\hline
    4 &  7           &$1.03$  &$1.20$   &$0.81$    &$0.76$   &$0.81$  &$0.03$  &$0.42$   & $1.08$ & $0.81$ &$0.77$ &$0.01$\\ %\hline
    5 & 13           &$3.66$  &$3.25$   &$1.61$    &$3.48$   &$1.58$  &$2.48$  &$1.42$   & $4.30$ & $1.77$ &$1.56$ &$0.26$\\ %\hline
    6 & 21           &$135.71$&$101.12$ &$15.44$   &$163.24$ &$16.43$ &$349.06$&$\textbf{6.52}$   & $143.08$ & $18.14$ &$10.76$ &$9.79$\\ %\hline
    7 & 32           &$3351.14$&$6225.06$&$574.01$  &$3714.34$&$531.74$&$-$     &$\textbf{267.05}$ & $3839.39$ & $498.39$ &$659.68$ &$175.23$\\ %\hline
    8 & 48           &$-$     &$-$      &$11343.43$&$-$      &$11745.36$  &$-$ &$\textbf{4127.32}$&$-$ & $7324.73$ &$8464.06$ &$2604.80$\\ %\hline
    9 & 67           &$-$     &$-$      &$-$       &$-$      &$-$     &$-$     &$-$      & $-$  &$-$   &$45481.27$ &$31525.97$\\ %\hline
   10 & 91           &$-$     &$-$      &$-$       &$-$      &$-$     &$-$     &$-$      & $-$  &$-$   &$36453.71$ &$87914.31$\\ %\hline
   11 &121           &$894.79$&$2029.97$&$1164.61$ &$1246.48$     &$2001.73$&$\textbf{0.00}$  &$0.02$   &$1.10$&$6.70$&$315.78$ &$0.01$\\ %\hline
   12 &133           &$-$     &$-$      &$-$       &$-$      &$-$     &$-$     &$-$      & $-$  &$-$   &$-$ &$-$\\ %\hline
   13 &169           &$-$     &$-$      &$-$       &$-$      &$-$     &$\textbf{0.01}$  &$0.06$   &$2.74$&$30.66$&$-$ &$0.05$\\ %\hline
   14 &$\geq172$     &$-$     &$-$      &$-$       &$-$      &$-$     &$-$     &$-$      & $-$  &$-$   &$-$ &$-$\\ \hline

\end{tabular}
\caption{\label{table-d3-runtimes}Runtime comparison for $(n,3)$}
\end{table}
The fastest variant is marked bold for instances where the base model exceeds $10s$ runtime and excludes the singular runs* on Intel Xeon Gold 6338. We note that warmstarts with suitable certificates produce no significant speedup, as the computationally expensive part of the solving process remains to close the dual. This, however, can be significantly improved by introducing the discussed clique inequalities and by proving infeasibility for solutions of greater size than the known certificates instead.\\
For $n=11,13$, the infeasibility of a larger solution $|Q_{max}|+1$ is trivially solved as it already violates the LP relaxation. Further, the existence and construction of a solution of size $n²$ is known due to theorem \ref{Klarner} theorem \ref{thm_vanrees_general_d}. Computational results are still provided for completeness. However, they will be ignored when comparing the variants due to the trivial result. 

\newpage
\subsection{The Partial $(n,4)$-Queens Problem}
\subsubsection{Base Model} % R
\begin{table}[H]
\noindent
\begin{tabular}{|rr|rrr|rrr|rrr|}
    \hline 
    $n$ & $|Q_{max}|$& avg. time & \hphantom{avg. }min & \hphantom{avg. }max & avg. work & \hphantom{avg. }min & \hphantom{avg. }max & avg. nodes & \hphantom{avg. }min & \hphantom{avg. }max\\ \hline%\hline
    1 &  1 & $0.00$ &$0.00$&$0.00$  & $0.00$ &$0.00$&$0.00$ &$0.0$ & $0$ & $0$\\ %\hline
    2 &  1 & $0.00$ &$0.00$&$0.00$  & $0.00$ &$0.00$&$0.00$ &$0.0$ & $0$ & $0$\\ %\hline
    3 &  6 & $0.02$ &$0.02$&$0.02$  & $0.02$ &$0.02$&$0.02$ &$1.0$ & $1$ & $1$\\ %\hline
    4 & 16 & $0.80$ &$0.56$&$0.99$  & $0.41$ &$0.33$&$0.45$ &$1.0$ & $1$ & $1$\\ %\hline
    5 & 38 & $-$ &$>36000.00$&$-$  & $-$ &$-$&$-$ & $-$& $-$ & $-$\\ %\hline
    6 & 80 & $569.57$&$137.25$&$1437.64$  & $169.51$ &$98.72$&$289.15$ & $430.5$& $31$ & $2020$\\ %\hline
    7 &145 & $-$&$>36000.00$&$-$  & $-$ &$-$&$-$ & $-$& $-$ & $-$\\ \hline
\end{tabular}
\caption{\label{table-d4-baseIP}Base IP for $(n,4)$}
\end{table}

\vspace{-1mm}
\subsubsection{LMR} 
\begin{table}[H]
\noindent
\begin{tabular}{|rr|rrr|rrr|rrr|}
    \hline 
    $n$ & $|Q_{max}|$& avg. time & \hphantom{avg. }min & \hphantom{avg. }max & avg. work & \hphantom{avg. }min & \hphantom{avg. }max & avg. nodes & \hphantom{avg. }min & \hphantom{avg. }max\\ \hline%\hline
    1 &  1 & $0.00$ &$0.00$&$0.00$  & $0.00$ &$0.00$&$0.00$ &  $0.0$ & $0$ & $0$\\ %\hline
    2 &  1 & $0.00$ &$0.00$&$0.00$  & $0.00$ &$0.00$&$0.00$ &  $0.0$ & $0$ & $0$\\ %\hline
    3 &  6 & $0.01$ &$0.01$&$0.01$  & $0.01$ &$0.01$&$0.01$ &  $1.0$ & $1$ & $1$\\ %\hline
    4 & 16 & $0.09$ &$0.07$&$0.15$  & $0.08$ &$0.08$&$0.10$ &  $1.0$ & $1$ & $1$\\ %\hline
    5 & 38 & $-$ &$>36000.00$&$-$  & $-$&$-$&$-$ &  $-$ & $-$ & $-$\\ %\hline
    6 & 80 &$760.08$&$291.53$&$2246.26$  & $202.53$ &$118.34$&$483.90$ &  $460.9$ & $55$ & $2570$\\ %\hline
    7 &145 &$-$&$>36000.00$&$-$  & $-$ &$-$&$-$ &  $-$ & $-$ & $-$\\ \hline
\end{tabular}
\caption{\label{table-d4-erika}\cite{Langlois2023} for $(n,4)$}
\end{table}

\newpage
\subsubsection{All additional Inequalities}
\begin{table}[H]
\noindent
\begin{tabular}{|rr|rrr|rrr|rrr|}
    \hline 
    $n$ & $|Q_{max}|$& avg. time & \hphantom{avg. }min & \hphantom{avg. }max & avg. work & \hphantom{avg. }min & \hphantom{avg. }max & avg. nodes & \hphantom{avg. }min & \hphantom{avg. }max\\ \hline%\hline
    1 &  1 & $0.00$ &$0.00$&$0.00$  & $0.00$ &$0.00$&$0.00$ &  $0.0$ & $0$ & $0$\\ %\hline
    2 &  1 & $0.00$ &$0.00$&$0.00$  & $0.00$ &$0.00$&$0.00$ &  $0.0$ & $0$ & $0$\\ %\hline
    3 &  6 & $0.02$ &$0.02$&$0.02$  & $0.03$ &$0.03$&$0.03$ &  $1.0$ & $1$ & $1$\\ %\hline
    4 & 16 & $0.21$ &$0.13$&$0.31$  & $0.19$ &$0.16$&$0.23$ &  $1.0$ & $1$ & $1$\\ %\hline
    5 & 38 & $-$ &$>36000.00$&$-$  & $-$&$-$&$-$ &$-$ & $-$ & $-$\\ %\hline
    6 & 80 &$725.91$&$438.90$&$1637.14$  & $276.35$ &$220.12$&$435.53$ &  $67.9$ & $15$ & $196$\\ %\hline
    7 & 145&$-$&$>36000.00$&$-$  & $-$ &$-$&$-$ &  $-$ & $-$ & $-$\\ \hline
\end{tabular}
\caption{\label{table-d4-all4}IP with cube- and star cliques, layer- and subsolution inequalities for $(n,4)$}
\end{table}

\subsubsection{Infeasibility for $|Q_{max}|+1$} 
\begin{table}[H]
\noindent
\begin{tabular}{|rr|rrr|rrr|rrr|}
    \hline 
    $n$ & $|Q_{max}|$& avg. time & \hphantom{avg. }min & \hphantom{avg. }max & avg. work & \hphantom{avg. }min & \hphantom{avg. }max & avg. nodes & \hphantom{avg. }min & \hphantom{avg. }max\\ \hline%\hline
    1 &  1 & $0.00$ &$0.00$&$0.00$  & $0.00$ &$0.00$&$0.00$ &  $0.0$ & $0$ & $0$\\ %\hline
    2 &  1 & $0.00$ &$0.00$&$0.00$  & $0.00$ &$0.00$&$0.00$ &  $0.0$ & $0$ & $0$\\ %\hline
    3 &  6 & $0.02$ &$0.02$&$0.02$  & $0.02$ &$0.02$&$0.03$ &  $1.0$ & $1$ & $1$\\ %\hline
    4 & 16 & $0.67$ &$0.55$&$0.80$  & $0.28$ &$0.23$&$0.34$ &  $1.0$ & $1$ & $1$\\ %\hline
    5 & 38 & $-$ &$>36000.00$&$-$  & $-$&$-$&$-$ &$-$ & $-$ & $-$\\ %\hline
    6 & 80 &$879.43$&$638.87$&$2072.33$  & $173.93$ &$152.10$&$303.54$ &  $4545.6$ & $3077$ & $6347$\\ %\hline
    7 & 145&$-$&$>36000.00$&$-$  & $-$ &$-$&$-$ &  $-$ & $-$ & $-$\\ \hline
\end{tabular}
\caption{\label{table-d4-all4}Proving infeasibility for $|Q_{max}|+1$, Base IP for $(n,4)$}
\end{table}

\subsubsection{Infeasibility for $|Q_{max}|+1$, all additional Inequalities} 
\begin{table}[H]
\noindent
\begin{tabular}{|rr|rrr|rrr|rrr|}
    \hline 
    $n$ & $|Q_{max}|$& avg. time & \hphantom{avg. }min & \hphantom{avg. }max & avg. work & \hphantom{avg. }min & \hphantom{avg. }max & avg. nodes & \hphantom{avg. }min & \hphantom{avg. }max\\ \hline%\hline
    1 &  1 & $0.00$ &$0.00$&$0.00$  & $0.00$ &$0.00$&$0.00$ &  $0.0$ & $0$ & $0$\\ %\hline
    2 &  1 & $0.00$ &$0.00$&$0.00$  & $0.00$ &$0.00$&$0.00$ &  $0.0$ & $0$ & $0$\\ %\hline
    3 &  6 & $0.02$ &$0.02$&$0.02$  & $0.02$ &$0.02$&$0.02$ &  $1.0$ & $1$ & $1$\\ %\hline
    4 & 16 & $0.01$ &$0.01$&$0.01$  & $0.01$ &$0.01$&$0.01$ &  $0.0$ & $0$ & $0$\\ %\hline
    5 & 38 & $-$ &$>36000.00$&$-$  & $-$&$-$&$-$ &$-$ & $-$ & $-$\\ %\hline
    6 & 80 &$10.68$&$ 8.58$&$17.49$  & $9.21$ &$ 7.75$&$15.30$ &  $1.0$ & $1$ & $1$\\ %\hline
    7 & 145&$-$&$>36000.00$&$-$  & $-$ &$-$&$-$ &  $-$ & $-$ & $-$\\ \hline
\end{tabular}
\caption{\label{table-d4-all4}Proving infeasibility for $|Q_{max}|+1$, IP with cube- and star cliques, layer- and subsolution inequalities for $(n,4)$}
\end{table}

\subsubsection{Warmstarts with Base Model} 
\begin{table}[H]
\noindent
\begin{tabular}{|rr|rrr|rrr|rrr|}
    \hline 
    $n$ & $|Q_{max}|$& avg. time & \hphantom{avg. }min & \hphantom{avg. }max & avg. work & \hphantom{avg. }min & \hphantom{avg. }max & avg. nodes & \hphantom{avg. }min & \hphantom{avg. }max\\ \hline%\hline
    1 &  1 & $0.00$ &$0.00$&$0.00$  & $0.00$ &$0.00$&$0.00$ &  $0.0$ & $0$ & $0$\\ %\hline
    2 &  1 & $0.00$ &$0.00$&$0.00$  & $0.00$ &$0.00$&$0.00$ &  $0.0$ & $0$ & $0$\\ %\hline
    3 &  6 & $0.02$ &$0.02$&$0.02$  & $0.02$ &$0.02$&$0.02$ &  $1.0$ & $1$ & $1$\\ %\hline
    4 & 16 & $0.57$ &$0.48$&$0.63$  & $0.31$ &$0.30$&$0.32$ &  $1.0$ & $1$ & $1$\\ %\hline
    5 & 38 & $-$ &$>36000.00$&$-$  & $-$&$-$&$-$ &$-$ & $-$ & $-$\\ %\hline
    6 & 80 &$64.98$&$62.51$&$67.19$  & $41.27$ &$39.96$&$ 42.66$ &  $1.0$ & $1$ & $1$\\ %\hline
    7 & 145&$-$&$>36000.00$&$-$  & $-$ &$-$&$-$ &  $-$ & $-$ & $-$\\ \hline
\end{tabular}
\caption{\label{table-d4-warmstarts-base}Warmstarts on base IP for $(n,4)$}
\end{table}

\subsubsection{Warmstarts with all additional Inequalities} 
\begin{table}[H]
\noindent
\begin{tabular}{|rr|rrr|rrr|rrr|}
    \hline 
    $n$ & $|Q_{max}|$& avg. time & \hphantom{avg. }min & \hphantom{avg. }max & avg. work & \hphantom{avg. }min & \hphantom{avg. }max & avg. nodes & \hphantom{avg. }min & \hphantom{avg. }max\\ \hline%\hline
    1 &  1 & $0.00$ &$0.00$&$0.00$  & $0.00$ &$0.00$&$0.00$ &  $0.0$ & $0$ & $0$\\ %\hline
    2 &  1 & $0.00$ &$0.00$&$0.00$  & $0.00$ &$0.00$&$0.00$ &  $0.0$ & $0$ & $0$\\ %\hline
    3 &  6 & $0.02$ &$0.02$&$0.02$  & $0.03$ &$0.03$&$0.03$ &  $1.0$ & $1$ & $1$\\ %\hline
    4 & 16 & $0.12$ &$0.10$&$0.12$  & $0.14$ &$0.13$&$0.15$ &  $1.0$ & $1$ & $1$\\ %\hline
    5 & 38 & $-$ &$>36000.00$&$-$  & $-$&$-$&$-$ &$-$ & $-$ & $-$\\ %\hline
    6 & 80 &$19.77$&$16.91$&$26.88$  & $23.99$ &$20.76$&$32.18$ &  $1.0$ & $1$ & $1$\\ %\hline
    7 & 145&$-$&$>36000.00$&$-$  & $-$ &$-$&$-$ &  $-$ & $-$ & $-$\\ \hline
\end{tabular}
\caption{\label{table-d4-warmstarts-all4}Warmstarts on IP with cube- and star cliques, layer- and subsolution inequalities for $(n,4)$}
\end{table}

\newpage
\subsubsection{Comparison of IP Formulations} 
\begin{table}[H]
\noindent
\begin{tabular}{|rr|rrrrrrr|}
    \hline 
    $n$ & $|Q_{max}|$& Base   & LMR     & Inq.   & Inf & Inf Inq.  & ws    & ws Inq. \\ \hline%\hline
    1 &  1           &$0.00$  &$0.00$   &$0.00$  &$0.00$  &$0.00$   &$0.00$  &$0.00$   \\ %\hline
    2 &  1           &$0.00$  &$0.00$   &$0.00$  &$0.00$  &$0.00$   &$0.00$  &$0.00$   \\ %\hline
    3 &  6           &$0.02$  &$0.01$   &$0.02$  &$0.02$  &$0.02$   &$0.02$  &$0.02$   \\ %\hline
    4 & 16           &$0.80$  &$0.09$   &$0.21$  &$0.67$  &$0.01$   &$0.57$  &$0.12$   \\ %\hline
    5 & 38           &$-$     &$-$      &$-$     & $-$    & $-$      &$-$      &$-$    \\ %\hline
    6 & 80           &$569.57$&$760.08$ &$725.91$&$879.43$&$\textbf{10.68}$&$64.98$  &$19.77$   \\ %\hline
    7 &145           &\hphantom{yyyyyyy}$-$     &\hphantom{yyyyyyy}$-$&\hphantom{yyyyyyy}$-$      &\hphantom{yyyyyyy}$-$      &\hphantom{yyyyyyy}$-$        &\hphantom{yyyyyyy}$-$      &\hphantom{yyyyyyy}$-$     \\ \hline
    
\end{tabular}
\caption{\label{table-d4-runtimes}Runtime comparison for $(n,4)$}
\end{table}

The fastest variant is marked bold for instances where the base model exceeds $10s$ runtime. For $d=4$, the smaller instances up to $n=4$ solve too fast and in the root node, leaving little room for comparing the proposed variants. In contrast to $d=3$, warmstarts with suitable certificates already produce a speedup by roughly one order of magnitude compared to the respective variant without a warmstart for $n=6$. The additional clique inequalities, layer- and subsolution inequalities do not provide a significant speedup for $n=6$ by themselves; however, they reduce the number of nodes by one order of magnitude. We expect the inequalities to result in a larger speedup for instances $n>6$, as a similar effect can be observed for $d=3$ (for smaller instances, the additional inequalities may slow down the solving process, which outweighs the speedup through the reduction of nodes). 

As is the case with $d=3$, the difficulty of the instances lies in the dual, and the mentioned additional inequalities can greatly improve the performance of closing the dual. Proving a solution of size $|Q_{max}|+1$ is infeasible (provided we know a certificate that we believe to be maximal) is the fastest method for $d=4, n=6$.

\newpage
\subsection{The Partial $(n,5)$-Queens Problem}
\subsubsection{Base Model} 
\begin{table}[H]
\noindent
\begin{tabular}{|rr|rrr|rrr|rrr|}
    \hline 
    $n$ & $|Q_{max}|$& avg. time & \hphantom{avg. }min & \hphantom{avg. }max & avg. work & \hphantom{avg. }min & \hphantom{avg. }max & avg. nodes & \hphantom{avg. }min & \hphantom{avg. }max\\ \hline%\hline
    1 &  1 & $0.00$ &$0.00$&$0.00$  & $0.00$ &$0.00$&$0.00$ &$0.0$ & $0$ & $0$\\ %\hline
    2 &  1 & $0.00$ &$0.00$&$0.00$  & $0.00$ &$0.00$&$0.00$ &$0.0$ & $0$ & $0$\\ %\hline
    3 & 11 & $0.78$ &$0.58$&$1.07$  & $0.42$ &$0.34$&$0.50$ &$1.0$ & $1$ & $1$\\ %\hline
    4 & 32 & $21.14$ &$18.14$&$23.82$  & $14.14$ &$13.02$&$15.86$ &$1.0$ & $1$ & $1$\\ %\hline
    5 & $\geq 90$ & $-$ &$>36000.00$&$-$  & $-$ &$-$&$-$ & $-$& $-$ & $-$\\ \hline
\end{tabular}
\caption{\label{table-d5-baseIP} Base IP for $(n,5)$}
\end{table}

\subsubsection{LMR} 
\begin{table}[H]
\noindent
\begin{tabular}{|rr|rrr|rrr|rrr|}
    \hline 
    $n$ & $|Q_{max}|$& avg. time & \hphantom{avg. }min & \hphantom{avg. }max & avg. work & \hphantom{avg. }min & \hphantom{avg. }max & avg. nodes & \hphantom{avg. }min & \hphantom{avg. }max\\ \hline%\hline
    1 &  1 & $0.00$ &$0.00$&$0.00$  & $0.00$ &$0.00$&$0.00$ &$0.0$ & $0$ & $0$\\ %\hline
    2 &  1 & $0.00$ &$0.00$&$0.00$  & $0.00$ &$0.00$&$0.00$ &$0.0$ & $0$ & $0$\\ %\hline
    3 & 11 & $0.20$ &$0.15$&$0.29$  & $0.12$ &$0.11$&$0.15$ &$1.0$ & $1$ & $1$\\ %\hline
    4 & 32 & $3.38$ &$2.66$&$5.30$  & $1.83$ &$1.57$&$2.61$ &$1.0$ & $1$ & $1$\\ %\hline
    5 & $\geq 90$ & $-$ &$>36000.00$&$-$  & $-$ &$-$&$-$ & $-$& $-$ & $-$\\ \hline
\end{tabular}
\caption{\label{table-d5-erika} \cite{Langlois2023} for $(n,5)$}
\end{table}

\subsubsection{All additional Inequalities}
\begin{table}[H]
\noindent
\begin{tabular}{|rr|rrr|rrr|rrr|}
    \hline 
    $n$ & $|Q_{max}|$& avg. time & \hphantom{avg. }min & \hphantom{avg. }max & avg. work & \hphantom{avg. }min & \hphantom{avg. }max & avg. nodes & \hphantom{avg. }min & \hphantom{avg. }max\\ \hline%\hline
    1 &  1 & $0.00$ &$0.00$&$0.00$  & $0.00$ &$0.00$&$0.00$ &$0.0$ & $0$ & $0$\\ %\hline
    2 &  1 & $0.00$ &$0.00$&$0.00$  & $0.00$ &$0.00$&$0.00$ &$0.0$ & $0$ & $0$\\ %\hline
    3 & 11 & $0.53$ &$0.40$&$0.66$  & $0.27$ &$0.23$&$0.31$ &$1.0$ & $1$ & $1$\\ %\hline
    4 & 32 & $2.40$ &$1.75$&$4.63$  & $14.14$ &$1.93$&$3.76$ &$1.0$ & $1$ & $1$\\ %\hline
    5 & $\geq 90$ & $-$ &$>36000.00$&$-$  & $-$ &$-$&$-$ & $-$& $-$ & $-$\\ \hline
\end{tabular}
\caption{\label{table-d4-all4} IP with cube- and star cliques, layer- and subsolution inequalities for $(n,5)$}
\end{table}

\subsubsection{Infeasibility for $|Q_{max}|+1$} 
\begin{table}[H]
\noindent
\begin{tabular}{|rr|rrr|rrr|rrr|}
    \hline 
    $n$ & $|Q_{max}|$& avg. time & \hphantom{avg. }min & \hphantom{avg. }max & avg. work & \hphantom{avg. }min & \hphantom{avg. }max & avg. nodes & \hphantom{avg. }min & \hphantom{avg. }max\\ \hline%\hline
    1 &  1 & $0.00$ &$0.00$&$0.00$  & $0.00$ &$0.00$&$0.00$ &$0.0$ & $0$ & $0$\\ %\hline
    2 &  1 & $0.00$ &$0.00$&$0.00$  & $0.00$ &$0.00$&$0.00$ &$0.0$ & $0$ & $0$\\ %\hline
    3 & 11 & $2.11$ &$0.87$&$4.36$  & $0.74$ &$0.39$&$1.17$ &$102.7$ & $1$ & $873$\\ %\hline
    4 & 32 & $44.47$ &$44.47$&$167.89$  &$18.21$ &$9.99$&$38.14$ &$604.0$ & $1$ & $5939$\\ %\hline
    5 & $\geq 90$ & $-$ &$>36000.00$&$-$  & $-$ &$-$&$-$ & $-$& $-$ & $-$\\ \hline
\end{tabular}
\caption{\label{table-d4-all4}Proving infeasibility for $|Q_{max}|+1$, base IP for $(n,5)$}
\end{table}

\subsubsection{Infeasibility for $|Q_{max}|+1$, all additional Inequalities} 
\begin{table}[H]
\noindent
\begin{tabular}{|rr|rrr|rrr|rrr|}
    \hline 
    $n$ & $|Q_{max}|$& avg. time & \hphantom{avg. }min & \hphantom{avg. }max & avg. work & \hphantom{avg. }min & \hphantom{avg. }max & avg. nodes & \hphantom{avg. }min & \hphantom{avg. }max\\ \hline%\hline
    1 &  1 & $0.00$ &$0.00$&$0.00$  & $0.00$ &$0.00$&$0.00$ &$0.0$ & $0$ & $0$\\ %\hline
    2 &  1 & $0.00$ &$0.00$&$0.00$  & $0.00$ &$0.00$&$0.00$ &$0.0$ & $0$ & $0$\\ %\hline
    3 & 11 & $0.40$ &$0.31$&$0.61$  & $0.26$ &$0.23$&$0.31$ &$1.0$ & $1$ & $1$\\ %\hline
    4 & 32 & $0.07$ &$0.07$&$0.07$  & $0.07$ &$0.07$&$0.07$ &$0.0$ & $0$ & $0$\\ %\hline
    5 & $\geq 90$ & $-$ &$-$&$-$  & $-$ &$-$&$-$ & $-$& $-$ & $-$\\ \hline
\end{tabular}
\caption{\label{table-d4-all4}Proving infeasibility for $|Q_{max}|+1$, IP with cube- and star cliques, layer- and subsolution inequalities for $(n,5)$}
\end{table}

\subsubsection{Comparison of IP Formulations} 
\begin{table}[H]
\noindent
\begin{tabular}{|rr|rrrrr|}
    \hline 
    $n$ & $|Q_{max}|$& Base   & LMR     & Inq.   & Inf & Inf Inq.  \\ \hline%\hline
    1 &  1           &$0.00$  &$0.00$   &$0.00$  &$0.00$  &$0.00$      \\ %\hline
    2 &  1           &$0.00$  &$0.00$   &$0.00$  &$0.00$  &$0.00$      \\ %\hline
    3 & 11           &$0.78$  &$0.20$   &$0.53$  &$2.11$  &$0.31$     \\ %\hline
    4 & 32           &$21.14$ &$3.38$   &$2.40$  &$44.47$  &$\textbf{0.07}$      \\ %\hline
    5 &$\geq 90$           &\hphantom{yyyyyyy}$-$     &\hphantom{yyyyyyy}$-$&\hphantom{yyyyyyy}$-$      &\hphantom{yyyyyyy}$-$      &\hphantom{yyyyyyy}$-$  \\ \hline
\end{tabular}
\caption{\label{table-d4-runtimes}Runtime comparison for $(n,5)$}
\end{table}

Again, the fastest variant is marked bold for $n=4$ where the base model exceeds $10s$ runtime. For $d=5$, this is the case for only $n=4$, as $n=5$ remains unsolved and $n=1,2,3$ solves under 1 second.

\end{landscape}
\newpage
\subsection{Density of Solutions}\label{subsec:comp_density}
Results for $n=1,2$ are excluded, as they follow trivially from proposition \ref{cor_2d}. The plots show the density for each 2d-layer (in order and row-wise), the color scale is chosen to visualize the empirical density reminiscent of a higher dimensional equivalent to the density described by \cite{Simkin2021}.
\subsubsection{$d=3$}

\begin{figure}[H]
\setkeys{Gin}{width=0.3\linewidth}
\includegraphics{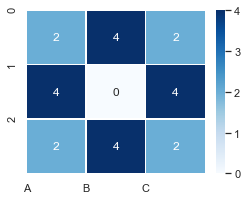}
\hfill
\includegraphics{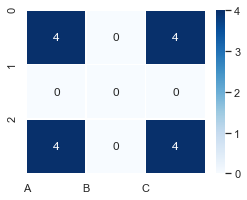}
\hfill
\includegraphics{density_3d/d3_n3_1.png}
\caption{Density for the $(3,3)$-queens problem}
\label{tab:density_3_d3}
\end{figure}

\begin{figure}[H]
\setkeys{Gin}{width=0.24\linewidth}
\includegraphics{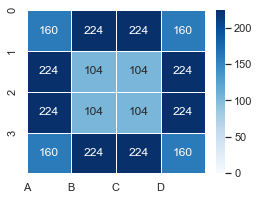}
\hfill
\includegraphics{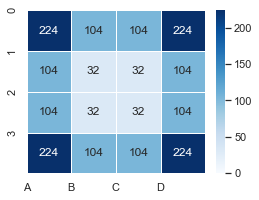}
\hfill
\includegraphics{density_3d/d3_n4_2.png}
\hfill
\includegraphics{density_3d/d3_n4_1.png}
\caption{Density for the $(4,3)$-queens problem}
\label{tab:density_n4_d3}
\end{figure}

\begin{figure}[H]
\setkeys{Gin}{width=0.29\linewidth}
\includegraphics{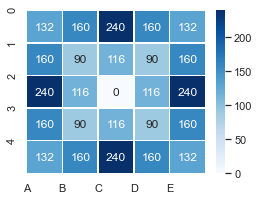}
\hfill
\includegraphics{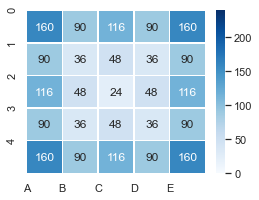}
\hfill
\includegraphics{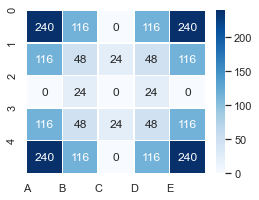}

\smallskip

\includegraphics{density_3d/d3_n5_2.png}
\hfill
\includegraphics{density_3d/d3_n5_1.png}
\hfill
\includegraphics{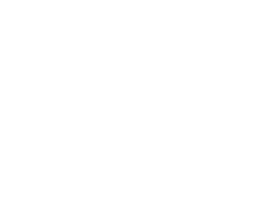}
\caption{Density for the $(5,3)$-queens problem}
\label{tab:density_n5_d3}
\end{figure}

\begin{figure}[H]
\setkeys{Gin}{width=0.29\linewidth}
\includegraphics{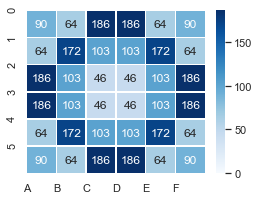}
\hfill
\includegraphics{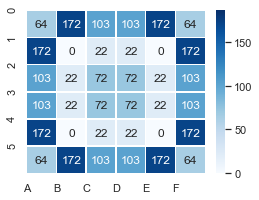}
\hfill
\includegraphics{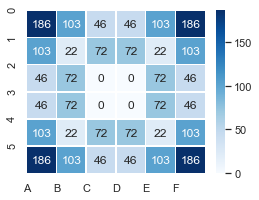}

\smallskip

\includegraphics{density_3d/d3_n6_3.png}
\hfill
\includegraphics{density_3d/d3_n6_2.png}
\hfill
\includegraphics{density_3d/d3_n6_1.png}
\caption{Density for the $(6,3)$-queens problem}
\label{tab:density_n6_d3}
\end{figure}

\begin{figure}[H]
\setkeys{Gin}{width=0.29\linewidth}
\includegraphics{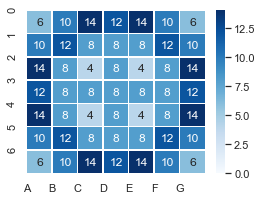}
\hfill
\includegraphics{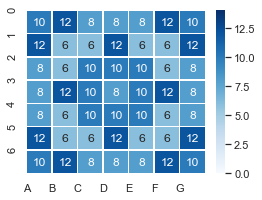}
\hfill
\includegraphics{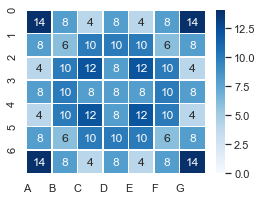}

\smallskip

\includegraphics{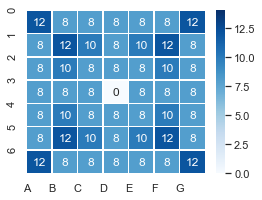}
\hfill
\includegraphics{density_3d/d3_n7_3.png}
\hfill
\includegraphics{density_3d/d3_n7_2.png}

\smallskip

\includegraphics{density_3d/d3_n7_1.png}
\hfill
\includegraphics{density_n11_d3/density_n11_d3_BLANK.png}
\hfill
\includegraphics{density_n11_d3/density_n11_d3_BLANK.png}
\caption{Density for the $(7,3)$-queens problem}
\label{tab:density_n7_d3}
\end{figure}

\newpage
Due to all solutions for $n=11$ and $n=13$ being regular, the density for those $n$ is simply the uniform distribution. However we may describe their structure further for a single class by fixing a queen on $(1,1,1)$.\\
\begin{figure}[H]
\setkeys{Gin}{width=0.3\linewidth}
\includegraphics{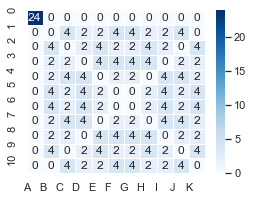}
\hfill
\includegraphics{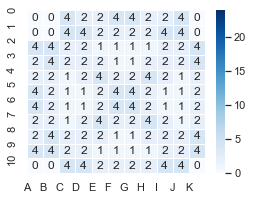}
\hfill
\includegraphics{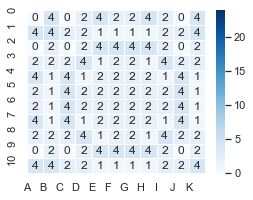}

\smallskip
\includegraphics{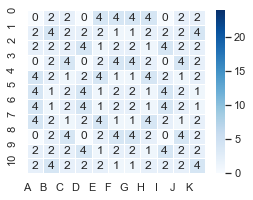}
\hfill
\includegraphics{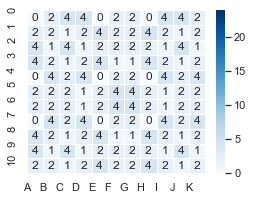}
\hfill
\includegraphics{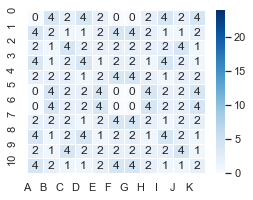}

\smallskip
\includegraphics{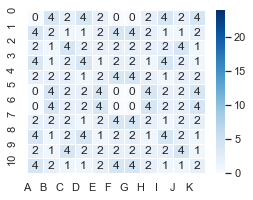}
\hfill
\includegraphics{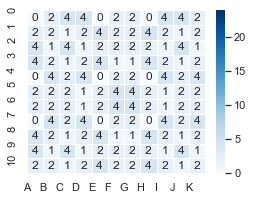}
\hfill
\includegraphics{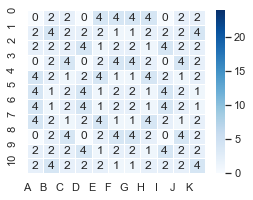}

\smallskip
\includegraphics{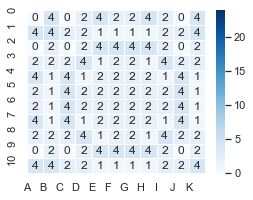}
\hfill
\includegraphics{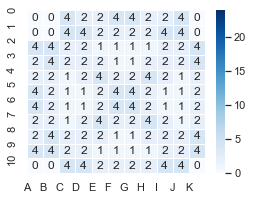}
\hfill
\includegraphics{density_n11_d3/density_n11_d3_BLANK.png}

\caption{Density of the $(0,0,0)$-class of the maximal partial solutions for the $(11,3)$-queens problem}
\label{tab:density_n11_d3}
\end{figure}

\newpage
\subsubsection{$d=4$}\label{subsec:comp_density_4d}

\begin{figure}[H]
\setkeys{Gin}{width=0.3\linewidth}
\includegraphics{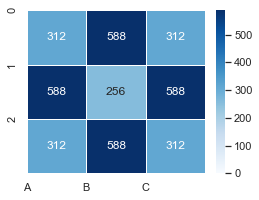}
\hfill
\includegraphics{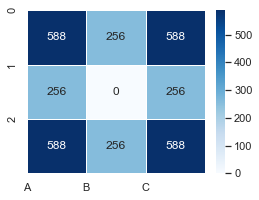}
\hfill
\includegraphics{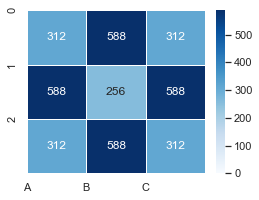}

\smallskip

\includegraphics{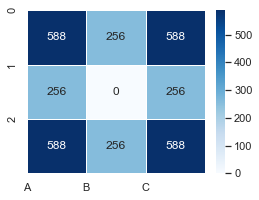}
\hfill
\includegraphics{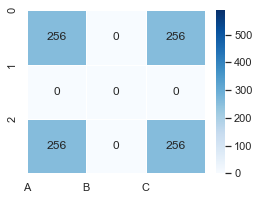}
\hfill
\includegraphics{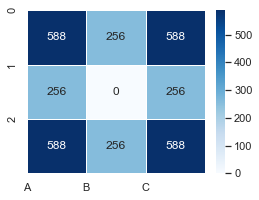}

\smallskip

\includegraphics{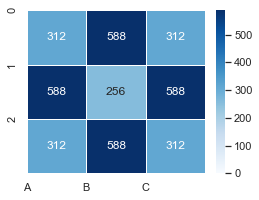}
\hfill
\includegraphics{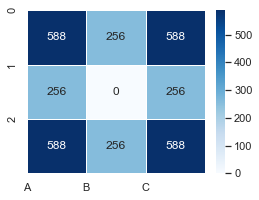}
\hfill
\includegraphics{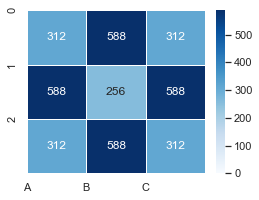}

\caption{Density for the $(3,4)$-queens problem}
\label{tab:density_n03_d4 }
\end{figure}

\begin{figure}[H]
\setkeys{Gin}{width=0.45\linewidth}
\includegraphics{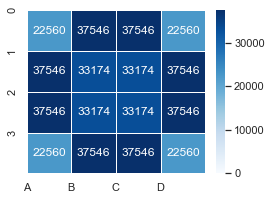}
\hfill
\includegraphics{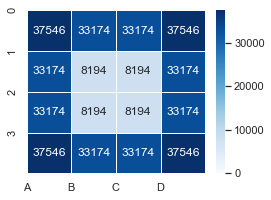}

\smallskip

\includegraphics{density_n04_d4/densityn_n04_d4_2.png}
\hfill
\includegraphics{density_n04_d4/density_n04_d4_1.png}

\bigskip
\bigskip
\bigskip
\includegraphics{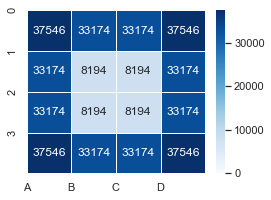}
\hfill
\includegraphics{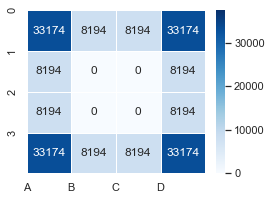}

\smallskip

\includegraphics{density_n04_d4/densityn_n04_d4_4.png}
\hfill
\includegraphics{density_n04_d4/densityn_n04_d4_3.png}
\caption{Density for the $(4,4)$-queens problem (limited to the first two layers due to symmetry)}
\label{tab:density_n04_d4 }
\end{figure}

\newpage
\section{Conclusions and Future Work}
We provided a first overview of the $(n,2)$-queens problem, connecting and improving existing results of related problems and methods for both theoretical results and the solving of large instances of the $(n,d)$-queens problem. We compare several different IP formulations and achieve a speedup of $15.5\times -71.2 \times $ less computational time compared to \cite{langlois2022complexity}, who recently succeeded in breaking new, never-before-solved instances. Our results suggest that breaking additional, previously unsolved instances with the proposed methods or improvements is likely possible. 

The problem of solving instances of the $(n,d)$-queens problem for $d \geq 3$ can be summarized as follows:
\begin{itemize}
    \item \textbf{Construction of Solutions} \\
    On the primal side, we discussed several heuristic approaches to constructing solutions that are either maximal partial solutions or very close to the optimum. Additionally, even if we discard this knowledge, the computational results showed that finding a solution of the (yet to be proven) optimal value takes an insignificant part of the total computational time.
    \item \textbf{Proving Maximality of Solutions} \\
    The time-intensive part of the solving process remains to close the dual. We greatly improved the computational time by adding additional valid inequalities, resulting in a significantly lower optimum to the LP relaxation of the problem. The corresponding theoretical results are described in section \ref{subsec_upperbounds}. These and the computational results suggest that the potential speedup scales with $d$.
\end{itemize}

We emphasize that this problem structure differs from the classic $(n,2)$-queens problem in complexity and, more importantly, as the former comes down to a feasibility problem and does not inherit the difficulty of closing the dual. Here, we showed that certain clique inequalities, in particular, the (hyper-)cube inequalities as discussed by \cite{fischetti2019finding} for the $(n,2)$-queens problem improve on the dual and thereby computational time.

Following this conclusion, a recipe for solving further instances seems to be a combination of  (a) improved theoretical results, also for subsets of the instances, (b) further study of clique inequalities and clique separation, and (c) making use of the unique structure of the problem during the solving process, for example by its corresponding symmetry group. One may also compare the IP with constraint programming and SAT solvers, specifically for those instances where we suspect already knowing the maximal partial solution.

Additionally, the generalization to the $(n,d)$-queens problem opens up the entire branch of variations to the classical $(n,2)$-queens problem to their higher dimensional equivalents. Through discussing the $(n,d)$-queens problem, we have also gained insight into $(n,d)$-queens completion and shown some preliminary results. Similarly, one may be interested in the many possible variations and related problems as discussed in the introductory section: minimum dominating sets, different board structures (or underlying graphs), and varying pieces (or 'vision'). These topics may be discussed in future work.

Finally, we have listed a number of open questions regarding the $(n,d)$-queens problem. We believe that a density function of the $(n,d)$-queens problem for $d \geq 3$ exists. If one imagines a hypersphere intersecting with a hypercube in the corresponding dimension $d$, then the density is the highest in the $(n,d)$-cube around the boundary of the sphere and lower in the middle and at all corners of the cube. We can observe this for $d=2$, see Fig. \ref{fig:dens_simkin} and empirically from the densities shown in section \ref{subsec:comp_density}. We may follow the approach \cite{Simkin2021} to show that such a density exists; however, it is not clear if the distinction between $n$ for which regular solutions exist and those for which they do not allow for this approach. As we have noted, proving that such a density exists would imply the existence of a non-regular solution of full size $n^{d-1}$.\\

In the following two tables, we summarize the results regarding the current state of solved instances for the $(n,d)$-queens problem (c.f. \cite{langlois2022complexity}).

\vspace{0.5cm}

\begin{table}[H]
\begin{center}
\noindent
\begin{tabular}{r|rrrrrrrrrrrrr}
    %\hline 
    d \textbackslash n &1&2&3&4&5&6&7&8&9&10&11&12&13  \\ \hline%\hline
    1 &  1&1&1&1&1&1&1&1&1&1&1&1&1  \\
    2 &  1&1&2&4&5&6&7&8&9&10&11&12&13  \\
    3 &  1&1&4&7&13&21&32&48&67&91&121& &169  \\
    4 &  1&1&6&16&38&80&145&&&&&&  \\
    5 &  1&1&11&32&&&&&&&&&  \\
    6 &  1&1&19&64&&&&&&&&&  \\
    7 &  1&1&32&128&&&&&&&&&  \\  
    8 &  1&1&52 &&&&&&&&&& \\  
\end{tabular}
\hspace{0.5cm}
\caption{\label{table-results-1}Known maximal partial solutions to the $(n,d)$-queens problem}
\end{center}
\end{table}

\newpage
\section{Bibliography}
\bibliography{main.bib}
\end{document}